\documentclass[10pt]{article}
\usepackage{amsthm,amsmath}

\usepackage[charter]{mathdesign}
\usepackage{euscript}
\usepackage{stmaryrd}
\usepackage{bookmark}

\usepackage{cleveref}
\usepackage{tikz-cd}
\usepackage{color}
\usepackage{graphics}
\usepackage{graphicx}
\usepackage{braket} 
\usetikzlibrary{arrows}
\usepackage[all,2cell]{xy} \UseAllTwocells \SilentMatrices
\usepackage{enumerate}
\usepackage{comment}
\usepackage{rotating}
\usepackage{tabularx}
\usepackage{multirow}
\usepackage{array}
\usepackage{authblk}

\newcommand{\m}[1]{\mathcal{#1}}
\newtheoremstyle{theoremdd}
  {\topsep}
  {\topsep}
  {\itshape}
  {0pt}
  {\bfseries}
  {. ---}
  { }
  {\thmname{#1}\thmnumber{ #2}\textnormal{\thmnote{ (#3)}}}

\theoremstyle{theoremdd}

\theoremstyle{definition}

\theoremstyle{remark}

\numberwithin{equation}{section}
\numberwithin{thm}{section}
\numberwithin{dfn}{section}
\numberwithin{rmk}{section}

\newcommand\be             {\begin{equation}}
\newcommand\ee             {\end{equation}}
\newcommand\bea            {\begin{eqnarray}}
\newcommand\eea            {\end{eqnarray}}
\newcommand\bnu            {\begin{enumerate}}
\newcommand\enu            {\end{enumerate}}
\newcommand{\bmm}          {\begin{matrix}}
\newcommand{\emm}          {\end{matrix}}
\newcommand{\pf}           {\begin{proof}}
\newcommand{\epf}          {\end{proof}}

\newcommand\id                {\mathrm{id}}

\newcommand\fpdim             {\mathrm{FPdim}}

\DeclareMathOperator{\coker}{coker}  
\DeclareMathOperator{\image}{Im} 
 
\newcommand\bZ            {\mathbb{Z}}

\theoremstyle{theorem}
\newtheorem{theorem}{Theorem}

\newtheorem{proposition}[theorem]{Proposition}
\newtheorem{corollary}[theorem]{Corollary}

\newtheorem{lemma}[theorem]{Lemma}

\newtheorem{phythm}[theorem]{Theorem$^{\mathrm{phy}}$}

\theoremstyle{definition}
\newtheorem{definition}[theorem]{Definition}
\newtheorem{remark}[theorem]{Remark}
\newtheorem{example}[theorem]{Example}

\numberwithin{equation}{section}
\numberwithin{theorem}{section}
\usepackage{hyperref}
\hypersetup{hypertex=true,
colorlinks=true,
linkcolor=black,
anchorcolor=red,
citecolor=black}

\makeindex

\title{Fiber product of condensable algebras and critical points of boundary phase transitions}

\begin{document}

\author[1]{Hanlin Lin\thanks{Email: \texttt{linhl24@mails.tsinghua.edu.cn}}}
\affil{Qiuzhen College, Tsinghua University, Beijing 100084, China}
\author[2,3]{Hao Zheng\thanks{Email: \texttt{haozheng@mail.tsinghua.edu.cn}}}
\affil[2]{Institute for Applied Mathematics, Tsinghua University, Beijing 100084, China}
\affil[3]{Beijing Institute of Mathematical Sciences and Applications, Beijing 101408, China}
\date{}
\maketitle
\begin{abstract}
    In this work, we propose a mathematical description of the critical point of boundary phase transitions in 2+1D topological orders. Given two boundary phases, $\mathcal{C}_A$ and $\mathcal{C}_B$, with corresponding Lagrangian algebras $A$ and $B$ in a 2+1D topological order $\mathcal{C}$, respectively, we propose that the critical point of the phase transition between these two phases, if it exists, corresponds to $D = A \times_{M} B$, which is the fiber product of $A$ and $B$ over $A$-$B$-algebra $M$. We prove that $D$ is also a condensable algebra. Additionally, we present an alternative classification of condensable algebras in $Z(Vec_G^\omega)$ and reveal splitting phenomenon that arises during the condensation of certain algebras. We also perform explicit computations for the condensable algebras within several examples of $Z(Vec_G^\omega)$, covering both Abelian and non-Abelian groups $G$. 

    \end{abstract}

\tableofcontents

\section{Introduction}

Boundaries have played an important role in the study of 2+1D topological orders since the early investigations of edge excitations in fractional quantum Hall systems. It has long been understood that a fixed bulk topological order may admit distinct gapped boundaries. This observation led to the categorical description of gapped boundaries and domain walls in terms of module categories \cite{kitaev2012models}, or, equivalently, within the anyon condensation framework, in terms of condensable and Lagrangian algebras \cite{kong2014anyon}. The categorical framework was subsequently extended to gapless boundaries, which can be described by enriched fusion categories \cite{kong2018gapless,kong2020mathematical,kong2021mathematical}.

These developments naturally lead to the study of boundary phase transitions. In particular, two distinct gapped boundaries of the same bulk topological order may be connected through a critical, and hence gapless, boundary phase without any change in the bulk phase. Such transitions are commonly referred to as \textit{pure edge phase transitions}. They have recently been investigated in lattice models, including the Ising topological order \cite{chen2020topological} and $\mathbb{Z}_N$ topological orders \cite{lu2023boundary}. Moreover, recent developments from the perspectives of categorical symmetry and condensable algebras \cite{Chatterjee_2023} suggest that the relations among condensable algebras associated with different gapped boundaries may encode information about the possible critical points separating them.

In this work, we investigate this relation from an algebraic perspective. Let $\mathcal{C}$ be a unitary modular tensor category (UMTC) describing a 2+1D topological order, and let $A,B\in\mathcal{C}$ be Lagrangian algebras corresponding to two gapped boundary phases $\mathcal{C}_A$ and $\mathcal{C}_B$, respectively. Suppose that these two boundary phases are connected by a pure edge phase transition. It was proposed in \cite{Chatterjee_2023} that a possible critical point should be associated with a common subalgebra of $A$ and $B$. This description, however, does not by itself specify a unique algebraic object. In Section~\ref{section2}, we propose a more precise formulation. We argue that, if a second-order pure edge phase transition between $\mathcal{C}_A$ and $\mathcal{C}_B$ exists, its critical point is described by the fiber product
$$
D=A\times_M B,
$$
where $M$ is an $A$-$B$-algebra. Physically, $M$ corresponds to a closed 0+1D defect, or equivalently a closed 0+1D domain wall, between the boundary phases $\mathcal{C}_A$ and $\mathcal{C}_B$.

To ensure that $\mathcal{C}_D$ can describe a gapless edge, $D$ must also be condensable. We further find that every common condensable subalgebra of $A$ and $B$ gives rise, at least algebraically, to a potential critical point for a pure edge phase transition. More precisely, given a common subalgebra $D$ of $A$ and $B$, one can associate to it the $A$-$B$-algebra $M=A\otimes_DB$ such that $M$ corresponds to the closed 0+1D defect between $\mathcal{C}_A$ and $\mathcal{C}_B$. Mathematically, the statement that $M$ corresponds to a closed 0+1D defect can be formulated as the condition that
$
\operatorname{Fun}_{\m{C}_A|\m{C}_B}(\m{C}_M,\m{C}_M)
\simeq \m{C}
$
as closed fusion $\m{C}$-$\m{C}$ bimodules \cite{kong2015boundary,kong2018center}. The following theorem, which is one of the main mathematical results of this work, provides mathematical support for this physical picture.

\begin{theorem}\label{main result}

    Let $\m{C}$ be a non-degenerated braided fusion category, and let $A,B \in \m{C}$ be two 2d condensable algebras. Let $M \in \m{C}$ be an algebra. Suppose there exist two algebra homomorphisms $f_A:A \to M$ and $f_B:B \to M$. Then the fiber product $D=A\times_{M}B$ is also a 2d condensable algebra.
    
    Conversely, let $D$ be a common subalgebra of $A$ and $B$, and set $M=A\otimes_D B$, then $D=A\times_M B$.
 Moreover, if $A$ and $B$ is lagrangian and $M$ is indecomposalbe, then $Fun_{\m{C}_A|\m{C}_B}(\m{C}_M, \m{C}_M)\simeq \m{C}$ as closed fusion $\m{C}$-$\m{C}$ bimodules.
\end{theorem}

The above results show that the embedding relations among condensable algebras contain information about possible boundary phase transitions. To make this structure more explicit and amenable to concrete calculations, we provide an alternative classification of condensable algebras in $Z(Vec_G^\omega)$. This classification is equivalent to \cite[Theorem 2.19]{davydov2017lagrangian}. We then study the embeddings between condensable algebras in $Z(Vec_G^\omega)$. In particular, we determine when such embeddings exist and classify them.

    \begin{theorem}\label{classalgebra}

        Let $G$ be a finite group, $\omega \in H^3(G, U(1))$. Then the isomorphic classes of condensable algebras in $Z(Vec_G^\omega)$ are classified by the equivalent class of quadruples $(H,F,\alpha,\phi  )$ where $F\trianglelefteq H \leqslant G $, $\alpha \in H^3(H/F, U(1))$ and $\phi: Vec_H^{\omega|_H}\rightarrow Vec_{H/F}^\alpha$ is a monoidal lifting of the forgetful functor $Vec_H \rightarrow Vec_{H/F}$.

    Two quadruples $(H,F,\alpha,\phi)$ and $(H',F',\alpha',\phi')$ are equivalent if and only if there exists a group element $g\in G$ such that $H' = {}^gH$, $F' = {}^gF$, and exists a monoidal lifting $\theta: Vec_{H/F}^\alpha \simeq Vec_{H'/F'}^{\alpha'}$ of the induced equivalence $Vec_{H/F} \simeq Vec_{H'/F'}$ such that the following diagram commutes up to monoidal natural isomorphism:
    \begin{equation}\label{quadruples_equdig}
        \begin{tikzcd}
            Vec_H^{\omega|_H} \arrow[d, "F_{l,\zeta_g|_H}"] \arrow[r, "\phi"] & Vec_{H/F}^\alpha \arrow[d, "\theta"] \\
            Vec_{H'}^{\omega|_{H'}}    \arrow[r, "{\phi'}"]           & Vec_{H'/F'}^{\alpha'} .              
            \end{tikzcd}
    \end{equation}
    Here, $\zeta_g|_H$ is the canonical monoidal structure defined in Appendix~\ref{appendix:A}.
    \end{theorem}

\begin{theorem}\label{embedding}

    Assume $A,B\in Z(Vec_G^\omega)$ are condensable algebras represented by quadruples $(H,F,\alpha,\phi)$ and $(K,E,\beta,\psi)$, respectively. An algebra embedding $A\to B$ is encoded by a pair $(g,\eta)$ where $g\in G$ is such that $F \trianglelefteq {}^g E \trianglelefteq {}^gK \leqslant H$, and $\eta$ is a monoidal structure on the induced functor $Vec_{^gK/F}^{\alpha|_{^gK/F}}\to Vec_{^gK/{}^gE}^{\beta^g}$ that makes the following diagram commutative up to a monoidal natural isomorphism
    \begin{equation}\label{diagram_embdig}
        \begin{tikzcd}
            Vec_{{}^g K}^{\omega|_{{}^g K}} \arrow[d, "\simeq"'] \arrow[r, "\phi|_{{}^gK}"] & Vec_{{}^gK/F}^{\alpha|_{{}^gK/F}} \arrow[d, "\eta"] \\
            Vec_{{}^gK}^{\omega|_{{}^gK}} \arrow[r, "L_g(\psi)"]                         & Vec_{{}^gK/{}^gE}^{\beta^g}  .                      
            \end{tikzcd}
    \end{equation} 

    The algebra embeddings encoded by $(g, \eta)$ and $(g', \eta')$ are equivalent if and only if the two quadruples $({}^gK/F, {}^gE/F, \beta^g, \eta)$ and $({}^{g'}K/F, {}^{g'}E/F, \beta^{g'}, \eta')$ associated to $Z(Vec_{H/F}^{\alpha})$ are equivalent.

    \end{theorem}

        This paper is organized as follows: in Section \ref{section2}, we explain our motivations from the
        physics of pure edge phase transition of 2+1D topological orders; In Section \ref{section3.1}, we recall some basic definitions and results about condensable algebras; In Section \ref{cen_and_cos}, we recall some results of internal hom and center; In Section \ref{section3.3}, we relate the boundary phase transitions to the cospan of centers; In Section \ref{section4}, we prove Theorem~\ref{main result}; In Section \ref{section5.class}, we prove Theorem~\ref{classalgebra}; In Section \ref{section5.embedding}, we prove Theorem~\ref{embedding}; In Section \ref{section6}, we present diagrams illustrating the classification and embedding relationships of condensable algebras in $Z(Vec_G^\omega)$ for some Abelian and non-Abelian groups $G$.

\section{Physics of boundary phase transition of 2+1D topological order}\label{section2}

It is well known that the anyon content of a 2+1D topological order is described by a unitary modular tensor category (UMTC). The particle-like excitations in topological order correspond to the objects of the UMTC. The fusion rules are encoded in the tensor product structure, while the mutual statistics are captured by the braiding structure of the UMTC. The mathematical framework for anyon condensation was systematized in \cite{kong2014anyon}. Given a 2+1D topological order $\mathcal{C}$ and a condensable algebra $A \in \mathcal{C}$, one can obtain a new 2+1D topological order $\mathcal{D}$ via 2d condensation. The vacuum of the $\mathcal{D}$-phase is identified with the condensable algebra $A$, and  $\mathcal{D}$ consists of all deconfined excitations, so that $\mathcal{D} = \mathcal{C}_A^{loc}$. The 1+1D gapped domain wall between the $\mathcal{C}$-phase and the $\mathcal{D}$-phase contains both confined and deconfined excitations and is described by the unitary fusion category (UFC) $\mathcal{C}_A$. If $\mathcal{C}_A^{loc} \simeq Vec$, then $A$ is called a {\it Lagrangian algebra}. In this case, $\mathcal{C}_A$ represents a gapped boundary of the phase $\mathcal{C}$.

A pure edge phase transition refers to a phase transition occurring on the boundary of a 2+1D topological order without altering the bulk \cite{chen2020topological,lu2023boundary}. Consider a 2+1D topological order $\mathcal{C}$ with two distinct gapped boundaries, $\mathcal{C}_A$ and $\mathcal{C}_B$. Using the language of categorical symmetry \cite{ji2020categorical,kong2020algebraic,Chatterjee_2023}, a pure edge phase transition between $\mathcal{C}_A$ and $\mathcal{C}_B$ can be viewed as a phase transition between the two 1+1D phases $\mathcal{C}_A$ and $\mathcal{C}_B$, both equipped with a categorical symmetry $\mathcal{C}$. The critical point of this phase transition, if it exists, corresponds to a gapless boundary of $\mathcal{C}$. In \cite{Chatterjee_2023}, it was proposed that this gapless phase is described by $\mathcal{C}_D$ endowed with a categorical symmetry $\mathcal{C}_D^{loc}$, where $D$ is a common subalgebra of $A$ and $B$. From the perspective developed in \cite{kong2018gapless,kong2020mathematical,kong2021mathematical}, such a gapless phase can be mathematically described by an enriched fusion category. The enriched fusion category has the same objects as $\mathcal{C}_D$, corresponding to the topological edge excitations (or boundary conditions) on the gapless boundary.

However, \cite{Chatterjee_2023} does not determine which common subalgebra of $A$ and $B$ corresponds to the critical algebra. Even if we consider the maximal common subalgebra, ambiguities remain. For example, there may exist two distinct subalgebras $D_1$ and $D_2$ of $A$ and $B$ such that both are common subalgebras and are both maximal. We present an example in Example \ref{common-subalgebras}. Moreover, a condensable algebra $D$ may admit more than one inequivalent embedding into $A$, giving rise to distinct $D$-module structures on $A$. These embeddings can therefore give rise to distinct algebras $A_1, A_2, \dots$ in $\mathcal{C}_D^{loc}$ with the same underlying algebra $A$ in $\mathcal{C}$. Consider a boundary phase transition between $(\mathcal{C}_D^{loc})_{A_1}$ and $(\mathcal{C}_D^{loc})_{A_2}$. According to the proposal in \cite{Chatterjee_2023}, the condensable algebra of this transition, if a critical point exists, is a common subalgebra of $A_1$ and $A_2$ in $\mathcal{C}_D^{loc}$. Combining the two stages of condensation, this algebra corresponds to a condensable algebra in $\mathcal{C}$ containing $D$. However, both $A_1$ and $A_2$ correspond to the same algebra $A$ in $\mathcal{C}$. Thus, selecting a maximal common subalgebra solely from the underlying algebras in $\mathcal{C}$ would yield $A$ itself, failing to find the critical point. These observations suggest that additional data is required to fully determine the critical algebra corresponding to the phase transition.

We propose that the data of the 0+1D defect between $\mathcal{C}_A$ and $\mathcal{C}_B$ is important. This defect would influence the excitations at the critical point and hence affect the corresponding category of excitations. Let $\mathcal{M}$ denote the category of 0+1D defects between $\mathcal{C}_A$ and $\mathcal{C}_B$ and $\m{M}^{\vee}$ denote the category of 0+1D defects between $\m{C}_B$ and $\m{C}_A$. In physics, when two boundaries $\mathcal{C}_A$ and $\mathcal{C}_B$ undergo a phase transition, the boundary of $\mathcal{C}$ can be viewed as a mixture of the $\mathcal{C}_A$ and $\mathcal{C}_B$ phases. In this context, certain 0+1D defects are required to connect these two phases, as shown in Figure \ref{Fig:mixture}. Here, $F \in \mathcal{M}$ and $F^{\vee} \in \mathcal{M}^{\vee}$ denote objects representing dual defects in the physical sense. (The explicit definition of $F^{\vee}$ is given in Section \ref{section3.3}.). For simplicity, we focus on the case involving only $F$ and $F^{\vee}$, without other 0+1D defects.

\begin{figure}
    \centering
\tikzset{every picture/.style={line width=0.75pt}} 

\begin{tikzpicture}[x=0.75pt,y=0.75pt,yscale=-0.8,xscale=0.8]

\draw  [color={rgb, 255:red, 255; green, 255; blue, 255 }  ,draw opacity=1 ][fill={rgb, 255:red, 80; green, 227; blue, 194 }  ,fill opacity=0.5 ] (40,12.67) -- (580,12.67) -- (580,165.67) -- (40,165.67) -- cycle ;
\draw [color={rgb, 255:red, 74; green, 144; blue, 226 }  ,draw opacity=1 ][fill={rgb, 255:red, 74; green, 144; blue, 226 }  ,fill opacity=1 ][line width=3]    (40,165.67) -- (180,165.67) ;
\draw [color={rgb, 255:red, 126; green, 211; blue, 33 }  ,draw opacity=1 ][line width=3]    (180,165.67) -- (310,165.67) ;
\draw  [color={rgb, 255:red, 245; green, 166; blue, 35 }  ,draw opacity=1 ][fill={rgb, 255:red, 245; green, 166; blue, 35 }  ,fill opacity=1 ] (174.33,165.67) .. controls (174.33,162.54) and (176.87,160) .. (180,160) .. controls (183.13,160) and (185.67,162.54) .. (185.67,165.67) .. controls (185.67,168.8) and (183.13,171.33) .. (180,171.33) .. controls (176.87,171.33) and (174.33,168.8) .. (174.33,165.67) -- cycle ;
\draw [color={rgb, 255:red, 74; green, 144; blue, 226 }  ,draw opacity=1 ][fill={rgb, 255:red, 74; green, 144; blue, 226 }  ,fill opacity=1 ][line width=3]    (310,165.67) -- (450,165.67) ;
\draw [color={rgb, 255:red, 126; green, 211; blue, 33 }  ,draw opacity=1 ][line width=3]    (450,165.67) -- (580,165.67) ;
\draw  [color={rgb, 255:red, 245; green, 166; blue, 35 }  ,draw opacity=1 ][fill={rgb, 255:red, 245; green, 166; blue, 35 }  ,fill opacity=1 ] (444.33,165.67) .. controls (444.33,162.54) and (446.87,160) .. (450,160) .. controls (453.13,160) and (455.67,162.54) .. (455.67,165.67) .. controls (455.67,168.8) and (453.13,171.33) .. (450,171.33) .. controls (446.87,171.33) and (444.33,168.8) .. (444.33,165.67) -- cycle ;
\draw  [color={rgb, 255:red, 245; green, 166; blue, 35 }  ,draw opacity=1 ][fill={rgb, 255:red, 245; green, 166; blue, 35 }  ,fill opacity=1 ] (304.33,165.67) .. controls (304.33,162.54) and (306.87,160) .. (310,160) .. controls (313.13,160) and (315.67,162.54) .. (315.67,165.67) .. controls (315.67,168.8) and (313.13,171.33) .. (310,171.33) .. controls (306.87,171.33) and (304.33,168.8) .. (304.33,165.67) -- cycle ;

\draw (314,95) node [anchor=north west][inner sep=0.75pt]   [align=left] {$\displaystyle \mathcal{C}$};
\draw (155,182) node [anchor=north west][inner sep=0.75pt]   [align=left] {$\displaystyle (\mathcal{M} ,F)$};
\draw (234,182) node [anchor=north west][inner sep=0.75pt]   [align=left] {$\displaystyle \mathcal{C}_{B}$};
\draw (95,182) node [anchor=north west][inner sep=0.75pt]   [align=left] {$\displaystyle \mathcal{C}_{A}$};
\draw (281,182) node [anchor=north west][inner sep=0.75pt]   [align=left] {$\displaystyle \left(\mathcal{M}^{\vee} ,F^{\lor }\right)$};
\draw (569,182) node [anchor=north west][inner sep=0.75pt]   [align=left] {$\displaystyle ......$};
\draw (374,182) node [anchor=north west][inner sep=0.75pt]   [align=left] {$\displaystyle \mathcal{C}_{A}$};
\draw (516,182) node [anchor=north west][inner sep=0.75pt]   [align=left] {$\displaystyle \mathcal{C}_{B}$};
\draw (428,182) node [anchor=north west][inner sep=0.75pt]   [align=left] {$\displaystyle (\mathcal{M} ,F)$};

\end{tikzpicture}
\caption{The mixture of $\m{C}_A$ and $\m{C}_B$}
\label{Fig:mixture}
\end{figure}
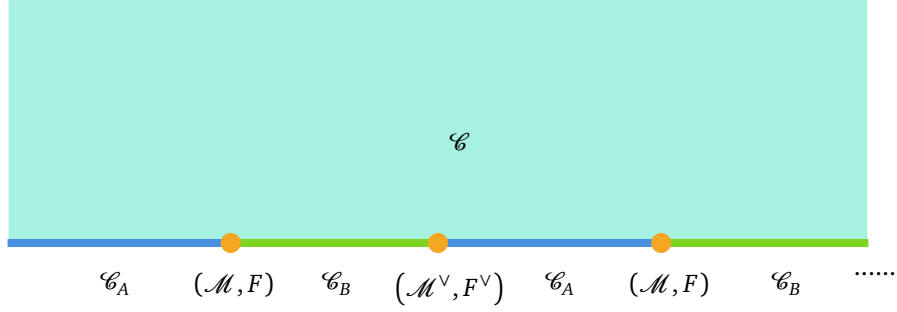

The $0+1$D defects $(\mathcal{M}, F)$ and $(\mathcal{M}^{\vee}, F^{\vee})$ determine an algebra $Z(F) = Z(F^{\vee})$ in $\mathcal{C}$, together with two natural algebra morphisms $f_A: A \to Z(F)$ and $f_B: B \to Z(F)$. (See Section~\ref{section3} for the definitions of the full centralizer $Z(F)$ and the morphisms $f_A$ and $f_B$.) Let $D$ be an appropriate common subalgebra of $A$ and $B$ describing the critical point. Composing the embeddings $\pi_A:D\to A$ and $\pi_B:D\to B$ with $f_A$ and $f_B$, respectively, yields two algebra morphisms from $D$ to $Z(F)$. We expect these two morphisms to coincide, since $Z(F)$ represents an excitation at the critical point and should therefore carry a unique canonical $D$-module structure.

We now impose the physical principle that, at the critical point, condensation should be maximal subject to the prescribed boundary and defect data. Since no further constraints are imposed, everything that can consistently condense should condense. Thus, among all compatible triples $(D',\pi'_A,\pi'_B)$, the critical algebra $D$ should realize the largest possible condensation compatible with the two maps to $Z(F)$. More precisely, for every compatible triple $(D',\pi'_A,\pi'_B)$, there should exist a unique algebra morphism $\pi':D'\to D$ such that $\pi_A\circ \pi'=\pi'_A$ and $\pi_B\circ \pi'=\pi'_B$. Hence $D$ is the terminal object in the category of compatible triples. This is precisely the universal property of the fiber product $D=A\times_{Z(F)}B$.

In summary, we have proposed the following physical theorem on the critical point of pure boundary phase transition of 2+1D topological order:

\begin{phythm}
    
    Let $\mathcal{C}$ be a 2+1D topological order, and let $A$ and $B$ be Lagrangian algebras in $\mathcal{C}$. Suppose $(\m{M},F)$ and $(\m{M}^{\vee},F^{\vee})$ are chosen 0+1D defects between the boundary phases $\mathcal{C}_A$ and $\mathcal{C}_B$. The critical point of the boundary phase transition associated with this boundary data, if it exists, is expected to be described by the fiber product $D = A \times_{Z(F)} B$ of the following diagram

    \begin{center}
        \begin{tikzcd}
            D \arrow[d] \arrow[r] & B \arrow[d, "f_B"] \\
            A \arrow[r, "f_A"]    & Z(F).            
            \end{tikzcd}
    \end{center}
\end{phythm}

In this work, we prove that $D$ is also a condensable algebra in $\m{C}$. Therefore, the gapless phase $\m{C}_D$, characterized by the categorical symmetry $\m{C}_D^{loc}$, is well-defined. Moreover, we prove that for any common condensable subalgebra $D$ of $A$ and $B$, if one defines $M = A \otimes_{D} B$, then the fiber product $A \times_{M} B$ is precisely $D$. This implies that any common condensable subalgebra $D$ could potentially serve as the critical point.

\begin{remark}
    The 0+1D defects (or domain walls) discussed here are all so-called {\it closed} or {\it anomaly-free} gapped defects, as they connect two 1+1D phases sharing the same bulk. (See \cite[Section 3.3]{kong2015boundary} for a detailed definition.)
\end{remark}
\begin{remark}
If the condensable algebras $A$ and $B$ are not Lagrangian, their fiber product $D$ is also a condensable algebra. Thus, the fiber product may also describe the critical point of boundary phase transition of gapless boundaries.
\end{remark}
\begin{remark}\label{rmk:multi F}
    Suppose we are given a collection of defects $F_i \in \mathcal{M}$ and $F^{\vee}_i \in \mathcal{M}^{\vee}$ for $i=1, 2, \dots$ between two gapped boundary phases. Then the critical point associated with this boundary data, if it exists, is expected to be described by the limit of the following diagram:

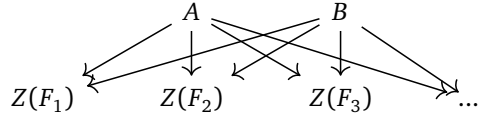
\begin{center}
    \begin{tikzcd}
        & A \arrow[ld] \arrow[d] \arrow[rd] \arrow[rrd] & B \arrow[lld] \arrow[ld] \arrow[d] \arrow[rd] &     \\
 Z(F_1) & Z(F_2)                                        & Z(F_3)                                        & ...
 \end{tikzcd}
\end{center} 

Remark~\ref{rmk:condense_limit} shows that the limit of this diagram is also a condensable algebra.
\end{remark}
\begin{remark}\label{rmk:multiphase}

    In physics, a "triple point" refers to a specific condition where three phases coexist. More generally, configurations involving the simultaneous existence of multiple phases are referred to as \textit{multiphase points}. To obtain the mathematical description of multiphase points, we only need to consider the limit of the following diagram:
    
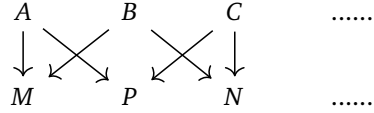
\begin{center}
        \begin{tikzcd}
            A \arrow[d] \arrow[rd] & B \arrow[ld] \arrow[rd] & C \arrow[d] \arrow[ld] & ...... \\
            M        & P          & N         & ......
            \end{tikzcd}
    \end{center}

    Here, $A, B, C$, etc., are Lagrangian algebras in a UMTC $\mathcal{C}$. $M$ is an $A$-$B$-algebra, $P$ is an $A$-$C$-algebra, $N$ is a $B$-$C$-algebra, and so on. In Remark~\ref{rmk:condense_limit}, we show that the limit of this diagram is also a condensable algebra.
\end{remark}

\section{Preliminaries}\label{section3}

\subsection{Condensable algebras}\label{section3.1}

In this subsection, we will recall some basic knowledge about condensable algebra.
We only list the notations and properties we need.
For an introduction to braided fusion category and condensable algebra in more detail, we recommend \cite{fuchs2002tft,kitaev2006anyons,kong2014anyon}.

Let $\m{C}$ be a braided fusion category.

\begin{definition}
    An {\it algebra} in $\m{C}$ is a triple $(A,m_A,u_A)$, where $A$ is an object in $\m{C}$, $m_A$ is a morphism $A\otimes A\to A$ and $u_A: 1 \to A$ satisfying the following conditions: 
    $$
    m_A \circ (m_A \otimes \id_A)  = m_A \circ (\id_A \otimes m_A),
    $$
    $$
    m_A \circ (u_A \otimes \id_A) = \id_A = m_A\circ (\id_A \otimes u_A). 
    $$

    An algebra $A$ is called {\it commutative} if $m_A = m_A \circ \beta_{A,A}$. Here $\beta_{X,Y}:X\otimes Y\to Y\otimes X$ is the braiding of $\m{C}$.
    
\end{definition}

A convenient way to visualize morphisms is using graphs in which lines stand for identity morphisms. We draw the algebra structure of $A$ as follows.

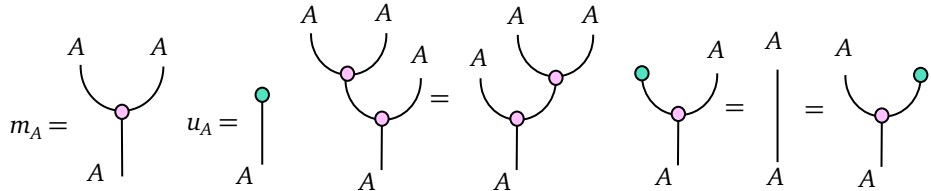
\begin{center}
    
    \tikzset{every picture/.style={line width=0.75pt}} 

    \begin{tikzpicture}[x=0.75pt,y=0.75pt,yscale=-0.8,xscale=0.8]
    
    \draw    (276.14,107.77) -- (275.91,140) ;
    \draw    (112.95,103) -- (113.23,144) ;
    \draw  [draw opacity=0] (139.05,75.03) .. controls (139.05,75.09) and (139.05,75.14) .. (139.05,75.2) .. controls (138.96,90.55) and (127.27,102.92) .. (112.93,102.83) .. controls (98.59,102.74) and (87.04,90.22) .. (87.13,74.86) .. controls (87.13,74.58) and (87.14,74.29) .. (87.15,74.01) -- (113.09,75.03) -- cycle ; \draw   (139.05,75.03) .. controls (139.05,75.09) and (139.05,75.14) .. (139.05,75.2) .. controls (138.96,90.55) and (127.27,102.92) .. (112.93,102.83) .. controls (98.59,102.74) and (87.04,90.22) .. (87.13,74.86) .. controls (87.13,74.58) and (87.14,74.29) .. (87.15,74.01) ;  
    \draw  [fill={rgb, 255:red, 255; green, 193; blue, 255 }  ,fill opacity=1 ] (108.88,103.46) .. controls (108.88,101) and (110.76,99) .. (113.08,99) .. controls (115.39,99) and (117.27,101) .. (117.27,103.46) .. controls (117.27,105.92) and (115.39,107.92) .. (113.08,107.92) .. controls (110.76,107.92) and (108.88,105.92) .. (108.88,103.46) -- cycle ;
    \draw  [draw opacity=0] (300.52,82.31) .. controls (300.54,82.7) and (300.54,83.1) .. (300.54,83.49) .. controls (300.46,97.37) and (289.82,108.56) .. (276.77,108.47) .. controls (263.72,108.39) and (253.2,97.07) .. (253.28,83.19) .. controls (253.28,82.93) and (253.29,82.67) .. (253.3,82.41) -- (276.91,83.34) -- cycle ; \draw   (300.52,82.31) .. controls (300.54,82.7) and (300.54,83.1) .. (300.54,83.49) .. controls (300.46,97.37) and (289.82,108.56) .. (276.77,108.47) .. controls (263.72,108.39) and (253.2,97.07) .. (253.28,83.19) .. controls (253.28,82.93) and (253.29,82.67) .. (253.3,82.41) ;  
    \draw  [fill={rgb, 255:red, 255; green, 193; blue, 255 }  ,fill opacity=1 ] (271.95,107.77) .. controls (271.95,105.3) and (273.83,103.31) .. (276.14,103.31) .. controls (278.46,103.31) and (280.34,105.3) .. (280.34,107.77) .. controls (280.34,110.23) and (278.46,112.23) .. (276.14,112.23) .. controls (273.83,112.23) and (271.95,110.23) .. (271.95,107.77) -- cycle ;
    \draw    (201.06,92.77) -- (201.06,136.6) ;
    \draw  [fill={rgb, 255:red, 80; green, 227; blue, 194 }  ,fill opacity=1 ] (196.87,92.77) .. controls (196.87,90.3) and (198.74,88.31) .. (201.06,88.31) .. controls (203.38,88.31) and (205.25,90.3) .. (205.25,92.77) .. controls (205.25,95.23) and (203.38,97.23) .. (201.06,97.23) .. controls (198.74,97.23) and (196.87,95.23) .. (196.87,92.77) -- cycle ;
    \draw  [draw opacity=0] (278.92,54.08) .. controls (278.94,54.48) and (278.94,54.87) .. (278.94,55.27) .. controls (278.86,69.15) and (268.22,80.33) .. (255.16,80.24) .. controls (242.11,80.16) and (231.6,68.84) .. (231.68,54.96) .. controls (231.68,54.7) and (231.69,54.44) .. (231.69,54.18) -- (255.31,55.11) -- cycle ; \draw   (278.92,54.08) .. controls (278.94,54.48) and (278.94,54.87) .. (278.94,55.27) .. controls (278.86,69.15) and (268.22,80.33) .. (255.16,80.24) .. controls (242.11,80.16) and (231.6,68.84) .. (231.68,54.96) .. controls (231.68,54.7) and (231.69,54.44) .. (231.69,54.18) ;  
    \draw  [fill={rgb, 255:red, 255; green, 193; blue, 255 }  ,fill opacity=1 ] (250.35,79.54) .. controls (250.35,77.08) and (252.22,75.08) .. (254.54,75.08) .. controls (256.86,75.08) and (258.73,77.08) .. (258.73,79.54) .. controls (258.73,82) and (256.86,84) .. (254.54,84) .. controls (252.22,84) and (250.35,82) .. (250.35,79.54) -- cycle ;
    \draw    (360.9,107.77) -- (360.67,140) ;
    \draw  [draw opacity=0] (385.28,82.31) .. controls (385.3,82.7) and (385.3,83.1) .. (385.3,83.49) .. controls (385.22,97.37) and (374.57,108.56) .. (361.52,108.47) .. controls (348.47,108.39) and (337.96,97.07) .. (338.04,83.19) .. controls (338.04,82.93) and (338.05,82.67) .. (338.05,82.41) -- (361.67,83.34) -- cycle ; \draw   (385.28,82.31) .. controls (385.3,82.7) and (385.3,83.1) .. (385.3,83.49) .. controls (385.22,97.37) and (374.57,108.56) .. (361.52,108.47) .. controls (348.47,108.39) and (337.96,97.07) .. (338.04,83.19) .. controls (338.04,82.93) and (338.05,82.67) .. (338.05,82.41) ;  
    \draw  [fill={rgb, 255:red, 255; green, 193; blue, 255 }  ,fill opacity=1 ] (356.71,107.77) .. controls (356.71,105.3) and (358.58,103.31) .. (360.9,103.31) .. controls (363.22,103.31) and (365.09,105.3) .. (365.09,107.77) .. controls (365.09,110.23) and (363.22,112.23) .. (360.9,112.23) .. controls (358.58,112.23) and (356.71,110.23) .. (356.71,107.77) -- cycle ;
    \draw  [draw opacity=0] (408.65,55.08) .. controls (408.67,55.48) and (408.67,55.87) .. (408.67,56.27) .. controls (408.59,70.15) and (397.95,81.33) .. (384.9,81.24) .. controls (371.84,81.16) and (361.33,69.84) .. (361.41,55.96) .. controls (361.41,55.7) and (361.42,55.44) .. (361.43,55.18) -- (385.04,56.11) -- cycle ; \draw   (408.65,55.08) .. controls (408.67,55.48) and (408.67,55.87) .. (408.67,56.27) .. controls (408.59,70.15) and (397.95,81.33) .. (384.9,81.24) .. controls (371.84,81.16) and (361.33,69.84) .. (361.41,55.96) .. controls (361.41,55.7) and (361.42,55.44) .. (361.43,55.18) ;  
    \draw  [fill={rgb, 255:red, 255; green, 193; blue, 255 }  ,fill opacity=1 ] (381.08,82.13) .. controls (381.08,79.66) and (382.96,77.67) .. (385.27,77.67) .. controls (387.59,77.67) and (389.47,79.66) .. (389.47,82.13) .. controls (389.47,84.59) and (387.59,86.59) .. (385.27,86.59) .. controls (382.96,86.59) and (381.08,84.59) .. (381.08,82.13) -- cycle ;
    \draw    (524.33,77) -- (524.33,135) ;
    \draw    (462.09,104.77) -- (461.86,137) ;
    \draw  [draw opacity=0] (486.47,79.31) .. controls (486.49,79.7) and (486.49,80.1) .. (486.49,80.49) .. controls (486.41,94.37) and (475.77,105.56) .. (462.71,105.47) .. controls (449.66,105.39) and (439.15,94.07) .. (439.23,80.19) .. controls (439.23,79.93) and (439.24,79.67) .. (439.24,79.41) -- (462.86,80.34) -- cycle ; \draw   (486.47,79.31) .. controls (486.49,79.7) and (486.49,80.1) .. (486.49,80.49) .. controls (486.41,94.37) and (475.77,105.56) .. (462.71,105.47) .. controls (449.66,105.39) and (439.15,94.07) .. (439.23,80.19) .. controls (439.23,79.93) and (439.24,79.67) .. (439.24,79.41) ;  
    \draw  [fill={rgb, 255:red, 255; green, 193; blue, 255 }  ,fill opacity=1 ] (457.9,104.77) .. controls (457.9,102.3) and (459.77,100.31) .. (462.09,100.31) .. controls (464.41,100.31) and (466.28,102.3) .. (466.28,104.77) .. controls (466.28,107.23) and (464.41,109.23) .. (462.09,109.23) .. controls (459.77,109.23) and (457.9,107.23) .. (457.9,104.77) -- cycle ;
    \draw  [fill={rgb, 255:red, 80; green, 227; blue, 194 }  ,fill opacity=1 ] (435.06,79.29) .. controls (435.06,76.83) and (436.93,74.83) .. (439.25,74.83) .. controls (441.57,74.83) and (443.44,76.83) .. (443.44,79.29) .. controls (443.44,81.75) and (441.57,83.75) .. (439.25,83.75) .. controls (436.93,83.75) and (435.06,81.75) .. (435.06,79.29) -- cycle ;
    \draw    (589.77,105.77) -- (589.54,138) ;
    \draw  [draw opacity=0] (614.15,80.31) .. controls (614.16,80.7) and (614.17,81.1) .. (614.17,81.49) .. controls (614.09,95.37) and (603.44,106.56) .. (590.39,106.47) .. controls (577.34,106.39) and (566.83,95.07) .. (566.91,81.19) .. controls (566.91,80.93) and (566.91,80.67) .. (566.92,80.41) -- (590.54,81.34) -- cycle ; \draw   (614.15,80.31) .. controls (614.16,80.7) and (614.17,81.1) .. (614.17,81.49) .. controls (614.09,95.37) and (603.44,106.56) .. (590.39,106.47) .. controls (577.34,106.39) and (566.83,95.07) .. (566.91,81.19) .. controls (566.91,80.93) and (566.91,80.67) .. (566.92,80.41) ;  
    \draw  [fill={rgb, 255:red, 255; green, 193; blue, 255 }  ,fill opacity=1 ] (585.57,105.77) .. controls (585.57,103.3) and (587.45,101.31) .. (589.77,101.31) .. controls (592.08,101.31) and (593.96,103.3) .. (593.96,105.77) .. controls (593.96,108.23) and (592.08,110.23) .. (589.77,110.23) .. controls (587.45,110.23) and (585.57,108.23) .. (585.57,105.77) -- cycle ;
    \draw  [fill={rgb, 255:red, 80; green, 227; blue, 194 }  ,fill opacity=1 ] (609.95,80.31) .. controls (609.95,77.85) and (611.83,75.85) .. (614.15,75.85) .. controls (616.46,75.85) and (618.34,77.85) .. (618.34,80.31) .. controls (618.34,82.77) and (616.46,84.77) .. (614.15,84.77) .. controls (611.83,84.77) and (609.95,82.77) .. (609.95,80.31) -- cycle ;
    
    \draw (90.38,134) node [anchor=north west][inner sep=0.75pt]   [align=left] {$\displaystyle A$};
    \draw (260.68,140) node [anchor=north west][inner sep=0.75pt]   [align=left] {$\displaystyle A$};
    \draw (183.87,136) node [anchor=north west][inner sep=0.75pt]   [align=left] {$\displaystyle A$};
    \draw (130.16,55) node [anchor=north west][inner sep=0.75pt]   [align=left] {$\displaystyle A$};
    \draw (78.27,55) node [anchor=north west][inner sep=0.75pt]   [align=left] {$\displaystyle A$};
    \draw (40.69,107) node [anchor=north west][inner sep=0.75pt]   [align=left] {$\displaystyle m_{A} =$};
    \draw (152.12,105) node [anchor=north west][inner sep=0.75pt]   [align=left] {$\displaystyle u_{A} =$};
    \draw (224.37,34.94) node [anchor=north west][inner sep=0.75pt]   [align=left] {$\displaystyle A$};
    \draw (293.56,59.94) node [anchor=north west][inner sep=0.75pt]   [align=left] {$\displaystyle A$};
    \draw (271.08,33.94) node [anchor=north west][inner sep=0.75pt]   [align=left] {$\displaystyle A$};
    \draw (345.44,140) node [anchor=north west][inner sep=0.75pt]   [align=left] {$\displaystyle A$};
    \draw (354.1,35.94) node [anchor=north west][inner sep=0.75pt]   [align=left] {$\displaystyle A$};
    \draw (330.75,61.94) node [anchor=north west][inner sep=0.75pt]   [align=left] {$\displaystyle A$};
    \draw (400.81,34.94) node [anchor=north west][inner sep=0.75pt]   [align=left] {$\displaystyle A$};
    \draw (304.34,91) node [anchor=north west][inner sep=0.75pt]   [align=left] {$\displaystyle =$};
    \draw (516.79,137) node [anchor=north west][inner sep=0.75pt]   [align=left] {$\displaystyle A$};
    \draw (446.63,137) node [anchor=north west][inner sep=0.75pt]   [align=left] {$\displaystyle A$};
    \draw (479.51,56.94) node [anchor=north west][inner sep=0.75pt]   [align=left] {$\displaystyle A$};
    \draw (489.42,97) node [anchor=north west][inner sep=0.75pt]   [align=left] {$\displaystyle =$};
    \draw (574.3,138) node [anchor=north west][inner sep=0.75pt]   [align=left] {$\displaystyle A$};
    \draw (559.62,55.94) node [anchor=north west][inner sep=0.75pt]   [align=left] {$\displaystyle A$};
    \draw (539.26,98) node [anchor=north west][inner sep=0.75pt]   [align=left] {$\displaystyle =$};
    \draw (515.79,51.33) node [anchor=north west][inner sep=0.75pt]   [align=left] {$\displaystyle A$};

    \end{tikzpicture}

\end{center}

Such pictures in this work are read from top to bottom. 
If an object is the tensor unit $1$, its corresponding line will be drawn as a dashed line or be omitted from the diagram. We sometimes omit object labels on lines and morphism labels on circles.

\begin{definition}
    A {\it algebra homomorphism} between two algebras $A=(A,m_A,u_A)$ and $B=(B,m_B,u_B)$ is a morphism $f: A \to B$ such that:
    $$
    u_B=f\circ u_A, 
    $$
    $$
    f\circ m_A=m_B\circ(f\otimes f).
    $$
\end{definition}

\begin{definition} \label{def:module} 
    A {\it left module} over an algebra $A=(A,m_A,u_A)$ is a pair $(M, m_M^A)$, where $M$ is an object in $\m{C}$ and $m_M^A: A\otimes M \to M$ such that 
    $$
    m_M^A \circ (m_A \otimes \id_M)  = m_M^A \circ (\id_A \otimes m_M^A)
    $$
    and $m_M^A \circ (u_A \otimes \id_M) =\id_M$. 
    The definition of a right $A$-module is similar. An $A$-$B$ bimodule is a triple $(M, m_M^{A,L},m_M^{B,R})$ such that $(M, m_M^{A,L})$ is a left $A$-module and $(M, m_M^{B,R})$ is a right $B$-module such that 
    $$
    m_M^{B,R} \circ (m_M^{A,L} \otimes \id_B)  = m_M^{A,L} \circ (\id_A \otimes m_M^{B,R}). 
    $$
    \end{definition}

    We sometimes omit superscripts or subscripts from notations such as $m_A$ and $m_M^A$ when there is no ambiguity.

    \begin{definition}
        For an algebra $A$, a {\it left $A$-module homomorphism} is a morphism $f: M \to N$ between two left $A$-modules $(M, m_M^A)$ and $(N, m_N^A)$ such that
        $$
        f \circ m_M^A = m_N^A \circ (\id_A  \otimes f).
        $$
        The definition of a right $A$-module homomorphism is similar. a {\it $A$-$A$ bimodule homomorphism} is a morphism $f: M \to N$ between two $A$-$A$ bimodules $(M, m_M^A)$ and $(N, m_N^A)$ such that $f$ is both a left $A$-module homomorphism and a right $A$-module homomorphism. We denote by ${}_A\mathcal{C}$, $\m{C}_A$, and ${}_A\m{C}_A$ the categories of left $A$-modules, right $A$-modules, and $A$-$A$ bimodules in $\m{C}$, respectively. In each case, morphisms are the corresponding module homomorphisms.
    \end{definition}
    
    \begin{definition}
    A right $A$-module $(M,m_M^{A,R})$ is called a {\it local module} if
    $$
    m_M^{A,R}\circ \beta_{A,M}\circ \beta_{M,A}=m_M^{A,R}.
    $$
    We denote by $\m{C}_A^{loc}$ the full subcategory of $\m{C}_A$ consisting of local $A$-modules.
\end{definition}

    \begin{definition}
    A \textit{(two side) ideal} of an algebra $A$ is a pair $(I,i)$ such that $I$ is an $A$-$A$ bimodule and $i:I\to A$ is a monomorphism of $A$-$A$ bimodules.

    An ideal $(I,i)$ of $A$ is called \textit{nilpotent} if there exists $n\geqslant 2$ such that $I^n:=Im(I^{\otimes n}\stackrel{i^{\otimes n}}{\rightarrow}A^{\otimes n}\stackrel{m^n}{\rightarrow}A)=0$. Here $m^n$ means the $n$-th iterated multiplication of $A$.
\end{definition}

\begin{definition}
An algebra $A$ is called {\it connected} if $\dim Hom_{\m{C}}(1,A)=1$.
\end{definition}

\begin{definition}
    An algebra $A$ is called {\it separable algebra} if there exists an $A$-$A$ bimodule morphism $\delta_A:A\to A\otimes A$ such that $m_A\circ \delta_A=id_A$. 
\end{definition}

\begin{definition}[\cite{kong2014anyon}]
    A {\it 2d condensable algebra} in a braided fusion category is a connected commutative separable algebra. A {\it 1d condensable algebra} in a fusion category is a connected separable algebra.
\end{definition}

\begin{definition}
    A 2d condensable algebra $A\in \m{C}$ is called a {\it Lagrangian algebra} if $\m{C}_A^{loc}\simeq Vec$ as braided fusion categories.
\end{definition}

In this paper, unless otherwise specified, the term "condensable algebra" always refers to 2d condensable algebra.

\begin{remark}
    The definitions above do not require a pivotal structure on the ambient braided fusion category. If the category is moreover unitary, equipped with its canonical spherical structure, then every connected separable algebra admits a unique symmetric normalized-special Frobenius algebra structure \cite{fuchs2002tft,kong2014anyon}. Conversely, every special Frobenius algebra is separable. Thus, in this setting, the definition of a 1d condensable algebra is equivalent to that of a connected symmetric normalized-special Frobenius algebra, and the definition of a 2d condensable algebra is equivalent to that of a connected commutative symmetric normalized-special Frobenius algebra.

\end{remark}

\begin{remark}
    A commutative separable algebra is also called \textit{\'etale algebra}.
\end{remark}

\subsection{Center and cospan}\label{cen_and_cos}

In this subsection, we recall the definition of center and cospan. Most of the results in this subsection can be found in \cite{davydov2010centre,davydov2015functoriality}.

Let $\m{C}$ be monoidal category with tensor product $\otimes$, $\m{M}$ be a left $\m{C}$-module with an action function $\odot: \m{C}\times \m{M} \to \m{M}$.

\begin{definition}
    For $M,N\in\m{M}$, the {\it  internal hom} $[M,N]$ is defined as the object in $\m{C}$ such that there is a natural isomorphism:
    $$
    Hom_{\m{M}}(X\odot M,N)\simeq Hom_{\m{C}}(X,[M,N]) ,\quad\forall X\in \m{C}.
    $$
\end{definition}

It is equivalent to stating that the functor $Hom_{\m{M}}(-\odot M ,N)$ is representable by $[M,N]$. Let $ev_M: [M,N]\odot M \to N$ be the  map corresponding to $id_{[M,N]}$ under the natural isomorphism. The internal hom $[M,N]$ can also be defined as the terminal object among pairs $(X,f)$, where $X \in \m{C}$ and $f: X\odot M \to N$ is a morphism in $\m{M}$. This means that for any such pair $(X,f)$, there exists a unique morphism $\underline{f}: X \to [M,N]$ making the following diagram commute:

\begin{center}
    \begin{tikzcd}
        X\odot M \arrow[rd, "f"] \arrow[rr, "\underline{f}\odot id_M"] &   & {[M,N]\odot M} \arrow[ld, "ev_M"'] \\
                                                                       & N &              .                     
        \end{tikzcd}
\end{center}

\begin{remark}
    For $M, N \in \m{M}$, the internal hom $[M,N] \in \m{C}$ does not always exist. 
In this subsection, all internal homs we write are considered under the assumption that they exist.

\end{remark}

\begin{lemma}\cite[Lemma 2.1.6]{kong2018center}\label{compute_internalhom}
    Let $\mathcal{C}$ be a rigid monoidal category that admits coequalizers, and let $\mathcal{M}=\mathcal{C}_A$ for some algebra $A\in \mathcal{C}$. Then for every pair of objects $x,y\in \mathcal{M}$, the internal hom $[x,y]$ exists, and we have $[x,y]\simeq (x\otimes_A y^R)^L$.
    
\end{lemma}
\begin{proposition}
    For any objects $M, N \in \m{M}$, the internal hom $[M,M]$ is an algebra in $\m{C}$.  Furthermore, $[M,N]$ is a right $[M,M]$-module and $[N,M]$ is a left $[M,M]$-module.

    The unit map $1\to [M,M]$ is induced by the morphism $1\odot M \simeq M$ under the natural isomorphism $Hom_{\m{C}}(1,[M,M])\simeq Hom_{\m{M}}(1\odot M,M)$. The morphism $[N,P] \otimes [M,N] \to [M,P]$, induced by the composition $ev_N \circ (\mathrm{id} \otimes ev_M) \colon [N,P] \otimes [M,N] \odot M \to [N,P] \odot N \to P$, determines the multiplication of $[M,M]$, as well as the actions of $[M,M]$ on $[M,N]$ and $[N,M]$.
\end{proposition}

Let $F: \m{M}\to \m{M}'$ be a $\m{C}$-module functor. For $M,N\in \m{M}$, suppose $[M,N]$ and $[F(M),F(N)]$ exist. Then there is a natural morpism $[F]_{[M,N]}$ determined by the follows

\begin{center}
    \begin{tikzcd}
        {[M,N]\odot F(M)} \arrow[d] \arrow[rr, "{\exists ! [F]_{[M,N]}\odot id}", dashed] &  & {[F(M),F(N)]\odot F(M)} \arrow[d, "ev_{F(M)}"'] \\
        {F([M,N]\odot M)} \arrow[rr, "F(ev_M)"]                                           &  & F(N)                                       .    
        \end{tikzcd}
\end{center}
Here, the left vertical morphism is the natural isomorphism of the $\m{C}$-module functor $F$. For simplicity, we sometimes write $[F]$ for $[F]_{[M,M]}$.

We use $Z(\m{C})$ to denote the {\it  Drinfeld center} of monoidal category $\m{C}$. The {\it $\alpha$-induction} is a monoidal functor from $Z(\m{C})$ to the category $Fun_{\m{C}}(\m{M},\m{M})$ (also denoted by $\m{C}_{\m{M}}^*$) of $\m{C}$-module endofunctors of $\m{M}$:
$$
\alpha(Z): \m{M} \to \m{M} ,\quad \alpha(Z)(M)=Z\otimes M
$$
with the $\m{C}$-module structure:

\begin{center}
    \begin{tikzcd}
        \alpha(Z)(X\odot M) \arrow[d, "="] \arrow[rrr, dashed] &                                                &                               & X\odot\alpha(Z)(M) \arrow[d, "="] \\
        Z\odot(X\odot M) \arrow[r]                             & (Z\otimes X)\odot M \arrow[r, "{\beta_{Z,X}}"] & (X\otimes Z)\odot M \arrow[r] & X\odot(Z\odot M)   .             
        \end{tikzcd}
\end{center}

Here $\beta_{Z,-}$ is the half-braiding of $Z$. Due to the monoidal functor $\alpha$, we can view $\m{C}_{\m{M}}^*$ as $Z(\m{C})$-module. 

\begin{definition}(\cite[Sec.6]{davydov2010centre})
    The {\it full center} $Z(\m{M})$ of $\m{M}$ is the terminal object in the category of pairs $(Z,f)$ where $Z\in Z(\m{C})$ and $f:\alpha(Z)\to id_{\m{M}}$ is a $\m{C}$-module natural transformation. 
\end{definition}

It is equivalent to stating that $Z(\m{M})$ is the terminal object in the comma category $\alpha\downarrow id_{\m{M}}$. By definition, $Z(\m{M})$ coincides with the internal hom $[id_{\m{M}},id_{\m{M}}]$ for the $Z(\m{C})$-module $\m{C}_{\m{M}}^*$. In fact, by \cite[Prop 6.1]{davydov2010centre}, $Z(\m{M})$ is a commutative algebra in $Z(\m{C})$.

Let $\m{M},\m{N}$ be left $\m{C}$-modules. The category $Fun_{\m{C}}(\m{M},\m{N})$ consists of the $\m{C}$-module functors from $\m{M}$ to $\m{N}$. By functor composition, it is a left $\m{C}_{\m{N}}^*$-module and a right $\m{C}_{\m{M}}^*$-module. The $\alpha$-inductions $\alpha_{\m{M}}:Z(\m{C}) \to \m{C}_{\m{N}}^*$ and $\alpha_{\m{N}}:Z(\m{C}) \to \m{C}_{\m{M}}^*$ endow $Fun_{\m{C}}(\m{M},\m{N})$ with both a left and a right $Z(\m{C})$-module structure. These two structures are isomorphic, as shown in \cite[Lemma 3.3]{davydov2015functoriality}. Thus, without loss of generality, we will only consider the left $Z(\m{C})$-module structure.

\begin{definition}
    Let $F:\m{M}\to \m{N}$ be a $\m{C}$-module functor. The {\it full centralizer} $Z(F)$ of the functor $F$ is defined to be the internal hom $[F,F]$ in $Z(\m{C})$ with respect to the $Z(\m{C})$-module structure on $Fun_{\m{C}}(\m{M},\m{N})$.
\end{definition}

A $\m{C}$-module functor $F:\m{M}\to \m{N}$ induces two functors by composition:
$$
F\circ -:Fun_{\m{C}}(\m{M},\m{M}) \to Fun_{\m{C}}(\m{M},\m{N}), \quad -\circ F:Fun_{\m{C}}(\m{N},\m{N}) \to Fun_{\m{C}}(\m{M},\m{N}).
$$
It can be verified that both functors are $Z(\m{C})$-module functors. This gives rise to two induced morphisms in $Z(\m{C})$:

\begin{equation}\label{def_cospan}
    \begin{tikzcd}
        & Z(F) &                                                      \\
Z(\mathcal{M}) \arrow[ru, "{[F\circ -]_{[id,id]}}"] &      & Z(\mathcal{N}) \arrow[lu, "{[-\circ F]_{[id,id]}}"'].
\end{tikzcd}
\end{equation}
A diagram like the one above is called a {\it cospan}. By \cite[Lemma 2.9]{davydov2015functoriality}, the morphisms in this cospan are both algebra homomorphisms.

\begin{definition}
    Let $A$ be an algebra in a braided monoidal category $\m{C}$. The {\it left center} $C_l(A)$ (resp. {\it right center} $C_r(A)$) in $\m{C}$ is defined as the terminal object in the category of morphisms $f:X \to A$ such that the following diagram is commutative:

      $$
      \raisebox{2.5em}{
\xymatrix@R=1em{ X\otimes A \ar[r]^{f\otimes id} \ar[dd]_{\beta_{X,A}} & A\otimes A \ar[rd]^{m_A}\\ & & A\\ A\otimes X \ar[r]_{id \otimes f} & A\otimes A \ar[ur]_{m_A} }
}
\quad \Bigg(\text{resp.} ~~
\raisebox{2.5em}{
\xymatrix@R=1em{ A\otimes X \ar[r]^{id \otimes f} \ar[dd]_{\beta_{A,X}} & A\otimes A \ar[rd]^{m_A}\\ & & A\\ X\otimes A \ar[r]_{f\otimes id} & A\otimes A \ar[ur]_{m_A} }
}
\Bigg).
      $$
      Here $\beta$ is the braiding of $\m{C}$.
\end{definition}

By \cite[Proposition 5.1]{davydov2010centre}, the left center and the right center admit unique algebra structures such that the morphisms $C_l(A) \to A$ and $C_r(A) \to A$ are algebra homomorphisms in $\m{C}$. Moreover, these algebra structures on $C_l(A)$ and $C_r(A)$ are commutative. By \cite[Corollary 3.9]{davydov2015functoriality}, the morphisms in \ref{def_cospan} factor through the left and right centers of $Z(F)$:

\begin{center}
    \begin{tikzcd}
        C_r(Z(F)) \arrow[r]                 & Z(F) & C_l(Z(F)) \arrow[l]                 \\
        Z(\mathcal{M}) \arrow[u] \arrow[ru] &      & Z(\mathcal{N}) \arrow[u] \arrow[lu].
        \end{tikzcd}
\end{center}

\begin{definition}
    Let $A$ be an algebra in a monoidal category $\m{C}$. The {\it full center} of $A$ is an object in $Z(\m{C})$ together with a morphism $Z(A) \to A$ in $\m{C}$,terminal among pairs $(X,\zeta)$, where  $X\in Z(\m{C})$ and $\zeta:X\to A$ is a morphism in $\m{C}$ such that the following diagram commutes:
    \begin{center}
        \begin{tikzcd}
            X\otimes A \arrow[dd, "{\beta_{X,A}}"'] \arrow[r, "\zeta\otimes id"] & A\otimes A \arrow[rd, "m_A"] &   \\
                                                                                 &                              & A \\
            A\otimes X \arrow[r, "id\otimes \zeta"]                              & A\otimes A \arrow[ru, "m_A"]. &  
            \end{tikzcd}
    \end{center}
    Here $\beta_{X,A}$ is the half-braiding of $X$. It means that for any such pair $(X,\zeta)$ here exist a unique morphism $X\to Z(A)$ in $Z(\m{C})$ such that the following diagram commutes:
\begin{center}
    \begin{tikzcd}
        X \arrow[rd, "\zeta"'] \arrow[r] & Z(A) \arrow[d] \\
                                         & A            . 
        \end{tikzcd}
\end{center}

\end{definition}

By \cite[Theorem 6.2]{davydov2010centre}, the full center $Z(A)$ of the algebra $A$ is equal to the full center $Z(\m{C}_A)$ of the $\m{C}$-module category $\m{C}_A$. We also know that $Z(\m{C}_A)$ coincides with the internal hom $[id_{\m{C}_A},id_{\m{C}_A}]$ of the $Z(\m{C})$-module $Fun_{\m{C}}(\m{C}_A,\m{C}_A)$. In the case that $\m{C}$ is a fusion category, we have $Fun_{\m{C}}(\m{C}_A,\m{C}_A)\simeq {}_A\m{C}_A$, which implies that $Z(\m{C}_A)$ is also equal to the internal hom $[1_{{}_A\m{C}_A},1_{{}_A\m{C}_A}]$ of the $Z(\m{C})$-module ${}_A\m{C}_A$.

Now, let $\m{C}$ be a fusion category and $\m{M}$ be an indecomposable semisimple $\m{C}$-module category. For any objects $M,N\in \m{M}$, the internal hom $[M,N]$ always exists. Furthermore, if we choose a non-zero object in $\m{M}$, we have the following theorem.
\begin{theorem}(\cite{ostrik2003module})\label{thm:int}
    There is an equivalence:
    $$
    [M,-]:\m{M}\to \m{C}_{[M,M]}, \quad N\mapsto [M,N].
    $$
\end{theorem}

Thus, for a fusion category $\m{C}$ and an indecomposable semisimple $\m{C}$-module category $\m{M}$, we can always find an algebra $A \in \m{C}$ such that $\m{M} \simeq \m{C}_A$. Furthermore, it can be proven that $Fun_{\m{C}}(\m{C}_A,\m{C}_B)\simeq {}_A\m{C}_B$. 

For indecomposable semisimple $\m{C}$-module categories $\m{M}$ and $\m{N}$, the category $Fun_{\m{C}}(\m{M},\m{N})$ is an indecomposable semisimple module over $\m{C}_{\m{M}}^*$ and $\m{C}_{\m{N}}^*$, according to \cite[Theorem 3.31]{etingof2004finite}. Via the $\alpha$-induction, $Fun_{\m{C}}(\m{M},\m{N})$ becomes a finite semisimple $Z(\m{C})$-module. Consequently, all internal homs of the $Z(\m{C})$-module $Fun_{\m{C}}(\m{M},\m{N})$ are guaranteed to exist.

Let $F, F' \in Fun_{\m{C}}(\m{M},\m{N})$. There are three internal homs between $F$ and $F'$, with values in $\m{C}_{\m{M}}^*$, $\m{C}_{\m{N}}^*$, and $Z(\m{C})$. We denote these by $[F,F']_{\m{C}_{\m{M}}^*}$, $[F,F']_{\m{C}_{\m{N}}^*}$, and $[F,F']$, respectively. Let $\alpha_{\m{M}}^\vee$ and $\alpha_{\m{N}}^\vee$ be the right adjoints of the $\alpha$-inductions $\alpha_{\m{M}}$ and $\alpha_{\m{N}}$. For any object $X \in Z(\m{C})$, we have the following relations:
$$
Hom_{Fun_{\m{C}}(\m{M},\m{N})}(F \odot X,F')\simeq Hom_{\m{C}_{\m{M}}^*}(\alpha_{\m{M}}(X),[F,F']_{\m{C}_{\m{M}}^*}) \simeq Hom_{Z(\m{C})}(X,\alpha_{\m{M}}^\vee([F,F']_{\m{C}_{\m{M}}^*})),
$$
$$
Hom_{Fun_{\m{C}}(\m{M},\m{N})}(X\odot F,F')\simeq Hom_{\m{C}_{\m{N}}^*}(\alpha_{\m{N}}(X),[F,F']_{\m{C}_{\m{N}}^*}) \simeq Hom_{Z(\m{C})}(X,\alpha_{\m{N}}^\vee([F,F']_{\m{C}_{\m{N}}^*})).
$$
Thus, we have 
\begin{equation}\label{com_internal_hom}
[F,F']\simeq \alpha_{\m{M}}^\vee([F,F']_{\m{C}_{\m{M}}^*})\simeq \alpha_{\m{N}}^\vee([F,F']_{\m{C}_{\m{N}}^*}).
\end{equation}

\subsection{Boundaries of 2+1D topological orders}\label{section3.3}

In this subsection, we recall some results about boundaries of 2+1D topological orders. In anyon condensation theory \cite{kong2014anyon}, a 2+1D topological order is described by a UMTC $\m{C}$. The gapped boundaries of this topological order are described by UFCs, which are classified by the Lagrangian algebras in $\m{C}$. Let $\m{S}$ be a UFC that serves as a boundary of $\m{C}$. Then $Z(\m{S}) \simeq \m{C}$ as UMTCs. There exists a Lagrangian algebra $S$ in $\m{C}$ such that $\m{C}_S \simeq \m{S}$. Furthermore, this algebra $S$ is identified with the internal hom $[1_{\m{S}},1_{\m{S}}]$ of the $\m{C}$-module $\m{S}$. From Section~\ref{cen_and_cos}, we know that $[1_{\m{S}},1_{\m{S}}]$ is equal to the full center $Z(\m{S})$ of $\m{S}$-module $\m{S}$ and is also equal to the full center $Z(1_{\m{S}})$ of the algebra $1_{\m{S}}$ in $\m{S}$.

The boundaries of $\m{C}$ can also be determined by the 1d condensable algebras in $\m{S}$. For a boundary $\m{A}$ of $\m{C}$, we can find a 1d condensable algebra $a$ in $\m{S}$ such that $\m{A} \simeq {}_a \m{S}_a$. Furthermore, we have $Z({}_a\m{S}_a) \simeq \m{C}$ as UMTCs and  ${}_a\m{S}_a \simeq \m{C}_{Z(a)}$ as UFCs \cite[Lemma 7.8]{davydov2015functoriality}.

The category of 0+1D defects between ${}_a\m{S}_a$ and ${}_b\m{C}_b$ is given by $_a\m{S}_b \simeq Fun_{\m{S}}(\m{S}_a,\m{S}_b)$. A given 0+1D defect between two boundaries ${}_a\m{S}_a$ and ${}_b\m{S}_b$ of $\m{C}$ is described by a pair $({}_a\m{S}_b,F)$, where $F$ is a chosen object in ${}_a\m{S}_b$. Figure~\ref{boundary_and_defect} illustrates two boundaries of a 2+1D topological order and a 0+1D defect between them. As shown in Section~\ref{cen_and_cos}, these boundary data induce a cospan in $\m{C}$:

\begin{equation}\label{critical_cospan}
    \begin{tikzcd}
        & Z(F)=[F,F] &                                                      \\
Z(\mathcal{S}_{a})= Z(a)=[id_{\m{S}_a},id_{\m{S}_a}] \arrow[ru, "{[F\circ -]_{[id,id]}}"] &      & Z(\mathcal{S}_{b})=Z(b)=[id_{\m{S}_b}, id_{\m{S}_b}] \arrow[lu, "{[-\circ F]_{[id,id]}}"'].
\end{tikzcd}
\end{equation}

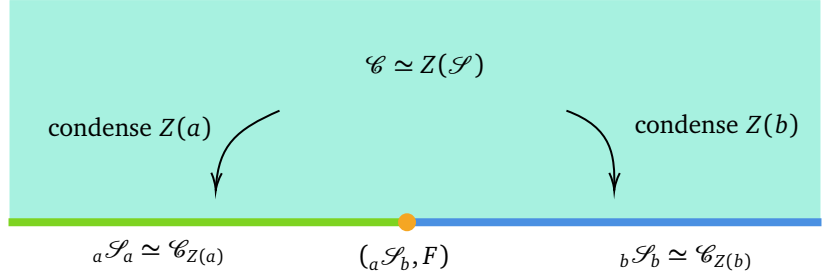
\begin{figure}
    \centering
        \tikzset{every picture/.style={line width=0.75pt}} 
    
        \begin{tikzpicture}[x=0.75pt,y=0.75pt,yscale=-0.8,xscale=0.8]
        
        \draw  [color={rgb, 255:red, 255; green, 255; blue, 255 }  ,draw opacity=1 ][fill={rgb, 255:red, 80; green, 227; blue, 194 }  ,fill opacity=0.5 ] (70,40) -- (580,40) -- (580,180) -- (70,180) -- cycle ;
        \draw [color={rgb, 255:red, 126; green, 211; blue, 33 }  ,draw opacity=1 ][line width=3]    (70,180) -- (320,180) ;
        \draw [color={rgb, 255:red, 74; green, 144; blue, 226 }  ,draw opacity=1 ][line width=3]    (580,180) -- (320,180) ;
        \draw  [color={rgb, 255:red, 245; green, 166; blue, 35 }  ,draw opacity=1 ][fill={rgb, 255:red, 245; green, 166; blue, 35 }  ,fill opacity=1 ] (315,180) .. controls (315,177.24) and (317.24,175) .. (320,175) .. controls (322.76,175) and (325,177.24) .. (325,180) .. controls (325,182.76) and (322.76,185) .. (320,185) .. controls (317.24,185) and (315,182.76) .. (315,180) -- cycle ;
        \draw    (240,110) .. controls (212.89,121.76) and (201.78,131.6) .. (200.09,158.34) ;
        \draw [shift={(200,160)}, rotate = 272.73] [color={rgb, 255:red, 0; green, 0; blue, 0 }  ][line width=0.75]    (10.93,-3.29) .. controls (6.95,-1.4) and (3.31,-0.3) .. (0,0) .. controls (3.31,0.3) and (6.95,1.4) .. (10.93,3.29)   ;
        \draw    (420,110) .. controls (449.42,120.67) and (450.31,138.87) .. (450.03,158.2) ;
        \draw [shift={(450,160)}, rotate = 270.95] [color={rgb, 255:red, 0; green, 0; blue, 0 }  ][line width=0.75]    (10.93,-3.29) .. controls (6.95,-1.4) and (3.31,-0.3) .. (0,0) .. controls (3.31,0.3) and (6.95,1.4) .. (10.93,3.29)   ;
        
        \draw (291,72) node [anchor=north west][inner sep=0.75pt]   [align=left] {$\displaystyle \mathcal{C} \simeq Z(\mathcal{S})$};
        \draw (121,189) node [anchor=north west][inner sep=0.75pt]   [align=left] {$\displaystyle _{a}\mathcal{S}_{a} \simeq \mathcal{C}_{Z(a)}$};
        \draw (451,192) node [anchor=north west][inner sep=0.75pt]   [align=left] {$\displaystyle _{b}\mathcal{S}_{b} \simeq \mathcal{C}_{Z( b)}$};
        \draw (288,192) node [anchor=north west][inner sep=0.75pt]   [align=left] {$\displaystyle (_{a}\mathcal{S}_{b} ,F)$};
        \draw (93,112) node [anchor=north west][inner sep=0.75pt]   [align=left] {condense $\displaystyle Z( a)$};
        \draw (461,110) node [anchor=north west][inner sep=0.75pt]   [align=left] {condense $\displaystyle Z( b)$};

        \end{tikzpicture}
        
    \caption{boundaries of a 2+1D topological order and 0+1D defect between boundaries}
    \label{boundary_and_defect}
    \end{figure}
    
    The above cospan contains the data of pure boundary phase transitions of the 2+1D topological order $\m{C}$. We refer to such a cospan in $\m{C}$ as a \textit{critical cospan}.

    The physical picture of the critical point between the phases ${}_a\m{S}_a$ and ${}_b\m{S}_b$ can be conceptualized as a "uniform mixture" of these two phases.
    When a $0+1$D defect $({}_a\mathcal{S}_b, F)$ is introduced between the ${}_a\mathcal{S}_a$ and ${}_b\mathcal{S}_b$ phases, its dual defect $({}_b\mathcal{S}_a, F^{\vee})$ is automatically induced between the ${}_b\mathcal{S}_b$ and ${}_a\mathcal{S}_a$ phases. We view $F$ as an object in $Fun_{\m{S}}(\m{S}_a,\m{S}_b)$. We choose $F^\vee$ to be the left adjoint of $F$. By \cite[Lemma 7.6]{davydov2015functoriality}, we have $[F,F]_{{}_b\m{S}_b} \simeq F\circ F^\vee$ and $[F^\vee,F^\vee]_{{}_b\m{S}_b} \simeq (F^\vee)^R\circ F^\vee \simeq F\circ F^\vee$. Thus, by Equation~\ref{com_internal_hom}, we obtain the following isomorphism:
    $$
    [F,F]\simeq \alpha^\vee_{{}_b\m{S}_b}([F,F]_{{}_b\m{S}_b})\simeq \alpha^\vee_{{}_b\m{S}_b}([F^\vee,F^\vee]_{{}_b\m{S}_b})\simeq [F^\vee,F^\vee].
    $$
    Therefore, it suffices to consider the cospan in Diagram~\ref{critical_cospan}.

    We propose that the limit of  cospan in Diagram~\ref{critical_cospan}, i.e., the fiber product $Z(a)\times_{Z(F)} Z(b)$, is a condensable algebra in $\m{C}$. This condensable algebra determines the critical point of the phase transition between the boundaries $_a\m{S}_a$ and $_b\m{S}_b$. We will discuss the properties of this fiber product in Section~\ref{section4}.
    
    \begin{remark}
     For the physical applications considered here, we assume that $\m{C}$ is a unitary modular tensor category (UMTC). However, most of our results do not require unitarity, and some hold more generally for braided fusion categories. Accordingly, we do not impose the unitarity assumption in Sections~\ref{section4} and~\ref{section5}.
    \end{remark}

    \begin{remark}
     If we allow the 0+1D defect to be composite, i.e., the category of $0+1D$ domain walls $\m{M}$ between two boundaries $\m{A}$ and $\m{B}$ is a decomposable $\m{A}$-$\m{B}$ bimodule category, then $\m{A}$ and $\m{B}$ may become multi-fusion categories. Due to the structure of multi-fusion categories, we still maintain the equivalence $Z(\m{A})\simeq Z(\m{B})$. These composite defects correspond to the multi-phase transitions \cite{kong2015boundary}. We will elaborate on this in Section \ref{section4.bimod}.
    \end{remark}
\section{Fiber product of condensable algebras}\label{section4}

\subsection{Condensable algebra structure on fiber product}

In this subsection, we will show that the limit of Diagram~\ref{critical_cospan} is a condensable algebra. 
\begin{lemma}\label{monomorphism}
Let $\m{C}$ be a fusion category, and let $A,B$ be algebras in $\m{C}$. Let $0\neq f:A\to B$ be an algebra homomorphism. If $A$ is connected and separable, then $f$ is a monomorphism.
\end{lemma}

\begin{proof}
Since $A$ is separable, the category $\m{C}_A$ is semisimple. Moreover, $Hom{\m{C}_A}(A,A)\simeq Hom_{\m{C}}(1,A)\simeq \mathbb{C}$. Therefore, $A$ is a simple object of $\m{C}_A$. The algebra homomorphism $f$ endows $B$ with a right $A$-module structure and $f:A\to B$ is a morphism in $\m{C}_A$. Since $A$ is simple in the semisimple category $\m{C}_A$, the morphism $f\neq 0$ is a monomorphism in $\m{C}_A$, we have $Ker(f)=0\in \m{C}_A$. Forgetting the right $A$-module structures, $0=Ker(f)\in\m{C}$. Thus, $f$ is a monomorphism in $\m{C}$.
\end{proof}

\begin{lemma}\label{sep_subalgebra}
Let $\m{C}$ be a fusion category and let $A\in \m{C}$ be a connected separable algebra. Then any subalgebra of $A$ is also connected and separable.
\end{lemma}

\begin{proof}
Since $A$ is connected, $B$ is also connected.

Let $\iota:B\to A$ be a monomorphism of algebras. Suppose that $B$ is not separable. By \cite[Theorem 7.1]{coulembier2025simple}, $B$ contains a non-zero nilpotent ideal. Let $i:I\to B$ be such an ideal and set $j=\iota\circ i$. Consider
$f=m_A\circ(j\otimes id_A):I\otimes A\to A$.
Then $f$ is a morphism of right $A$-modules. Moreover,
$f\circ(id_I\otimes u_A)=j\neq 0$, hence $f\neq 0$. Since $A$ is a simple right $A$-module, $f$ is an epimorphism.
 
Define $f^r=f\circ(id_I\otimes f^{r-1}):I^{\otimes r}\otimes A\to A$
for $r\in\mathbb{N}_{+}$, with $f^1=f$. Since tensor product in $\m{C}$ is exact, it follows inductively that each $f^r$ is an epimorphism. On the other hand, since $I$ is nilpotent, there exists $n \geqslant 2$ such that $I^n=Im(m_n:I^{\otimes n}\to B^{\otimes n}\to B)=0$. Hence
$f^n=m_A\circ\bigl((\iota\circ m_n)\otimes id_A\bigr)=0$,
contradicting the fact that $f^n$ is an epimorphism. Therefore, $B$ is separable.
\end{proof}

\begin{remark}
\cite[Theorem 7.1]{coulembier2025simple} is stated for exact algebras in finite tensor categories. In a fusion category over $\mathbb{C}$, an algebra is exact if and only if it is separable \cite{etingof2016tensor}\cite{kong2017semisimple}.
\end{remark}

\begin{proposition}\label{1d_condensable}
Let $\m{C}$ be a fusion category, and let $A,B \in \m{C}$ be two connected separable algebras. Let $M \in \m{C}$ be an algebra, and suppose that there exist two algebra homomorphisms $f_A:A \to M$ and $f_B:B \to M$. Denote by $(D := A \times_M B, \pi_A:D \to A, \pi_B:D \to B)$ the limit of the cospan
\begin{equation}\label{Frobenius_cospan} 
\begin{tikzcd} & M \\ A \arrow[ru, "f_A"] & & B \arrow[lu, "f_B"']. 
\end{tikzcd} 
\end{equation}
Then $D$ is also a connected separable algebra, and the morphisms $\pi_A$ and $\pi_B$ are algebra homomorphisms.
\end{proposition}

\begin{proof}
Let $Alg(\m{C})$ denote the category of algebras and algebra homomorphisms in $\m{C}$. The forgetful functor $Alg(\m{C})\to \m{C}$ preserves finite limits. Therefore, the limit of the cospan~\ref{Frobenius_cospan} in $\m{C}$ carries a canonical algebra structure, with respect to which $\pi_A$ and $\pi_B$ are algebra homomorphisms.

Suppose that $x_1,x_2:X\to D$ satisfy $\pi_A\circ x_1=\pi_A\circ x_2$. Then
$f_B\circ \pi_B\circ x_1=f_A\circ \pi_A\circ x_1=f_A\circ \pi_A\circ x_2=f_B\circ \pi_B\circ x_2$. By Lemma~\ref{monomorphism}, $f_B$ is a monomorphism, and hence $\pi_B\circ x_1=\pi_B\circ x_2$. By the universal property of $D$, we obtain $x_1=x_2$. Thus, $\pi_A$ is a monomorphism. It then follows from Lemma~\ref{sep_subalgebra} that $D$ is a connected separable algebra.
\end{proof}

This proposition shows that the fiber product of two (1d) condensable algebras over an algebra is again a (1d) condensable algebra.

\begin{theorem}\label{Thm:2d condense}
    Let $\m{C}$ be a braided fusion category, and let $A,B \in \m{C}$ be two (2d) condensable algebras. Let $M \in \m{C}$ be an algebra. Suppose there exist two algebra homomorphisms $f_A:A \to M$ and $f_B:B \to M$. Denote by $(D := A \times_M B, \pi_A:D \to A, \pi_B:D \to B)$ the limit of the cospan \ref{Frobenius_cospan}, then $D$ is also a (2d)condensable algebra.
    \end{theorem}
    \begin{proof}
        By Proposition~\ref{1d_condensable}, we only need to show $D$ is commutative. There is a commutative diagram:

\begin{equation}\label{commutative}
    \begin{tikzcd}
        D\otimes D \arrow[ddd, "{\beta_{D,D}}"'] \arrow[rrr, "m_D"] \arrow[rd, "\pi_A\otimes \pi_A"] &                                                        &   & D \arrow[ld, "\pi_A"'] \\
                                                                                                     & A\otimes A \arrow[d, "{\beta_{A,A}}"] \arrow[r, "m_A"] & A &                        \\
                                                                                                     & A\otimes A \arrow[ru, "m_A"']                          &   &                        \\
        D\otimes D \arrow[rrruuu, "m_D", bend right=49] \arrow[ru, "\pi_A\otimes \pi_A"]             &                                                        &   &                       .
        \end{tikzcd}
    \end{equation}

Here, $\beta$ is the braiding of $\m{C}$.
The commutativity of the left block is due to the naturality of $\beta$.
The commutativity of the upper block and the lower-left block is due to the definition of $m_D$.
The commutativity of the middle block is because $A$ is a commutative algebra. A similar diagram holds for $B$. By the universal property of $D$, we obtain $m_D=  m_D \circ\beta_{D,D}$.
\end{proof}

    \begin{corollary}
        The limit $Z(a)\times _{Z(F)}Z(b)$ of the critical cospan \ref{critical_cospan} is a condensable algebra.
    \end{corollary}

    \begin{remark}\label{rmk:condense_limit}
When $\m{C}$ is a braided fusion category and $A\in \m{C}$ is a commutative separable algebra, any subalgebra of $A$ is also commutative and separable \cite[Remark 5.7]{ostrik2026non}. Using this result, we obtain another proof of Theorem~\ref{Thm:2d condense}. Moreover, since finite products of commutative separable algebras are again commutative and separable, and any finite limit can be expressed as an equalizer of two morphisms between finite products, any finite limit of commutative separable algebras is a subalgebra of a finite product of commutative separable algebras. Therefore, commutative separable algebras are closed under finite limits. Moreover, the limits in Remarks~\ref{rmk:multi F} and~\ref{rmk:multiphase} are condensable algebras.
\end{remark}

    Now we know that any subalgebra of a condensable algebra is again condensable. We refer to such a subalgebra as a \textit{condensable subalgebra} to emphasize that it is condensable. We call an algebra monomorphism $D\to A$ between condensable algebras an \textit{embedding}. Taking a fiber product determines a common condensable subalgebra together with embeddings into the two given condensable algebras. It is then natural to ask whether every common condensable subalgebra of two condensable algebras can be obtained as a fiber product. In Proposition~\ref{Por:all_fiber_product}, we give an affirmative answer.

Consider a braided fusion category $\m{C}$ and two condensable algebras $A,B\in \m{C}$. Let $D$ be a common condensable subalgebra of $A$ and $B$. We denote the embeddings of $D$ into $A$ and $B$ by $\iota_A$ and $\iota_B$, respectively.
Consider relative tensor product $M=A\otimes_D B$, which is the coequalizer

\begin{center}
    \begin{tikzcd}
        A\otimes D\otimes B \arrow[r, "m_A^D\otimes id", shift left] \arrow[r, "id\otimes m_B^D"', shift right] & A\otimes B \arrow[r, "\rho"] & A\otimes_D B.
        \end{tikzcd} 
\end{center}

Here, $m_A^D=m_A\circ \iota_A$ is the right $D$-module structure on $A$, and $m_B^D=m_B\circ \iota_B$ is the left $D$-module structure on $B$. $A\otimes_D B$ acquires an algebra structure by the following lemma.

\begin{lemma}\cite[Lemma 4.5]{davydov2015functoriality}
    Let $\m{C}$ be a braided monoidal category. $A,B,D \in \m{C}$ are algebras. Assuming there exist algebra homomorphisms $\pi_A:D\to A$ and $\pi_B:D\to B$. Suppose that $\pi_B$ factors through $C_l(B)$ and that $\pi_A$ factors through $C_r(A)$. Then, if both A and B are separable, the relative tensor product $A \otimes_D B$ is also separable. 
\end{lemma}

\begin{proposition}\label{Por:all_fiber_product}
    Let $\m{C}$ be a braided fusion category, let $A$ and $B$ be condensable algebras in $\m{C}$, and let $D$ be a common condensable subalgebra of $A$ and $B$. Then $D = A \times_{A \otimes_D B} B$.
    \end{proposition}
\begin{proof}

We have the following two algebra homomorphisms
\begin{equation}\label{fAfB}
f_A:=(A\simeq A\otimes 1 \stackrel{id_A\otimes u_B}{\longrightarrow} A\otimes B \stackrel{\rho}{\longrightarrow} A\otimes_D B), \quad  f_B:=(B\simeq 1\otimes B \stackrel{u_A\otimes id_B}{\longrightarrow} A\otimes B \stackrel{\rho}{\longrightarrow} A\otimes_D B).
\end{equation}

We denote $A\otimes_D B$, the morphism $\rho: A\otimes B \rightarrow A\otimes_D B$, and the morphism $u_{A\otimes_D B}: 1\rightarrow A\otimes_D B$ by the following graphs:

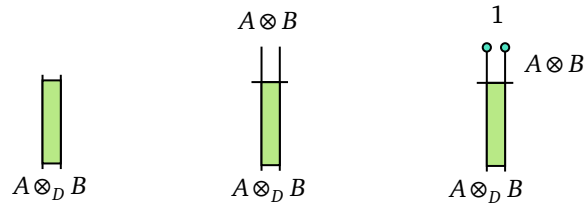
\begin{center}

\tikzset{every picture/.style={line width=0.75pt}} 

\begin{tikzpicture}[x=0.75pt,y=0.75pt,yscale=-0.8,xscale=0.8]

\draw    (177.79,56.22) -- (177.79,115.36) ;
\draw    (189.25,56.22) -- (189.25,115.36) ;
\draw  [fill={rgb, 255:red, 184; green, 233; blue, 134 }  ,fill opacity=1 ] (177.79,59.73) -- (189.25,59.73) -- (189.25,111.85) -- (177.79,111.85) -- cycle ;
\draw    (315.17,38.56) -- (315.17,116.44) ;
\draw    (326.63,38.56) -- (326.63,116.44) ;
\draw  [fill={rgb, 255:red, 184; green, 233; blue, 134 }  ,fill opacity=1 ] (315.17,60.81) -- (326.63,60.81) -- (326.63,112.92) -- (315.17,112.92) -- cycle ;
\draw    (308.87,60.81) -- (331.79,60.81) ;
\draw    (456.5,39.04) -- (456.5,116.92) ;
\draw    (467.97,39.04) -- (467.97,116.92) ;
\draw  [fill={rgb, 255:red, 184; green, 233; blue, 134 }  ,fill opacity=1 ] (456.5,61.29) -- (467.97,61.29) -- (467.97,113.41) -- (456.5,113.41) -- cycle ;
\draw    (450.2,61.29) -- (473.12,61.29) ;
\draw  [fill={rgb, 255:red, 80; green, 227; blue, 194 }  ,fill opacity=1 ] (453.93,39.04) .. controls (453.93,37.59) and (455.08,36.41) .. (456.5,36.41) .. controls (457.93,36.41) and (459.08,37.59) .. (459.08,39.04) .. controls (459.08,40.5) and (457.93,41.68) .. (456.5,41.68) .. controls (455.08,41.68) and (453.93,40.5) .. (453.93,39.04) -- cycle ;
\draw  [fill={rgb, 255:red, 80; green, 227; blue, 194 }  ,fill opacity=1 ] (465.39,39.04) .. controls (465.39,37.59) and (466.54,36.41) .. (467.97,36.41) .. controls (469.39,36.41) and (470.55,37.59) .. (470.55,39.04) .. controls (470.55,40.5) and (469.39,41.68) .. (467.97,41.68) .. controls (466.54,41.68) and (465.39,40.5) .. (465.39,39.04) -- cycle ;

\draw (158.26,118.66) node [anchor=north west][inner sep=0.75pt]   [align=left] {$\displaystyle A\otimes _{D} B$};
\draw (299.64,14.61) node [anchor=north west][inner sep=0.75pt]   [align=left] {$\displaystyle A\otimes B$};
\draw (295.64,119.73) node [anchor=north west][inner sep=0.75pt]   [align=left] {$\displaystyle A\otimes _{D} B$};
\draw (479.61,41.75) node [anchor=north west][inner sep=0.75pt]   [align=left] {$\displaystyle A\otimes B$};
\draw (431.98,120.22) node [anchor=north west][inner sep=0.75pt]   [align=left] {$\displaystyle A\otimes _{D} B$};
\draw (457.04,10.72) node [anchor=north west][inner sep=0.75pt]   [align=left] {$\displaystyle 1$};

\end{tikzpicture}

\end{center}

We have $f_A\circ \iota_A= f_B\circ \iota_B$:

\begin{center}

    \tikzset{every picture/.style={line width=0.75pt}} 

    
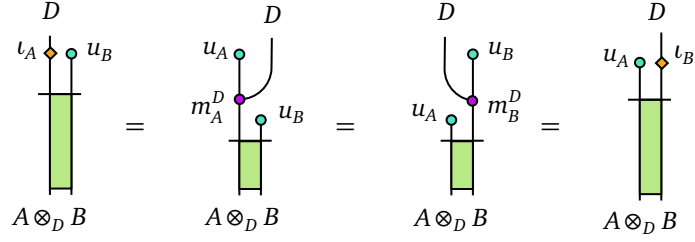
\begin{tikzpicture}[x=0.75pt,y=0.75pt,yscale=-0.8,xscale=0.8]
    
    \draw    (141.46,39.98) -- (141.46,139.85) ;
    \draw    (154.95,51.3) -- (154.95,139.85) ;
    \draw  [fill={rgb, 255:red, 184; green, 233; blue, 134 }  ,fill opacity=1 ] (141.46,76.6) -- (154.95,76.6) -- (154.95,135.85) -- (141.46,135.85) -- cycle ;
    \draw    (134.04,76.6) -- (161.02,76.6) ;
    \draw  [fill={rgb, 255:red, 80; green, 227; blue, 194 }  ,fill opacity=1 ] (151.91,51.3) .. controls (151.91,49.64) and (153.27,48.3) .. (154.95,48.3) .. controls (156.63,48.3) and (157.98,49.64) .. (157.98,51.3) .. controls (157.98,52.95) and (156.63,54.3) .. (154.95,54.3) .. controls (153.27,54.3) and (151.91,52.95) .. (151.91,51.3) -- cycle ;
    \draw  [fill={rgb, 255:red, 245; green, 166; blue, 35 }  ,fill opacity=1 ] (141.8,47.3) -- (145.51,50.97) -- (141.8,54.63) -- (138.09,50.97) -- cycle ;
    \draw    (393.31,92.86) -- (393.31,140.13) ;
    \draw    (406.8,51.58) -- (406.8,140.13) ;
    \draw  [fill={rgb, 255:red, 184; green, 233; blue, 134 }  ,fill opacity=1 ] (393.31,106.18) -- (406.8,106.18) -- (406.8,136.14) -- (393.31,136.14) -- cycle ;
    \draw    (386.57,105.85) -- (413.55,105.85) ;
    \draw  [fill={rgb, 255:red, 80; green, 227; blue, 194 }  ,fill opacity=1 ] (403.77,51.58) .. controls (403.77,49.93) and (405.13,48.59) .. (406.8,48.59) .. controls (408.48,48.59) and (409.84,49.93) .. (409.84,51.58) .. controls (409.84,53.24) and (408.48,54.58) .. (406.8,54.58) .. controls (405.13,54.58) and (403.77,53.24) .. (403.77,51.58) -- cycle ;
    \draw  [fill={rgb, 255:red, 80; green, 227; blue, 194 }  ,fill opacity=1 ] (390.28,92.86) .. controls (390.28,91.21) and (391.64,89.87) .. (393.31,89.87) .. controls (394.99,89.87) and (396.35,91.21) .. (396.35,92.86) .. controls (396.35,94.52) and (394.99,95.86) .. (393.31,95.86) .. controls (391.64,95.86) and (390.28,94.52) .. (390.28,92.86) -- cycle ;
    \draw  [draw opacity=0] (409.49,80.88) .. controls (409.42,80.88) and (409.35,80.88) .. (409.28,80.88) .. controls (398.1,80.88) and (389.04,71.94) .. (389.04,60.91) .. controls (389.04,60.68) and (389.05,60.46) .. (389.05,60.25) -- (409.28,60.91) -- cycle ; \draw   (409.49,80.88) .. controls (409.42,80.88) and (409.35,80.88) .. (409.28,80.88) .. controls (398.1,80.88) and (389.04,71.94) .. (389.04,60.91) .. controls (389.04,60.68) and (389.05,60.46) .. (389.05,60.25) ;  
    \draw    (389.27,40.71) -- (389.05,60.25) ;
    \draw  [fill={rgb, 255:red, 189; green, 16; blue, 224 }  ,fill opacity=1 ] (403.42,80.88) .. controls (403.42,79.22) and (404.78,77.88) .. (406.46,77.88) .. controls (408.14,77.88) and (409.49,79.22) .. (409.49,80.88) .. controls (409.49,82.53) and (408.14,83.87) .. (406.46,83.87) .. controls (404.78,83.87) and (403.42,82.53) .. (403.42,80.88) -- cycle ;
    \draw    (511.51,56.81) -- (511.51,143.37) ;
    \draw    (525,37.95) -- (525,143.37) ;
    \draw  [fill={rgb, 255:red, 184; green, 233; blue, 134 }  ,fill opacity=1 ] (511.51,80.12) -- (525,80.12) -- (525,139.37) -- (511.51,139.37) -- cycle ;
    \draw    (504.09,80.12) -- (531.07,80.12) ;
    \draw  [fill={rgb, 255:red, 80; green, 227; blue, 194 }  ,fill opacity=1 ] (508.47,56.81) .. controls (508.47,55.16) and (509.83,53.82) .. (511.51,53.82) .. controls (513.18,53.82) and (514.54,55.16) .. (514.54,56.81) .. controls (514.54,58.47) and (513.18,59.81) .. (511.51,59.81) .. controls (509.83,59.81) and (508.47,58.47) .. (508.47,56.81) -- cycle ;
    \draw  [fill={rgb, 255:red, 245; green, 166; blue, 35 }  ,fill opacity=1 ] (525.34,53.49) -- (529.05,57.15) -- (525.34,60.81) -- (521.63,57.15) -- cycle ;
    \draw    (260.33,51.67) -- (260.31,140.13) ;
    \draw    (273.8,92.86) -- (273.8,140.13) ;
    \draw  [fill={rgb, 255:red, 184; green, 233; blue, 134 }  ,fill opacity=1 ] (260.31,106.18) -- (273.8,106.18) -- (273.8,136.14) -- (260.31,136.14) -- cycle ;
    \draw    (253.57,105.85) -- (280.55,105.85) ;
    \draw  [fill={rgb, 255:red, 80; green, 227; blue, 194 }  ,fill opacity=1 ] (270.77,92.86) .. controls (270.77,91.21) and (272.13,89.87) .. (273.8,89.87) .. controls (275.48,89.87) and (276.84,91.21) .. (276.84,92.86) .. controls (276.84,94.52) and (275.48,95.86) .. (273.8,95.86) .. controls (272.13,95.86) and (270.77,94.52) .. (270.77,92.86) -- cycle ;
    \draw  [fill={rgb, 255:red, 80; green, 227; blue, 194 }  ,fill opacity=1 ] (257.3,51.67) .. controls (257.3,50.01) and (258.66,48.67) .. (260.33,48.67) .. controls (262.01,48.67) and (263.37,50.01) .. (263.37,51.67) .. controls (263.37,53.32) and (262.01,54.66) .. (260.33,54.66) .. controls (258.66,54.66) and (257.3,53.32) .. (257.3,51.67) -- cycle ;
    \draw  [draw opacity=0] (280.46,60.86) .. controls (279.95,71.41) and (271.16,79.82) .. (260.36,79.88) -- (260.24,59.91) -- cycle ; \draw   (280.46,60.86) .. controls (279.95,71.41) and (271.16,79.82) .. (260.36,79.88) ;  
    \draw    (280.67,41.32) -- (280.46,60.86) ;
    \draw  [fill={rgb, 255:red, 189; green, 16; blue, 224 }  ,fill opacity=1 ] (257.42,79.88) .. controls (257.42,78.22) and (258.78,76.88) .. (260.46,76.88) .. controls (262.14,76.88) and (263.49,78.22) .. (263.49,79.88) .. controls (263.49,81.53) and (262.14,82.87) .. (260.46,82.87) .. controls (258.78,82.87) and (257.42,81.53) .. (257.42,79.88) -- cycle ;
    
    \draw (162.93,42.3) node [anchor=north west][inner sep=0.75pt]   [align=left] {$\displaystyle u_{B}$};
    \draw (117.1,146.16) node [anchor=north west][inner sep=0.75pt]   [align=left] {$\displaystyle A\otimes _{D} B$};
    \draw (134.75,17) node [anchor=north west][inner sep=0.75pt]   [align=left] {$\displaystyle D$};
    \draw (118.91,41.63) node [anchor=north west][inner sep=0.75pt]   [align=left] {$\displaystyle \iota _{A}$};
    \draw (414.79,41.92) node [anchor=north west][inner sep=0.75pt]   [align=left] {$\displaystyle u_{B}$};
    \draw (368.95,146.45) node [anchor=north west][inner sep=0.75pt]   [align=left] {$\displaystyle A\otimes _{D} B$};
    \draw (382.6,17.28) node [anchor=north west][inner sep=0.75pt]   [align=left] {$\displaystyle D$};
    \draw (413.81,72.04) node [anchor=north west][inner sep=0.75pt]   [align=left] {$\displaystyle m_{B}^{D}$};
    \draw (486.96,46.48) node [anchor=north west][inner sep=0.75pt]   [align=left] {$\displaystyle u_{A}$};
    \draw (487.82,148.35) node [anchor=north west][inner sep=0.75pt]   [align=left] {$\displaystyle A\otimes _{D} B$};
    \draw (514.8,17.52) node [anchor=north west][inner sep=0.75pt]   [align=left] {$\displaystyle D$};
    \draw (531.1,44.49) node [anchor=north west][inner sep=0.75pt]   [align=left] {$\displaystyle \iota _{B}$};
    \draw (366.71,80.53) node [anchor=north west][inner sep=0.75pt]   [align=left] {$\displaystyle u_{A}$};
    \draw (447.35,92.71) node [anchor=north west][inner sep=0.75pt]   [align=left] {$\displaystyle =$};
    \draw (318.69,92.61) node [anchor=north west][inner sep=0.75pt]   [align=left] {$\displaystyle =$};
    \draw (283.27,81.71) node [anchor=north west][inner sep=0.75pt]   [align=left] {$\displaystyle u_{B}$};
    \draw (273.6,19.28) node [anchor=north west][inner sep=0.75pt]   [align=left] {$\displaystyle D$};
    \draw (227.81,73.04) node [anchor=north west][inner sep=0.75pt]   [align=left] {$\displaystyle m_{A}^{D}$};
    \draw (235.71,41.53) node [anchor=north west][inner sep=0.75pt]   [align=left] {$\displaystyle u_{A}$};
    \draw (238.1,146.33) node [anchor=north west][inner sep=0.75pt]   [align=left] {$\displaystyle A\otimes _{D} B$};
    \draw (186.69,92.61) node [anchor=north west][inner sep=0.75pt]   [align=left] {$\displaystyle =$};

    \end{tikzpicture}
    
\end{center}

The first and third equalities hold due to the definition $m_A^D$ and $m_B^D$. The second equality holds due to the definition of $A\otimes_D B$. 

Denote $E=A\times_{A\otimes_DB} B$ is the fiber product of the following diagram: 
\begin{center}
    \begin{tikzcd}
        E \arrow[d, "\pi_A"'] \arrow[r, "\pi_B"] & B \arrow[d, "f_B"] \\
        A \arrow[r, "f_A"]                       & A\otimes_D B .     
        \end{tikzcd}
\end{center}
By the universal property of a fiber product, there is a unique morphism $\iota_E:D\to E$. By Theorem~\ref{Thm:2d condense}, $E$ is a condensable algebra, and $\pi_A, \pi_B,$ and $\iota_E$ are all embeddings.
We denote $E=D\oplus D'$, $A=D\oplus D'\oplus A''$, and $B=D\oplus D'\oplus B''$.
If $D'\ncong 0$, then consider the images of $f_A\circ \pi_A$ and $f_B\circ \pi_B$ in $A\otimes_D B$.
We have
$
A\otimes_D B=(D\otimes_D D )  \oplus (D\otimes_D D' )  \oplus (D\otimes_D B'' )  \oplus (D'\otimes_D D )  \oplus (D'\otimes_D D' )  \oplus (D'\otimes_D B'' )  \oplus (A''\otimes_D D )  \oplus (A''\otimes_D D' )  \oplus (A''\otimes_D B'' )
$. The image of $f_A\circ \pi_A$ is in the first and fourth direct summands. The image of $f_B\circ \pi_B$ is in the first and second summands. This is a contradiction to the equality $f_A\circ \pi_A=f_B\circ \pi_B$. Thus, we have $D=E$. 
\end{proof}

\subsection{From common condensable subalgebras to critical cospans}\label{section4.bimod}

In this subsection, we show that given two condensable algebras $A$ and $B$ sharing a suitable common condensable subalgebra $D$, one can obtain a critical cospan whose limit is $D$.
To obtain the critical cospan associated with the boundary data of bulk $\mathcal{C}$, we need to show that $M = A \otimes_D B$ can be realized as the full centralizer of a 0+1D closed defect on the boundary. We show that $\mathcal{C}_M$ coincides with the category of closed $0+1$D defects between $\mathcal{C}_A$ and $\mathcal{C}_B$, assuming $M$ is indecomposable. Furthermore, we establish that $M \in \mathcal{C}$ is identified with the full centralizer of $M \in \mathcal{C}_M$.

The algebra $A\otimes_D B$ inherits various structures from $A$ and $B$. 

\begin{lemma}\label{lemma:separable}

    Let $\m{C}$ be a braided monoidal category, let $D,A,B \in \m{C}$ be algebras, and let $\pi_A : D \to A$ and $\pi_B : D \to B$ be algebra homomorphisms. Suppose that $\pi_B$ factors through $C_l(B)$ and that $\pi_A$ factors through $C_r(A)$. Then, if both $A$ and $B$ are separable, the relative tensor product $A \otimes_D B$ is also separable.

\end{lemma}
\begin{proof}
    For separable algebras $A$ and $B$, we have an $A$-$A$ bimodule map $\delta_A$ and a $B$-$B$ bimodule map $\delta_B$ such that $m_A\circ \delta_A=id_A$ and $m_B\circ \delta_B=id_B$. The multiplication on the algebra $A\otimes B$ is defined by $(m_A\otimes m_B)\circ (id_A\otimes \beta_{A,B}^{-1}\otimes id_B).$ Denote by $\delta_{A\otimes B}:= (id_A\otimes \beta_{A,B}\otimes id_B)\circ(\delta_A\otimes \delta_B)$. It can be checked that $\delta_{A\otimes B}$ is an $A\otimes B$-$A\otimes B$ bimodule map and $m_{A\otimes B}\circ \delta_{A\otimes B}=id_{A\otimes B}$. Thus, $A\otimes B$ is separable.

    What's more, we have $(\rho\otimes \rho)\circ \delta_{A\otimes B}\circ(m_A^D\otimes id)=(\rho\otimes \rho)\circ \delta_{A\otimes B}\circ(id\otimes m_B^D)$:

    \begin{center}

        \tikzset{every picture/.style={line width=0.75pt}} 

        \begin{tikzpicture}[x=0.75pt,y=0.75pt,yscale=-0.85,xscale=0.85]
        
        \draw    (550.79,62.08) .. controls (551.07,99.24) and (567.12,123.47) .. (557.83,137.2) ;
        \draw    (533.91,102.47) .. controls (542.63,102.47) and (560.36,125.89) .. (559.24,150.93) ;
        \draw    (407.66,103.01) .. controls (416.38,103.01) and (434.12,126.43) .. (432.99,151.47) ;
        \draw    (270.47,102.2) .. controls (279.2,102.2) and (296.93,125.62) .. (295.8,150.66) ;
        \draw    (145.98,140.54) -- (111.62,181.59) ;
        \draw    (115.06,50.24) -- (149.42,91.29) ;
        \draw    (270.47,53.74) -- (270.47,102.2) ;
        \draw    (287.36,53.74) -- (287.36,61.82) ;
        \draw    (304.25,53.74) -- (304.25,102.2) ;
        \draw    (287.36,61.82) .. controls (287.64,62.62) and (287.64,86.05) .. (270.47,86.05) ;
        \draw    (253.59,150.66) .. controls (253.03,143.39) and (257.25,103.01) .. (270.47,102.2) ;
        \draw    (304.25,102.2) .. controls (313.82,103.01) and (320.57,134.51) .. (321.13,150.66) ;
        \draw  [color={rgb, 255:red, 255; green, 255; blue, 255 }  ,draw opacity=1 ][line width=1.5] [line join = round][line cap = round] (288.49,119.97) .. controls (289.61,119.97) and (290.74,119.97) .. (291.86,119.97) ;
        \draw  [fill={rgb, 255:red, 184; green, 233; blue, 134 }  ,fill opacity=1 ] (253.59,150.66) -- (278.92,150.66) -- (278.92,166.82) -- (253.59,166.82) -- cycle ;
        \draw  [fill={rgb, 255:red, 184; green, 233; blue, 134 }  ,fill opacity=1 ] (295.8,150.66) -- (321.13,150.66) -- (321.13,166.82) -- (295.8,166.82) -- cycle ;
        \draw  [fill={rgb, 255:red, 189; green, 16; blue, 224 }  ,fill opacity=1 ] (266.25,86.05) .. controls (266.25,83.82) and (268.14,82.01) .. (270.47,82.01) .. controls (272.81,82.01) and (274.7,83.82) .. (274.7,86.05) .. controls (274.7,88.28) and (272.81,90.09) .. (270.47,90.09) .. controls (268.14,90.09) and (266.25,88.28) .. (266.25,86.05) -- cycle ;
        \draw  [color={rgb, 255:red, 255; green, 255; blue, 255 }  ,draw opacity=1 ][line width=1.5] [line join = round][line cap = round] (97.31,51.06) .. controls (97.31,51.06) and (97.31,51.06) .. (97.31,51.06) ;
        \draw  [color={rgb, 255:red, 255; green, 255; blue, 255 }  ,draw opacity=1 ][fill={rgb, 255:red, 255; green, 255; blue, 255 }  ,fill opacity=1 ] (285.67,116.74) -- (290.18,116.74) -- (290.18,121.05) -- (285.67,121.05) -- cycle ;
        \draw  [fill={rgb, 255:red, 255; green, 98; blue, 118 }  ,fill opacity=1 ] (266.25,102.2) .. controls (266.25,99.97) and (268.14,98.16) .. (270.47,98.16) .. controls (272.81,98.16) and (274.7,99.97) .. (274.7,102.2) .. controls (274.7,104.43) and (272.81,106.24) .. (270.47,106.24) .. controls (268.14,106.24) and (266.25,104.43) .. (266.25,102.2) -- cycle ;
        \draw  [color={rgb, 255:red, 255; green, 255; blue, 255 }  ,draw opacity=1 ][fill={rgb, 255:red, 255; green, 255; blue, 255 }  ,fill opacity=1 ] (130.09,68.71) -- (134.38,68.71) -- (134.38,72.82) -- (130.09,72.82) -- cycle ;
        \draw  [color={rgb, 255:red, 255; green, 255; blue, 255 }  ,draw opacity=1 ][fill={rgb, 255:red, 255; green, 255; blue, 255 }  ,fill opacity=1 ] (126.65,159.01) -- (130.95,159.01) -- (130.95,163.12) -- (126.65,163.12) -- cycle ;
        \draw    (407.66,54.55) -- (407.66,103.01) ;
        \draw    (424.55,54.55) -- (424.55,62.62) ;
        \draw    (441.43,54.55) -- (441.43,103.01) ;
        \draw    (424.55,62.62) .. controls (424.83,63.43) and (425.67,93.59) .. (419.48,110.28) ;
        \draw    (390.77,151.47) .. controls (390.21,144.2) and (394.43,103.82) .. (407.66,103.01) ;
        \draw    (441.43,103.01) .. controls (451,103.82) and (457.76,135.32) .. (458.32,151.47) ;
        \draw  [color={rgb, 255:red, 255; green, 255; blue, 255 }  ,draw opacity=1 ][line width=1.5] [line join = round][line cap = round] (425.67,120.78) .. controls (426.8,120.78) and (427.92,120.78) .. (429.05,120.78) ;
        \draw  [fill={rgb, 255:red, 184; green, 233; blue, 134 }  ,fill opacity=1 ] (390.77,151.47) -- (416.1,151.47) -- (416.1,167.63) -- (390.77,167.63) -- cycle ;
        \draw  [fill={rgb, 255:red, 184; green, 233; blue, 134 }  ,fill opacity=1 ] (432.99,151.47) -- (458.32,151.47) -- (458.32,167.63) -- (432.99,167.63) -- cycle ;
        \draw  [color={rgb, 255:red, 255; green, 255; blue, 255 }  ,draw opacity=1 ][fill={rgb, 255:red, 255; green, 255; blue, 255 }  ,fill opacity=1 ] (422.86,117.55) -- (427.36,117.55) -- (427.36,121.86) -- (422.86,121.86) -- cycle ;
        \draw  [fill={rgb, 255:red, 255; green, 98; blue, 118 }  ,fill opacity=1 ] (403.44,103.01) .. controls (403.44,100.78) and (405.33,98.97) .. (407.66,98.97) .. controls (409.99,98.97) and (411.88,100.78) .. (411.88,103.01) .. controls (411.88,105.24) and (409.99,107.05) .. (407.66,107.05) .. controls (405.33,107.05) and (403.44,105.24) .. (403.44,103.01) -- cycle ;
        \draw  [fill={rgb, 255:red, 189; green, 16; blue, 224 }  ,fill opacity=1 ] (415.26,110.28) .. controls (415.26,108.05) and (417.15,106.24) .. (419.48,106.24) .. controls (421.81,106.24) and (423.7,108.05) .. (423.7,110.28) .. controls (423.7,112.51) and (421.81,114.32) .. (419.48,114.32) .. controls (417.15,114.32) and (415.26,112.51) .. (415.26,110.28) -- cycle ;
        \draw    (533.91,54.01) -- (533.91,102.47) ;
        \draw    (550.79,54.01) -- (550.79,62.08) ;
        \draw    (567.68,54.01) -- (567.68,102.47) ;
        \draw    (517.02,150.93) .. controls (516.46,143.66) and (520.68,103.28) .. (533.91,102.47) ;
        \draw    (567.68,102.47) .. controls (577.25,103.28) and (584,134.78) .. (584.57,150.93) ;
        \draw  [color={rgb, 255:red, 255; green, 255; blue, 255 }  ,draw opacity=1 ][line width=1.5] [line join = round][line cap = round] (551.92,120.24) .. controls (553.04,120.24) and (554.17,120.24) .. (555.3,120.24) ;
        \draw  [fill={rgb, 255:red, 184; green, 233; blue, 134 }  ,fill opacity=1 ] (517.02,150.93) -- (542.35,150.93) -- (542.35,167.09) -- (517.02,167.09) -- cycle ;
        \draw  [fill={rgb, 255:red, 184; green, 233; blue, 134 }  ,fill opacity=1 ] (559.24,150.93) -- (584.57,150.93) -- (584.57,167.09) -- (559.24,167.09) -- cycle ;
        \draw  [color={rgb, 255:red, 255; green, 255; blue, 255 }  ,draw opacity=1 ][fill={rgb, 255:red, 255; green, 255; blue, 255 }  ,fill opacity=1 ] (549.1,117.01) -- (553.61,117.01) -- (553.61,121.32) -- (549.1,121.32) -- cycle ;
        \draw  [fill={rgb, 255:red, 255; green, 98; blue, 118 }  ,fill opacity=1 ] (529.68,102.47) .. controls (529.68,100.24) and (531.57,98.43) .. (533.91,98.43) .. controls (536.24,98.43) and (538.13,100.24) .. (538.13,102.47) .. controls (538.13,104.7) and (536.24,106.51) .. (533.91,106.51) .. controls (531.57,106.51) and (529.68,104.7) .. (529.68,102.47) -- cycle ;
        \draw  [color={rgb, 255:red, 255; green, 255; blue, 255 }  ,draw opacity=1 ][fill={rgb, 255:red, 255; green, 255; blue, 255 }  ,fill opacity=1 ] (556.14,106.51) -- (559.52,106.51) -- (559.52,109.74) -- (556.14,109.74) -- cycle ;
        \draw  [fill={rgb, 255:red, 189; green, 16; blue, 224 }  ,fill opacity=1 ] (553.61,137.2) .. controls (553.61,134.97) and (555.5,133.16) .. (557.83,133.16) .. controls (560.16,133.16) and (562.05,134.97) .. (562.05,137.2) .. controls (562.05,139.43) and (560.16,141.24) .. (557.83,141.24) .. controls (555.5,141.24) and (553.61,139.43) .. (553.61,137.2) -- cycle ;
        \draw    (149.42,50.24) -- (115.06,91.29) ;
        \draw    (111.62,140.54) -- (145.98,181.59) ;
        \draw    (416.1,151.47) .. controls (415.54,127.24) and (433.27,103.01) .. (441.43,103.01) ;
        \draw    (278.92,150.66) .. controls (278.36,126.43) and (296.09,102.2) .. (304.25,102.2) ;
        \draw  [fill={rgb, 255:red, 255; green, 98; blue, 118 }  ,fill opacity=1 ] (300.03,102.2) .. controls (300.03,99.97) and (301.92,98.16) .. (304.25,98.16) .. controls (306.58,98.16) and (308.47,99.97) .. (308.47,102.2) .. controls (308.47,104.43) and (306.58,106.24) .. (304.25,106.24) .. controls (301.92,106.24) and (300.03,104.43) .. (300.03,102.2) -- cycle ;
        \draw  [fill={rgb, 255:red, 255; green, 98; blue, 118 }  ,fill opacity=1 ] (437.21,103.01) .. controls (437.21,100.78) and (439.1,98.97) .. (441.43,98.97) .. controls (443.76,98.97) and (445.65,100.78) .. (445.65,103.01) .. controls (445.65,105.24) and (443.76,107.05) .. (441.43,107.05) .. controls (439.1,107.05) and (437.21,105.24) .. (437.21,103.01) -- cycle ;
        \draw    (542.35,150.93) .. controls (541.79,126.7) and (559.52,102.47) .. (567.68,102.47) ;
        \draw  [fill={rgb, 255:red, 255; green, 98; blue, 118 }  ,fill opacity=1 ] (563.46,102.47) .. controls (563.46,100.24) and (565.35,98.43) .. (567.68,98.43) .. controls (570.01,98.43) and (571.9,100.24) .. (571.9,102.47) .. controls (571.9,104.7) and (570.01,106.51) .. (567.68,106.51) .. controls (565.35,106.51) and (563.46,104.7) .. (563.46,102.47) -- cycle ;
        
        \draw (256.72,35.85) node [anchor=north west][inner sep=0.75pt]   [align=left] {$\displaystyle A$};
        \draw (281.13,35.85) node [anchor=north west][inner sep=0.75pt]   [align=left] {$\displaystyle D$};
        \draw (304.16,35.85) node [anchor=north west][inner sep=0.75pt]   [align=left] {$\displaystyle B$};
        \draw (222.89,176.51) node [anchor=north west][inner sep=0.75pt]   [align=left] {$\displaystyle A\otimes _{D} B$};
        \draw (295.37,177.51) node [anchor=north west][inner sep=0.75pt]   [align=left] {$\displaystyle A\otimes _{D} B$};
        \draw (235.99,97.01) node [anchor=north west][inner sep=0.75pt]   [align=left] {$\displaystyle \delta _{A}$};
        \draw (318.06,89.01) node [anchor=north west][inner sep=0.75pt]   [align=left] {$\displaystyle \delta _{B}$};
        \draw (235.21,68.45) node [anchor=north west][inner sep=0.75pt]   [align=left] {$\displaystyle m_{A}^{D}$};
        \draw (52.12,58.48) node [anchor=north west][inner sep=0.75pt]   [align=left] {$\displaystyle \beta _{A,B} =$};
        \draw (52.12,148.51) node [anchor=north west][inner sep=0.75pt]   [align=left] {$\displaystyle \beta _{B,A}^{-1} =$};
        \draw (101.19,33.85) node [anchor=north west][inner sep=0.75pt]   [align=left] {$\displaystyle A$};
        \draw (149.43,33.85) node [anchor=north west][inner sep=0.75pt]   [align=left] {$\displaystyle B$};
        \draw (97.75,122.51) node [anchor=north west][inner sep=0.75pt]   [align=left] {$\displaystyle A$};
        \draw (145.99,122.51) node [anchor=north west][inner sep=0.75pt]   [align=left] {$\displaystyle B$};
        \draw (348.67,101.66) node [anchor=north west][inner sep=0.75pt]   [align=left] {$\displaystyle =$};
        \draw (393.91,36.66) node [anchor=north west][inner sep=0.75pt]   [align=left] {$\displaystyle A$};
        \draw (418.31,36.66) node [anchor=north west][inner sep=0.75pt]   [align=left] {$\displaystyle D$};
        \draw (441.34,36.66) node [anchor=north west][inner sep=0.75pt]   [align=left] {$\displaystyle B$};
        \draw (475.55,98.09) node [anchor=north west][inner sep=0.75pt]   [align=left] {$\displaystyle =$};
        \draw (520.15,36.12) node [anchor=north west][inner sep=0.75pt]   [align=left] {$\displaystyle A$};
        \draw (544.56,36.12) node [anchor=north west][inner sep=0.75pt]   [align=left] {$\displaystyle D$};
        \draw (567.59,36.12) node [anchor=north west][inner sep=0.75pt]   [align=left] {$\displaystyle B$};
        \draw (607.55,98.43) node [anchor=north west][inner sep=0.75pt]   [align=left] {$\displaystyle =$};

        \end{tikzpicture}

    \end{center}

    \begin{center}

        \tikzset{every picture/.style={line width=0.75pt}} 

        \begin{tikzpicture}[x=0.75pt,y=0.75pt,yscale=-0.85,xscale=0.85]
        
        \draw    (589.43,91.12) .. controls (598.2,91.12) and (616.01,115.39) .. (614.88,141.33) ;
        \draw    (466.43,91.12) .. controls (475.2,91.12) and (493.01,115.39) .. (491.88,141.33) ;
        \draw    (328.84,90.12) .. controls (337.61,90.12) and (355.42,114.39) .. (354.29,140.33) ;
        \draw    (500.43,59.24) .. controls (512.51,62.1) and (509.12,79.68) .. (500.22,77.86) ;
        \draw    (483.39,49.27) .. controls (483.68,55.41) and (490.46,57.92) .. (500.43,59.24) ;
        \draw    (345.81,48.27) .. controls (346.09,65.57) and (346.09,65.57) .. (365.6,73.1) ;
        \draw    (365.6,73.1) .. controls (389.34,80.63) and (392.74,122.2) .. (378.32,130.29) ;
        \draw    (213.74,48.27) .. controls (214.02,77.56) and (208.08,95.97) .. (234.38,101.83) ;
        \draw    (196.77,90.12) .. controls (205.54,90.12) and (223.35,114.39) .. (222.22,140.33) ;
        \draw    (234.38,101.83) .. controls (253.88,106.02) and (260.67,122.2) .. (246.25,130.29) ;
        \draw    (58.44,89.28) .. controls (67.2,89.28) and (85.01,113.55) .. (83.88,139.5) ;
        \draw    (75.4,47.43) .. controls (75.68,85.93) and (86.71,114.95) .. (107.07,122.76) ;
        \draw    (58.44,39.06) -- (58.44,89.28) ;
        \draw    (75.4,39.06) -- (75.4,47.43) ;
        \draw    (92.36,39.06) -- (92.36,89.28) ;
        \draw    (41.48,139.5) .. controls (40.91,131.96) and (45.15,90.12) .. (58.44,89.28) ;
        \draw    (92.36,89.28) .. controls (101.98,90.12) and (108.76,122.76) .. (109.33,139.5) ;
        \draw  [color={rgb, 255:red, 255; green, 255; blue, 255 }  ,draw opacity=1 ][line width=1.5] [line join = round][line cap = round] (76.53,107.69) .. controls (77.66,107.69) and (78.79,107.69) .. (79.93,107.69) ;
        \draw  [fill={rgb, 255:red, 184; green, 233; blue, 134 }  ,fill opacity=1 ] (41.48,139.5) -- (66.92,139.5) -- (66.92,156.24) -- (41.48,156.24) -- cycle ;
        \draw  [fill={rgb, 255:red, 184; green, 233; blue, 134 }  ,fill opacity=1 ] (83.88,139.5) -- (109.33,139.5) -- (109.33,156.24) -- (83.88,156.24) -- cycle ;
        \draw  [color={rgb, 255:red, 255; green, 255; blue, 255 }  ,draw opacity=1 ][fill={rgb, 255:red, 255; green, 255; blue, 255 }  ,fill opacity=1 ] (73.71,104.34) -- (78.23,104.34) -- (78.23,108.81) -- (73.71,108.81) -- cycle ;
        \draw  [fill={rgb, 255:red, 255; green, 98; blue, 118 }  ,fill opacity=1 ] (54.2,89.28) .. controls (54.2,86.97) and (56.1,85.09) .. (58.44,85.09) .. controls (60.78,85.09) and (62.68,86.97) .. (62.68,89.28) .. controls (62.68,91.59) and (60.78,93.46) .. (58.44,93.46) .. controls (56.1,93.46) and (54.2,91.59) .. (54.2,89.28) -- cycle ;
        \draw  [color={rgb, 255:red, 255; green, 255; blue, 255 }  ,draw opacity=1 ][fill={rgb, 255:red, 255; green, 255; blue, 255 }  ,fill opacity=1 ] (80.77,93.46) -- (84.17,93.46) -- (84.17,96.81) -- (80.77,96.81) -- cycle ;
        \draw  [fill={rgb, 255:red, 189; green, 16; blue, 224 }  ,fill opacity=1 ] (102.82,122.76) .. controls (102.82,120.45) and (104.72,118.57) .. (107.07,118.57) .. controls (109.41,118.57) and (111.31,120.45) .. (111.31,122.76) .. controls (111.31,125.07) and (109.41,126.94) .. (107.07,126.94) .. controls (104.72,126.94) and (102.82,125.07) .. (102.82,122.76) -- cycle ;
        \draw    (196.77,39.9) -- (196.77,90.12) ;
        \draw    (213.74,39.9) -- (213.74,48.27) ;
        \draw    (230.7,39.9) -- (230.7,90.12) ;
        \draw    (179.81,140.33) .. controls (179.25,132.8) and (183.49,90.95) .. (196.77,90.12) ;
        \draw  [color={rgb, 255:red, 255; green, 255; blue, 255 }  ,draw opacity=1 ][line width=1.5] [line join = round][line cap = round] (214.87,108.53) .. controls (216,108.53) and (217.13,108.53) .. (218.26,108.53) ;
        \draw  [fill={rgb, 255:red, 184; green, 233; blue, 134 }  ,fill opacity=1 ] (179.81,140.33) -- (205.26,140.33) -- (205.26,157.07) -- (179.81,157.07) -- cycle ;
        \draw  [fill={rgb, 255:red, 184; green, 233; blue, 134 }  ,fill opacity=1 ] (222.22,140.33) -- (247.66,140.33) -- (247.66,157.07) -- (222.22,157.07) -- cycle ;
        \draw  [color={rgb, 255:red, 255; green, 255; blue, 255 }  ,draw opacity=1 ][fill={rgb, 255:red, 255; green, 255; blue, 255 }  ,fill opacity=1 ] (212.04,105.18) -- (215.72,105.18) -- (215.72,108.81) -- (212.04,108.81) -- cycle ;
        \draw  [fill={rgb, 255:red, 255; green, 98; blue, 118 }  ,fill opacity=1 ] (192.53,90.12) .. controls (192.53,87.8) and (194.43,85.93) .. (196.77,85.93) .. controls (199.12,85.93) and (201.02,87.8) .. (201.02,90.12) .. controls (201.02,92.43) and (199.12,94.3) .. (196.77,94.3) .. controls (194.43,94.3) and (192.53,92.43) .. (192.53,90.12) -- cycle ;
        \draw  [color={rgb, 255:red, 255; green, 255; blue, 255 }  ,draw opacity=1 ][fill={rgb, 255:red, 255; green, 255; blue, 255 }  ,fill opacity=1 ] (219.11,94.3) -- (222.5,94.3) -- (222.5,97.65) -- (219.11,97.65) -- cycle ;
        \draw  [color={rgb, 255:red, 255; green, 255; blue, 255 }  ,draw opacity=1 ][fill={rgb, 255:red, 255; green, 255; blue, 255 }  ,fill opacity=1 ] (238.62,106.02) -- (243.7,106.02) -- (243.7,101) -- (238.62,101) -- cycle ;
        \draw    (328.84,39.9) -- (328.84,90.12) ;
        \draw    (345.81,39.9) -- (345.81,48.27) ;
        \draw    (311.88,140.33) .. controls (311.32,132.8) and (315.56,90.95) .. (328.84,90.12) ;
        \draw    (362.77,90.12) .. controls (372.38,90.95) and (379.17,123.59) .. (379.73,140.33) ;
        \draw  [color={rgb, 255:red, 255; green, 255; blue, 255 }  ,draw opacity=1 ][line width=1.5] [line join = round][line cap = round] (346.94,108.53) .. controls (348.07,108.53) and (349.2,108.53) .. (350.33,108.53) ;
        \draw  [fill={rgb, 255:red, 184; green, 233; blue, 134 }  ,fill opacity=1 ] (311.88,140.33) -- (337.33,140.33) -- (337.33,157.07) -- (311.88,157.07) -- cycle ;
        \draw  [fill={rgb, 255:red, 184; green, 233; blue, 134 }  ,fill opacity=1 ] (354.29,140.33) -- (379.73,140.33) -- (379.73,157.07) -- (354.29,157.07) -- cycle ;
        \draw  [color={rgb, 255:red, 255; green, 255; blue, 255 }  ,draw opacity=1 ][fill={rgb, 255:red, 255; green, 255; blue, 255 }  ,fill opacity=1 ] (344.11,105.18) -- (347.79,105.18) -- (347.79,108.81) -- (344.11,108.81) -- cycle ;
        \draw  [fill={rgb, 255:red, 255; green, 98; blue, 118 }  ,fill opacity=1 ] (324.6,90.12) .. controls (324.6,87.8) and (326.5,85.93) .. (328.84,85.93) .. controls (331.19,85.93) and (333.08,87.8) .. (333.08,90.12) .. controls (333.08,92.43) and (331.19,94.3) .. (328.84,94.3) .. controls (326.5,94.3) and (324.6,92.43) .. (324.6,90.12) -- cycle ;
        \draw  [color={rgb, 255:red, 255; green, 255; blue, 255 }  ,draw opacity=1 ][fill={rgb, 255:red, 255; green, 255; blue, 255 }  ,fill opacity=1 ] (365.33,75.67) -- (357.33,75.67) -- (357.33,67.67) -- (365.33,67.67) -- cycle ;
        \draw  [fill={rgb, 255:red, 189; green, 16; blue, 224 }  ,fill opacity=1 ] (374.08,130.29) .. controls (374.08,127.98) and (375.98,126.11) .. (378.32,126.11) .. controls (380.66,126.11) and (382.56,127.98) .. (382.56,130.29) .. controls (382.56,132.6) and (380.66,134.47) .. (378.32,134.47) .. controls (375.98,134.47) and (374.08,132.6) .. (374.08,130.29) -- cycle ;
        \draw    (466.43,40.9) -- (466.43,91.12) ;
        \draw    (483.39,40.9) -- (483.39,49.27) ;
        \draw    (449.47,141.33) .. controls (448.9,133.8) and (453.14,91.95) .. (466.43,91.12) ;
        \draw    (500.36,91.12) .. controls (509.97,91.95) and (516.75,124.59) .. (517.32,141.33) ;
        \draw  [color={rgb, 255:red, 255; green, 255; blue, 255 }  ,draw opacity=1 ][line width=1.5] [line join = round][line cap = round] (484.53,109.53) .. controls (485.66,109.53) and (486.79,109.53) .. (487.92,109.53) ;
        \draw  [fill={rgb, 255:red, 184; green, 233; blue, 134 }  ,fill opacity=1 ] (449.47,141.33) -- (474.91,141.33) -- (474.91,158.07) -- (449.47,158.07) -- cycle ;
        \draw  [fill={rgb, 255:red, 184; green, 233; blue, 134 }  ,fill opacity=1 ] (491.88,141.33) -- (517.32,141.33) -- (517.32,158.07) -- (491.88,158.07) -- cycle ;
        \draw  [color={rgb, 255:red, 255; green, 255; blue, 255 }  ,draw opacity=1 ][fill={rgb, 255:red, 255; green, 255; blue, 255 }  ,fill opacity=1 ] (481.7,106.18) -- (485.37,106.18) -- (485.37,109.81) -- (481.7,109.81) -- cycle ;
        \draw  [fill={rgb, 255:red, 255; green, 98; blue, 118 }  ,fill opacity=1 ] (462.19,91.12) .. controls (462.19,88.8) and (464.09,86.93) .. (466.43,86.93) .. controls (468.77,86.93) and (470.67,88.8) .. (470.67,91.12) .. controls (470.67,93.43) and (468.77,95.3) .. (466.43,95.3) .. controls (464.09,95.3) and (462.19,93.43) .. (462.19,91.12) -- cycle ;
        \draw  [color={rgb, 255:red, 255; green, 255; blue, 255 }  ,draw opacity=1 ][fill={rgb, 255:red, 255; green, 255; blue, 255 }  ,fill opacity=1 ] (496.8,62.92) -- (504.05,62.92) -- (504.05,55.56) -- (496.8,55.56) -- cycle ;
        \draw    (589.43,40.9) -- (589.43,91.12) ;
        \draw    (606.39,40.9) -- (606.39,49.27) ;
        \draw    (623.36,40.9) -- (623.36,91.12) ;
        \draw    (572.47,141.33) .. controls (571.9,133.8) and (576.14,91.95) .. (589.43,91.12) ;
        \draw    (623.36,91.12) .. controls (632.97,91.95) and (639.75,124.59) .. (640.32,141.33) ;
        \draw  [color={rgb, 255:red, 255; green, 255; blue, 255 }  ,draw opacity=1 ][line width=1.5] [line join = round][line cap = round] (607.53,109.53) .. controls (608.66,109.53) and (609.79,109.53) .. (610.92,109.53) ;
        \draw  [fill={rgb, 255:red, 184; green, 233; blue, 134 }  ,fill opacity=1 ] (572.47,141.33) -- (597.91,141.33) -- (597.91,158.07) -- (572.47,158.07) -- cycle ;
        \draw  [fill={rgb, 255:red, 184; green, 233; blue, 134 }  ,fill opacity=1 ] (614.88,141.33) -- (640.32,141.33) -- (640.32,158.07) -- (614.88,158.07) -- cycle ;
        \draw  [color={rgb, 255:red, 255; green, 255; blue, 255 }  ,draw opacity=1 ][fill={rgb, 255:red, 255; green, 255; blue, 255 }  ,fill opacity=1 ] (604.7,106.18) -- (608.37,106.18) -- (608.37,109.81) -- (604.7,109.81) -- cycle ;
        \draw  [fill={rgb, 255:red, 255; green, 98; blue, 118 }  ,fill opacity=1 ] (585.19,91.12) .. controls (585.19,88.8) and (587.09,86.93) .. (589.43,86.93) .. controls (591.77,86.93) and (593.67,88.8) .. (593.67,91.12) .. controls (593.67,93.43) and (591.77,95.3) .. (589.43,95.3) .. controls (587.09,95.3) and (585.19,93.43) .. (585.19,91.12) -- cycle ;
        \draw    (606.39,49.27) .. controls (607.33,72.33) and (608.33,72.33) .. (623.22,77.86) ;
        \draw  [fill={rgb, 255:red, 189; green, 16; blue, 224 }  ,fill opacity=1 ] (618.97,77.86) .. controls (618.97,75.55) and (620.87,73.68) .. (623.22,73.68) .. controls (625.56,73.68) and (627.46,75.55) .. (627.46,77.86) .. controls (627.46,80.18) and (625.56,82.05) .. (623.22,82.05) .. controls (620.87,82.05) and (618.97,80.18) .. (618.97,77.86) -- cycle ;
        \draw    (66.92,139.5) .. controls (66.35,114.39) and (84.17,89.28) .. (92.36,89.28) ;
        \draw  [fill={rgb, 255:red, 255; green, 98; blue, 118 }  ,fill opacity=1 ] (88.12,89.28) .. controls (88.12,86.97) and (90.02,85.09) .. (92.36,85.09) .. controls (94.71,85.09) and (96.61,86.97) .. (96.61,89.28) .. controls (96.61,91.59) and (94.71,93.46) .. (92.36,93.46) .. controls (90.02,93.46) and (88.12,91.59) .. (88.12,89.28) -- cycle ;
        \draw    (230.7,90.12) .. controls (240.31,90.95) and (247.1,123.59) .. (247.66,140.33) ;
        \draw    (205.26,140.33) .. controls (204.69,115.22) and (222.5,90.12) .. (230.7,90.12) ;
        \draw  [fill={rgb, 255:red, 255; green, 98; blue, 118 }  ,fill opacity=1 ] (226.46,90.12) .. controls (226.46,87.8) and (228.36,85.93) .. (230.7,85.93) .. controls (233.04,85.93) and (234.94,87.8) .. (234.94,90.12) .. controls (234.94,92.43) and (233.04,94.3) .. (230.7,94.3) .. controls (228.36,94.3) and (226.46,92.43) .. (226.46,90.12) -- cycle ;
        \draw  [fill={rgb, 255:red, 189; green, 16; blue, 224 }  ,fill opacity=1 ] (242.01,130.29) .. controls (242.01,127.98) and (243.91,126.11) .. (246.25,126.11) .. controls (248.59,126.11) and (250.49,127.98) .. (250.49,130.29) .. controls (250.49,132.6) and (248.59,134.47) .. (246.25,134.47) .. controls (243.91,134.47) and (242.01,132.6) .. (242.01,130.29) -- cycle ;
        \draw    (362.77,39.9) -- (362.77,74.33) -- (362.77,90.12) ;
        \draw    (500.36,40.9) -- (500.36,55.33) -- (500.36,91.12) ;
        \draw  [fill={rgb, 255:red, 189; green, 16; blue, 224 }  ,fill opacity=1 ] (495.97,77.86) .. controls (495.97,75.55) and (497.87,73.68) .. (500.22,73.68) .. controls (502.56,73.68) and (504.46,75.55) .. (504.46,77.86) .. controls (504.46,80.18) and (502.56,82.05) .. (500.22,82.05) .. controls (497.87,82.05) and (495.97,80.18) .. (495.97,77.86) -- cycle ;
        \draw    (337.33,140.33) .. controls (336.76,115.22) and (354.57,90.12) .. (362.77,90.12) ;
        \draw  [fill={rgb, 255:red, 255; green, 98; blue, 118 }  ,fill opacity=1 ] (358.53,90.12) .. controls (358.53,87.8) and (360.43,85.93) .. (362.77,85.93) .. controls (365.11,85.93) and (367.01,87.8) .. (367.01,90.12) .. controls (367.01,92.43) and (365.11,94.3) .. (362.77,94.3) .. controls (360.43,94.3) and (358.53,92.43) .. (358.53,90.12) -- cycle ;
        \draw    (474.91,141.33) .. controls (474.35,116.22) and (492.16,91.12) .. (500.36,91.12) ;
        \draw  [fill={rgb, 255:red, 255; green, 98; blue, 118 }  ,fill opacity=1 ] (496.12,91.12) .. controls (496.12,88.8) and (498.01,86.93) .. (500.36,86.93) .. controls (502.7,86.93) and (504.6,88.8) .. (504.6,91.12) .. controls (504.6,93.43) and (502.7,95.3) .. (500.36,95.3) .. controls (498.01,95.3) and (496.12,93.43) .. (496.12,91.12) -- cycle ;
        \draw    (597.91,141.33) .. controls (597.35,116.22) and (615.16,91.12) .. (623.36,91.12) ;
        \draw  [fill={rgb, 255:red, 255; green, 98; blue, 118 }  ,fill opacity=1 ] (619.12,91.12) .. controls (619.12,88.8) and (621.01,86.93) .. (623.36,86.93) .. controls (625.7,86.93) and (627.6,88.8) .. (627.6,91.12) .. controls (627.6,93.43) and (625.7,95.3) .. (623.36,95.3) .. controls (621.01,95.3) and (619.12,93.43) .. (619.12,91.12) -- cycle ;
        
        \draw (44.65,20.86) node [anchor=north west][inner sep=0.75pt]   [align=left] {$\displaystyle A$};
        \draw (69.17,20.86) node [anchor=north west][inner sep=0.75pt]   [align=left] {$\displaystyle D$};
        \draw (92.3,20.86) node [anchor=north west][inner sep=0.75pt]   [align=left] {$\displaystyle B$};
        \draw (114.99,104.59) node [anchor=north west][inner sep=0.75pt]   [align=left] {$\displaystyle m_{B}^{D}$};
        \draw (182.99,21.69) node [anchor=north west][inner sep=0.75pt]   [align=left] {$\displaystyle A$};
        \draw (207.51,21.69) node [anchor=north west][inner sep=0.75pt]   [align=left] {$\displaystyle D$};
        \draw (230.64,21.69) node [anchor=north west][inner sep=0.75pt]   [align=left] {$\displaystyle B$};
        \draw (140.2,82.23) node [anchor=north west][inner sep=0.75pt]   [align=left] {$\displaystyle =$};
        \draw (270.06,82.23) node [anchor=north west][inner sep=0.75pt]   [align=left] {$\displaystyle =$};
        \draw (315.06,21.69) node [anchor=north west][inner sep=0.75pt]   [align=left] {$\displaystyle A$};
        \draw (339.58,21.69) node [anchor=north west][inner sep=0.75pt]   [align=left] {$\displaystyle D$};
        \draw (362.71,21.69) node [anchor=north west][inner sep=0.75pt]   [align=left] {$\displaystyle B$};
        \draw (452.65,22.69) node [anchor=north west][inner sep=0.75pt]   [align=left] {$\displaystyle A$};
        \draw (477.17,22.69) node [anchor=north west][inner sep=0.75pt]   [align=left] {$\displaystyle D$};
        \draw (500.29,22.69) node [anchor=north west][inner sep=0.75pt]   [align=left] {$\displaystyle B$};
        \draw (404,84) node [anchor=north west][inner sep=0.75pt]   [align=left] {$\displaystyle =$};
        \draw (575.65,22.69) node [anchor=north west][inner sep=0.75pt]   [align=left] {$\displaystyle A$};
        \draw (600.17,22.69) node [anchor=north west][inner sep=0.75pt]   [align=left] {$\displaystyle D$};
        \draw (623.29,22.69) node [anchor=north west][inner sep=0.75pt]   [align=left] {$\displaystyle B$};
        \draw (11.55,165.12) node [anchor=north west][inner sep=0.75pt]   [align=left] {$\displaystyle A\otimes _{D} B$};
        \draw (78,166.12) node [anchor=north west][inner sep=0.75pt]   [align=left] {$\displaystyle A\otimes _{D} B$};
        \draw (531,83) node [anchor=north west][inner sep=0.75pt]   [align=left] {$\displaystyle =$};

        \end{tikzpicture}
        
    \end{center}

The first equality holds since $\delta_A$ is an $A$-$A$ bimodule map and hence also a $D$-$D$ bimodule map. The second equality holds by the naturality of the braiding. The third equality follows from the definition of $A\otimes_D B$. The fourth and seventh equalities hold due to the fact that $D$ factors through $C_l(B)$. The fifth equality again holds by the naturality of the braiding. The sixth equality holds since $\delta_B$ is a $B$-$B$ bimodule map and hence also a $D$-$D$ bimodule map. Therefore, by the universal property of the coequalizer, there exists a unique morphism 
$
\delta_{A\otimes_D B}: A\otimes_D B \to (A\otimes_D B)\otimes (A\otimes_D B)
$
such that 
$
(\rho\otimes \rho)\circ \delta_{A\otimes B}=\delta_{A\otimes_D B}\circ \rho.
$ We also have $\rho\circ m_{A\otimes B}\circ \delta_{A\otimes B}=m_{A\otimes_D B}\circ \delta_{A\otimes_D B}\circ \rho$. By the universal property of the coequalizer, it follows that $m_{A\otimes_D B}\circ \delta_{A\otimes_D B}=id_{A\otimes_D B}$. 

From the definition of $m_{A\otimes_D B}$ and $\delta_{A\otimes_D B}$, together with the identity $\delta_{A\otimes B}\circ m_{A\otimes B}=(m_{A\otimes B}\otimes id)\circ(id\otimes \delta_{A\otimes B})$, we obtain
$
\delta_{A\otimes_D B}\circ m_{A\otimes_D B}\circ(\rho\otimes \rho)=(m_{A\otimes_D B}\otimes id)\circ(id\otimes \delta_{A\otimes_D B})\circ (\rho\otimes \rho).
$
By \cite[Lemma~4.5]{davydov2015functoriality}, $\rho\otimes \rho$ is a coequalizer, and by its universal property we conclude that
$
\delta_{A\otimes_D B}\circ m_{A\otimes_D B}=(m_{A\otimes_D B}\otimes id)\circ(id\otimes \delta_{A\otimes_D B}).
$
Similarly, one shows that
$
\delta_{A\otimes_D B}\circ m_{A\otimes_D B}=(id\otimes m_{A\otimes_D B})\circ(\delta_{A\otimes_D B}\otimes id).
$
Thus, $\delta_{A\otimes_D B}$ is an $A\otimes_D B$-$A\otimes_D B$ bimodule map. Therefore, $A\otimes_D B$ is a separable algebra.
\end{proof}

\begin{corollary}
    Let $\m{C}$ be a braided fusion category, let $A$ and $B$ be condensable algebras in $\m{C}$, and let $D$ be a common condensable subalgebra of $A$ and $B$. Then $M = A \otimes_D B$ is a separable algebra.
\end{corollary}

\begin{lemma}
    Let $\m{C}$ be a braided monoidal category, and let $A, B \in \m{C}$ be commutative algebras. Then the algebra homomorphisms $id \otimes u_B : A \to A \otimes B$ and $u_A \otimes id : B \to A \otimes B$ factor through $C_l(A \otimes B)$ and $C_r(A \otimes B)$, respectively.

\end{lemma}
\begin{proof}
    By the definition of the right center, to prove that $u_A \otimes id$ factors through the right center, it suffices to show that the following diagram is commutative:

    \begin{center}
        \begin{tikzcd}
            (A\otimes B)\otimes B \arrow[dd, "{\beta_{A\otimes B,B}}"'] \arrow[r, "id\otimes (u_A\otimes id)"] & (A\otimes B)\otimes (A\otimes B) \arrow[rd, "m_{A\otimes B}"] &            \\
                                                                                                               &                                                               & A\otimes B \\
            B\otimes (A\otimes B) \arrow[r, "(u_A\otimes id)\otimes id"]                                       & (A\otimes B)\otimes (A\otimes B). \arrow[ru, "m_{A\otimes B}"] &         
            \end{tikzcd}
    \end{center}
    This diagram commutes due to the following identities:
    \begin{center}

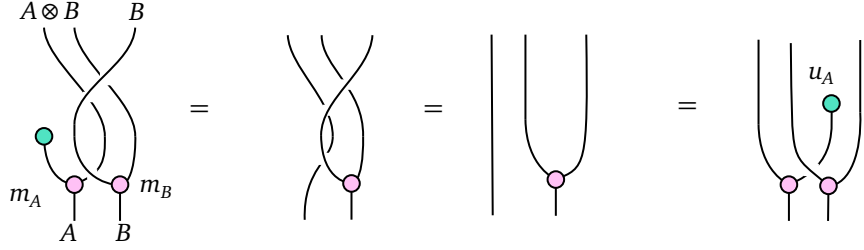
        \tikzset{every picture/.style={line width=0.75pt}} 

        \begin{tikzpicture}[x=0.75pt,y=0.75pt,yscale=-0.85,xscale=0.85]
        
        \draw    (541.33,91.33) .. controls (541.33,105.33) and (545.33,127.33) .. (516.22,139.23) ;
        \draw    (112.31,110.71) .. controls (112.61,124.26) and (111.71,132.1) .. (94.36,139.23) ;
        \draw    (76.41,46.52) .. controls (77.61,67.21) and (112.61,89.31) .. (112.31,110.71) ;
        \draw    (94.36,46.52) .. controls (95.56,67.21) and (130.56,89.31) .. (130.26,110.71) ;
        \draw    (94.36,103.58) -- (94.36,110.71) ;
        \draw    (76.41,110.71) .. controls (76.71,139.23) and (94.66,139.23) .. (94.36,139.23) ;
        \draw    (94.36,139.23) -- (94.36,160.63) ;
        \draw    (121.29,139.23) -- (121.29,160.63) ;
        \draw    (130.26,110.71) .. controls (130.56,111.42) and (130.56,139.23) .. (121.29,139.23) ;
        \draw  [color={rgb, 255:red, 255; green, 255; blue, 255 }  ,draw opacity=1 ][fill={rgb, 255:red, 255; green, 255; blue, 255 }  ,fill opacity=1 ] (118.29,70.77) -- (110,70.77) -- (110,80) -- (118.29,80) -- cycle ;
        \draw  [color={rgb, 255:red, 255; green, 255; blue, 255 }  ,draw opacity=1 ][fill={rgb, 255:red, 255; green, 255; blue, 255 }  ,fill opacity=1 ] (102.44,130.68) -- (107.23,130.68) -- (107.23,135.67) -- (102.44,135.67) -- cycle ;
        \draw  [color={rgb, 255:red, 255; green, 255; blue, 255 }  ,draw opacity=1 ][fill={rgb, 255:red, 255; green, 255; blue, 255 }  ,fill opacity=1 ] (100,85.01) -- (104.79,85.01) -- (104.79,90) -- (100,90) -- cycle ;
        \draw    (94.36,103.58) .. controls (94.66,81.47) and (128.77,68.63) .. (130.26,46.52) ;
        \draw    (220.33,51) .. controls (221.33,70) and (247.17,89.31) .. (246.87,110.71) ;
        \draw    (240.33,51) .. controls (240.33,70) and (265.2,82.18) .. (264.82,110.71) ;
        \draw    (257.95,138.53) -- (257.95,159.92) ;
        \draw    (264.82,110.71) .. controls (265.2,139.95) and (257.33,139) .. (257.95,138.53) ;
        \draw  [color={rgb, 255:red, 255; green, 255; blue, 255 }  ,draw opacity=1 ][fill={rgb, 255:red, 255; green, 255; blue, 255 }  ,fill opacity=1 ] (251.33,77.67) -- (257.33,77.67) -- (257.33,82.67) -- (251.33,82.67) -- cycle ;
        \draw  [color={rgb, 255:red, 255; green, 255; blue, 255 }  ,draw opacity=1 ][fill={rgb, 255:red, 255; green, 255; blue, 255 }  ,fill opacity=1 ] (246.33,92.67) -- (241.33,92.67) -- (241.33,98) -- (246.33,98) -- cycle ;
        \draw    (240,110) .. controls (240.3,87.89) and (268.84,73.11) .. (270.33,51) ;
        \draw    (246.87,110.71) .. controls (247.17,124.26) and (229.33,130) .. (230.33,160) ;
        \draw    (378.6,136.23) -- (378.6,157.63) ;
        \draw    (396.57,50.66) .. controls (397.76,129.1) and (395.96,129.1) .. (378.6,136.23) ;
        \draw    (498.28,53.66) -- (498.28,103.58) ;
        \draw    (558.33,53.84) -- (558.33,96.63) ;
        \draw    (498.28,103.58) .. controls (498.58,132.1) and (516.52,139.23) .. (516.22,139.23) ;
        \draw    (516.22,139.23) -- (516.22,160.63) ;
        \draw  [color={rgb, 255:red, 255; green, 255; blue, 255 }  ,draw opacity=1 ][fill={rgb, 255:red, 255; green, 255; blue, 255 }  ,fill opacity=1 ] (525.51,133.58) -- (533.33,133.58) -- (533.33,127.52) -- (525.51,127.52) -- cycle ;
        \draw    (539.57,139.95) .. controls (552.43,139.23) and (559.33,120) .. (558.33,96.63) ;
        \draw  [fill={rgb, 255:red, 255; green, 193; blue, 245 }  ,fill opacity=1 ] (89.47,139.23) .. controls (89.47,136.53) and (91.66,134.33) .. (94.36,134.33) .. controls (97.06,134.33) and (99.25,136.53) .. (99.25,139.23) .. controls (99.25,141.94) and (97.06,144.13) .. (94.36,144.13) .. controls (91.66,144.13) and (89.47,141.94) .. (89.47,139.23) -- cycle ;
        \draw  [fill={rgb, 255:red, 80; green, 227; blue, 194 }  ,fill opacity=1 ] (71.52,110.71) .. controls (71.52,108) and (73.71,105.81) .. (76.41,105.81) .. controls (79.11,105.81) and (81.3,108) .. (81.3,110.71) .. controls (81.3,113.41) and (79.11,115.61) .. (76.41,115.61) .. controls (73.71,115.61) and (71.52,113.41) .. (71.52,110.71) -- cycle ;
        \draw  [fill={rgb, 255:red, 255; green, 193; blue, 245 }  ,fill opacity=1 ] (511.33,139.23) .. controls (511.33,136.53) and (513.52,134.33) .. (516.22,134.33) .. controls (518.92,134.33) and (521.11,136.53) .. (521.11,139.23) .. controls (521.11,141.94) and (518.92,144.13) .. (516.22,144.13) .. controls (513.52,144.13) and (511.33,141.94) .. (511.33,139.23) -- cycle ;
        \draw    (94.36,110.71) .. controls (94.66,131.63) and (111.71,139.95) .. (121.29,139.23) ;
        \draw  [fill={rgb, 255:red, 255; green, 193; blue, 245 }  ,fill opacity=1 ] (116.4,139.23) .. controls (116.4,136.53) and (118.59,134.33) .. (121.29,134.33) .. controls (123.99,134.33) and (126.18,136.53) .. (126.18,139.23) .. controls (126.18,141.94) and (123.99,144.13) .. (121.29,144.13) .. controls (118.59,144.13) and (116.4,141.94) .. (116.4,139.23) -- cycle ;
        \draw  [color={rgb, 255:red, 255; green, 255; blue, 255 }  ,draw opacity=1 ][fill={rgb, 255:red, 255; green, 255; blue, 255 }  ,fill opacity=1 ] (245.33,121) -- (239.33,121) -- (239.33,127) -- (245.33,127) -- cycle ;
        \draw    (240,110) .. controls (240.3,138.53) and (258.25,138.53) .. (257.95,138.53) ;
        \draw  [fill={rgb, 255:red, 255; green, 193; blue, 245 }  ,fill opacity=1 ] (253.06,138.53) .. controls (253.06,135.82) and (255.25,133.63) .. (257.95,133.63) .. controls (260.65,133.63) and (262.84,135.82) .. (262.84,138.53) .. controls (262.84,141.23) and (260.65,143.43) .. (257.95,143.43) .. controls (255.25,143.43) and (253.06,141.23) .. (253.06,138.53) -- cycle ;
        \draw    (360.66,50.66) -- (360.66,100.58) ;
        \draw    (360.66,100.58) .. controls (360.96,129.1) and (378.9,136.23) .. (378.6,136.23) ;
        \draw  [fill={rgb, 255:red, 255; green, 193; blue, 245 }  ,fill opacity=1 ] (373.72,136.23) .. controls (373.72,133.53) and (375.9,131.33) .. (378.6,131.33) .. controls (381.3,131.33) and (383.49,133.53) .. (383.49,136.23) .. controls (383.49,138.94) and (381.3,141.13) .. (378.6,141.13) .. controls (375.9,141.13) and (373.72,138.94) .. (373.72,136.23) -- cycle ;
        \draw    (539.57,139.95) -- (539.33,160.67) ;
        \draw    (517.33,55.67) .. controls (518.33,130.67) and (516.33,119.33) .. (539.57,139.95) ;
        \draw  [fill={rgb, 255:red, 255; green, 193; blue, 245 }  ,fill opacity=1 ] (534.68,139.95) .. controls (534.68,137.24) and (536.87,135.05) .. (539.57,135.05) .. controls (542.27,135.05) and (544.46,137.24) .. (544.46,139.95) .. controls (544.46,142.65) and (542.27,144.85) .. (539.57,144.85) .. controls (536.87,144.85) and (534.68,142.65) .. (534.68,139.95) -- cycle ;
        \draw  [fill={rgb, 255:red, 80; green, 227; blue, 194 }  ,fill opacity=1 ] (536.44,91.33) .. controls (536.44,88.63) and (538.63,86.43) .. (541.33,86.43) .. controls (544.03,86.43) and (546.22,88.63) .. (546.22,91.33) .. controls (546.22,94.04) and (544.03,96.23) .. (541.33,96.23) .. controls (538.63,96.23) and (536.44,94.04) .. (536.44,91.33) -- cycle ;
        \draw    (340.66,50.66) -- (341.33,157.33) ;
        
        \draw (160.5,90.3) node [anchor=north west][inner sep=0.75pt]   [align=left] {$\displaystyle =$};
        \draw (448.63,88.87) node [anchor=north west][inner sep=0.75pt]   [align=left] {$\displaystyle =$};
        \draw (298.73,90.3) node [anchor=north west][inner sep=0.75pt]   [align=left] {$\displaystyle =$};
        \draw (125,30) node [anchor=north west][inner sep=0.75pt]   [align=left] {$\displaystyle B$};
        \draw (61.64,29.68) node [anchor=north west][inner sep=0.75pt]   [align=left] {$\displaystyle A\otimes B$};
        \draw (85,160.46) node [anchor=north west][inner sep=0.75pt]   [align=left] {$\displaystyle A$};
        \draw (117,160.46) node [anchor=north west][inner sep=0.75pt]   [align=left] {$\displaystyle B$};
        \draw (54,140) node [anchor=north west][inner sep=0.75pt]   [align=left] {$\displaystyle m_{A}$};
        \draw (131,134.46) node [anchor=north west][inner sep=0.75pt]   [align=left] {$\displaystyle m_{B}$};
        \draw (526.19,67.5) node [anchor=north west][inner sep=0.75pt]   [align=left] {$\displaystyle u_{A}$};

        \end{tikzpicture}
        
\end{center}
Similarly, we can prove that $id\otimes u_B$ factors through $C_r(A\otimes B)$.
\end{proof}
\begin{lemma}\label{lemma:factor through center}
    Let $\m{C}$ be a braided monoidal category, and let $D \in \m{C}$ be an algebra, while $A, B \in \m{C}$ are commutative algebras. Suppose there exist algebra homomorphisms $\pi_A : D \to A$ and $\pi_B : D \to B$. Then the algebra homomorphisms $f_A : A \to A \otimes_D B$ and $f_B : B \to A \otimes_D B$, defined in Equation~\ref{fAfB}, factor through $C_l(A \otimes_D B)$ and $C_r(A \otimes_D B)$, respectively.
\end{lemma}
\begin{proof}
    Denote by $\rho : A \otimes B \to A \otimes_D B$ the morphism in the coequalizer. Since we have  
$$
(f_A\otimes id_{A \otimes B}) \circ (id_A \otimes \rho) 
= (\rho \otimes \rho) \circ ((id_A \otimes u_B) \otimes (id_{A\otimes B})),
$$
$$
\beta_{A, A \otimes_D B} \circ (id_A \otimes \rho) 
= (\rho\otimes id_A) \circ \beta_{A, A\otimes B},
$$
$$
\rho \circ m_{A \otimes B} 
= m_{A \otimes_D B} \circ (\rho \otimes \rho),
$$
it follows that, by $m_{A \otimes B} \circ ( (id_A \otimes u_B)\otimes id_{A \otimes B})
= m_{A \otimes B} \circ (id_{A \otimes B}\otimes (id_A \otimes u_B)) \circ \beta_{A, A\otimes B}$, we obtain
$$
m_{A \otimes_D B} \circ (f_A\otimes id_{A\otimes_D B} )\circ (id_A\otimes \rho)
= m_{A \otimes_D B} \circ (id_{A \otimes_D B}\otimes f_A) \circ \beta_{A, A\otimes_D B}\circ (id_A\otimes \rho).
$$
By the universal property of the coequalizer, we have
$$
m_{A \otimes_D B} \circ (f_A\otimes id_{A \otimes_D B}) 
= m_{A \otimes_D B} \circ (id_{A \otimes B}\otimes f_A) \circ \beta_{A, A\otimes_D B}.
$$
Thus, $f_A$ factors through $C_l(A \otimes_D B)$. Similarly, $f_B$ factors through $C_r(A \otimes_D B)$.
\end{proof}

\begin{corollary}\label{cor:lagcen}
Let $\mathcal{C}$ be a braided fusion category, let $A$ and $B$ be Lagrangian algebras in $\mathcal{C}$, and let $D$ be a common condensable subalgebra of $A$ and $B$. If $M = A \otimes_D B$ is indecomposable, then $C_l(M)=A$ and $C_r(M)=B$.
\end{corollary}

\begin{proof}
Since $M$ is indecomposable, by \cite{ostrik2003module}, there exists a connected algebra $M' \in \m{C}$ that is Morita equivalent to $M$. By \cite[Section 5.4]{ostrik2003module}, we have $C_l(M) \cong C_l(M')$ and $C_r(M) \cong C_r(M')$. Since $M'$ is connected, $C_l(M')$ and $C_r(M')$ are also connected. By Lemma~\ref{lemma:separable}, $M$ is separable. Moreover, the left and right centers of a separable algebra are commutative and separable \cite{frohlich2006correspondences,davydov2013structure}. Hence, $C_l(M)$ and $C_r(M)$ are both condensable algebras.

By Lemma~\ref{lemma:factor through center}, we obtain algebra homomorphisms $A \to C_l(M)$ and $B \to C_r(M)$. Since $A$ and $B$ are Lagrangian algebras, they are simple and have maximal dimension among condensable algebras. Consequently, the algebra homomorphisms $A \to C_l(M)$ and $B \to C_r(M)$ are both isomorphisms.
\end{proof}

\begin{remark}
For a separable algebra $M$ in a braided fusion category $\m{C}$, it is known that $C_r(M)$ and $C_l(M)$ are both commutative and separable. Proposition~2.37 of \cite{frohlich2006correspondences} establishes the corresponding result for symmetric special Frobenius algebras. Corollary~3.3 of \cite{davydov2013structure} provides a different approach: one considers $Z(M)\in Z(\m{C})\simeq \m{C}\boxtimes \m{C}^{rev}$ together with the identifications
$C_l(M)\simeq Z(M)\cap (\m{C}\boxtimes 1)$ and $C_r(M)\simeq Z(M)\cap (1\boxtimes \m{C}^{rev})$.
\end{remark}

\begin{lemma}\label{lemma:left CA module}
    Let $\m{C}$ be a braided monoidal category. Let $A$ be a commutative algebra in $\m{C}$,  and let $M$ be an algebra in $\m{C}$. Assume that there exists an algebra homomorphism $f_A : A \to M$ such that $f_A$ factors through $C_l(M)$ (resp. $C_r(M)$). Then the category $\m{C}_M$ is a left (resp. right) $\m{C}_A$-module.
\end{lemma}

\begin{proof}
    For any $P\in \mathcal{C}_M$, $P$ is automatically a left $A$-module.
Since $A$ is a commutative algebra, $P$ is also a right $A$-module via the natural construction:
$$
A\otimes P\stackrel{\beta_{A,P}}{\longrightarrow} P\otimes A \stackrel{1\otimes f_A}{\longrightarrow} P\otimes M \stackrel{m^{M}_{P}}{\longrightarrow} P.
$$
Since $f_A$ factors through $C_l(M)$, it can be checked that $P$ is an $A$-$M$-bimodule in the natural way.
Thus, for $X\in \mathcal{C}_A$ and $P\in \mathcal{C}_M$, there is a natural morphism $X\otimes_A P\otimes M\to X\otimes_A P$ induced by $m^{M}_P$, which in turn makes $X\otimes_A P$ a right $M$-module.
We can therefore define the action of $\mathcal{C}_A$ on $\mathcal{C}_M$ as $X\odot P:= X\otimes_A P$.
It can be checked that this action makes $\mathcal{C}_M$ a left $\mathcal{C}_A$-module category.
\end{proof}

\begin{corollary}\label{cor: C-C-bimodule}
Let $\m{C}$ be a braided monoidal category. Let $A$ and $B$ be commutative algebras in $\m{C}$, and let $M$ be an algebra in $\m{C}$.
Assume that there exist algebra homomorphisms $f_A : A \to M$ and $f_B:B\to M$ such that $f_A$ factors through $C_l(M)$ and $f_B$ factors through $C_r(M)$.
Then the category $\m{C}_M$ is a $\m{C}_A$-$\m{C}_B$ bimodule.
Furthermore, $\m{C}_M$ is a $\m{C}$-$\m{C}$ bimodule, and the left and right actions of $\m{C}$ on $\m{C}_M$ are naturally isomorphic.
\end{corollary}

\begin{proof}
Since $f_A$ factors through $C_l(M)$ and $f_B$ factors through $C_r(M)$, Lemma \ref{lemma:left CA module} implies that $\mathcal{C}_M$ is a left $\mathcal{C}_A$-module and a right $\mathcal{C}_B$-module. For $X\in \mathcal{C}_A$, $Y\in \mathcal{C}_B$, and $P\in \mathcal{C}_M$, the morphisms $X\otimes P\to X\otimes _A P$ and $P\otimes Y\to P\otimes_B Y$ are all right $M$-module morphisms. Consequently, there is a natural right $M$-module isomorphism
$
(X\otimes_A P)\otimes_B Y\simeq X\otimes_A(P\otimes_B Y),
$
which makes $\mathcal{C}_M$ a $\mathcal{C}_A$-$\mathcal{C}_B$ bimodule.

Since there are monoidal functors $\mathcal{C}\to \mathcal{C}_A: X\mapsto X\otimes A$ and $\mathcal{C}\to \mathcal{C}_B: X\mapsto B\otimes X$, $\mathcal{C}_M$ is automatically a $\mathcal{C}$-$\mathcal{C}$ bimodule. For any $X\in \mathcal{C}$ and $P\in \mathcal{C}_M$, the left action of $\mathcal{C}$ on $\mathcal{C}_M$ is defined as $X\odot P:= (X\otimes A)\otimes_A P$, and the right action of $\mathcal{C}$ on $\mathcal{C}_M$ is $P \odot X:= P\otimes_B(B\otimes X)$. The braiding $X\otimes P\stackrel{\beta_{P,X}^{-1}}{\longrightarrow} P\otimes X$ induces a natural right $M$-module isomorphism $X\odot P \simeq P\odot X$. This isomorphism shows that the left and right actions of $\mathcal{C}$ on $\mathcal{C}_M$ are naturally isomorphic.
\end{proof}

\begin{proposition}\label{prop:inv}
    Let $\m{C}$ be a braided fusion category, let $A,B$ are Lagrangian algebras in $\m{C}$, and let $D$ is a common condensable algebra of $A$ and $B$. Assume that $M=A\otimes_D B$ is indecomposable, then $\m{C}_M$ is an invertible $\m{C}_A$-$\m{C}_B$ bimodule.
\end{proposition}
\begin{proof}
    By Lemma \ref{lemma:separable}, $M$ is separable; thus, $\m{C}_M$ is a finite semisimple category. Since $M$ is indecomposable, $\m{C}_M$ is an indecomposable $\m{C}$-module category, as well as an indecomposable $\m{C}_A$-$\m{C}_B$ bimodule category. 

According to Lemma \ref{lemma:left CA module}, the left action of $\m{C}_A$ on $\m{C}_M$ is given by $X \odot P = X \otimes_A P$ for $X \in \m{C}_A$ and $P \in \m{C}_M$. Since we have the natural isomorphism $Hom_{\m{C}_A}(X, M) \simeq Hom_{\m{C}_M}(X \otimes_A M, M)$, it follows that the internal hom $[M, M]_{\m{C}_A} \in \m{C}_A$ is identified with $M$. Consequently, by Theorem \ref{thm:int}, we have an equivalence of categories $\m{C}_M \simeq (\m{C}_A)_{[M, M]_{\m{C}_A}} = (\m{C}_A)_M$. It then follows that $Fun_{\m{C}_A}(\m{C}_M, \m{C}_M)^{\text{rev}} \simeq Fun_{\m{C}_A}((\m{C}_A)_M, (\m{C}_A)_M)^{\text{rev}} \simeq {}_M(\m{C}_A)_M$.

For $X \in \m{C}$, we endow $M \otimes X$ with a right $M$-module structure via the composition

$$
M \otimes X \otimes M
\stackrel{id_M\otimes \beta_{M,X}^{-1}}{\longrightarrow}
M \otimes M \otimes X
\stackrel{m_M\otimes id_X}{\longrightarrow}
M \otimes X.
$$

This construction yields a natural isomorphism

$$
Hom_{_M(\m{C}_A)_M}(M \otimes X, M) \simeq Hom_{\m{C}}(X, C_r(M)),
$$

given by

$$
f \mapsto \Phi(f),
$$

with inverse

$$
m_M\circ (id_M\otimes (\iota\circ g))\mapsfrom g.
$$

Here $\iota:C_r(M)\to M$ denotes the canonical inclusion appearing in the definition of the right center $C_r(M)$. To define $\Phi(f)$, consider $
\tilde{f}=f\circ (u_M\otimes id_X):X\to M$, we have $m_M\circ(\tilde{f}\otimes id_M)=m_M\circ(id_M\otimes\tilde{f})\circ\beta_{M,X}^{-1}$.
Thus, by the universal property of $C_r(M)$, there exists a unique morphism $\Phi(f):X\to C_r(M)$ such that $\iota\circ\Phi(f)=\tilde{f}$. It is straightforward to check that these two constructions are mutually inverse.

The natural isomorphism $Hom_{_M(\m{C}_A)_M}(M \otimes X, M) \simeq Hom_{\m{C}}(X, C_r(M))$ implies that the right adjoint $F^R$ of the functor $F:\m{C}\to {}_M(\m{C}_A)_M, X\mapsto M\otimes X$ satisfies $F^R(1_{{}_M(\m{C}_A)_M})=C_r(M)$. By \cite[Lemma 3.5]{davydov2013witt}, we obtain $F(\m{C})\simeq \m{C}_{C_r(M)}$ as fusion categories. By Corollary \ref{cor:lagcen}, we have $C_r(M)=B$. Therefore,
Therefore,
$$
\fpdim(F(\mathcal C))
=
\fpdim(\mathcal C_{C_r(M)})
=
\frac{\fpdim(\mathcal C)}{\fpdim(C_r(M))}
=
\frac{\fpdim(\mathcal C)}{\fpdim(B)}
$$
$$
=
\frac{\fpdim(\mathcal C)}{\fpdim(A)}
=
\fpdim(\mathcal C_A)
=
\fpdim({}_{M}(\mathcal C_A)_{M}).
$$

Since $F(\m{C})$ is a fusion subcategory of ${}_M(\m{C}_A)_M$, the equality of Frobenius--Perron dimensions implies $F(\m{C})={}_M(\m{C}_A)_M$. Consequently, ${}_M(\m{C}_A)_M\simeq \m{C}_B$ as fusion categories (A similar argument can be found in \cite[Section 3.1]{kong2014anyon}). Therefore, $\m{C}_A$ and $\m{C}_B$ are Morita equivalent, and $\m{C}_M$ is an invertible $\m{C}_A$-$\m{C}_B$ bimodule category.
\end{proof}

\begin{proposition}\label{prop:fusionbimod}
    Let $\m{C}$ be a non-degenerate braided fusion category, let $A,B$ are Lagrangian algebras in $\m{C}$, and let $D$ is a common condensable algebra of $A$ and $B$.
    Assume that $M=A\otimes_D B$ is indecomposable, then $Fun_{\m{C}_A|\m{C}_B}(\m{C}_M, \m{C}_M)$ is equivalent to $\m{C}$ as a closed fusion $\m{C}$-$\m{C}$ bimodule (see \cite{kong2018center} for the definition of a closed fusion bimodule). Here, the closed fusion $\m{C}$-$\m{C}$ bimodule structure on $\m{C}$ is induced by the tensor product $\otimes$ of $\m{C}$ in the canonical way.
\end{proposition}
\begin{proof}
By Proposition \ref{prop:inv}, $\m{C}_M$ is an invertible $\m{C}_A$-$\m{C}_B$ bimodule. According to \cite[Proposition 2.4.4]{kong2018center}, we have the following two equivalences of fusion categories:
$$
Z(\m{C}_A)\rightarrow Fun_{\m{C}_A|\m{C}_B}(\m{C}_M,\m{C}_M), \quad \quad X\mapsto X\otimes_A -
$$
$$
Z(\m{C}_B)\rightarrow Fun_{\m{C}_A|\m{C}_B}(\m{C}_M,\m{C}_M), \quad \quad Y\mapsto -\otimes_B Y.
$$

Furthermore, it is known that there are two braided monoidal category equivalences: $\m{C}\to Z(\m{C}_A),\quad X\mapsto X\otimes A$ and $\m{C}\to Z(\m{C}_B),\quad X\mapsto B\otimes X$. Thus, $Fun_{\m{C}_A|\m{C}_B}(\m{C}_M,\m{C}_M)$ is a closed fusion $\m{C}$-$\m{C}$ bimodule with the left action $(X\odot F)(-):= X\otimes F(-)$ and the right action $(F\otimes X)(-):= F(-)\otimes X$ for $X\in \m{C}$ and $F\in Fun_{\m{C}_A|\m{C}_B}(\m{C}_M,\m{C}_M)$. By Corollary \ref{cor: C-C-bimodule}, the left and right actions are naturally isomorphic, and these two actions are equivalent to the action of $\m{C}$ on itself. Consequently, $Fun_{\m{C}_A|\m{C}_B}(\m{C}_M, \m{C}_M)$ is equivalent to $\m{C}$ as a closed fusion $\m{C}$-$\m{C}$ bimodule.
\end{proof}

\begin{remark}

The physical meaning of Proposition \ref{prop:fusionbimod} is as follows. According to \cite[Theorem 3.3.7]{kong2018center}, there exists a fully faithful functor, called the center functor:
$$
\m{M}\m{F}us^{ind} \stackrel{Z}{\longrightarrow} \m{B}\m{F}us^{cl}
$$
$$
(\m{S}\stackrel{\m{M}}{\longrightarrow}\m{S}')\mapsto (Z(\m{S})\stackrel{Fun_{\m{S}|\m{S}'}(\m{M},\m{M})}{\longrightarrow}Z(\m{S}')).
$$
Here, $\m{M}\m{F}us^{ind}$ is the category of indecomposable multi-fusion categories, where the morphisms are the equivalence classes of non-zero semisimple bimodules. $\m{B}\m{F}us^{cl}$ is the category of non-degenerate braided fusion categories, where the morphisms are the equivalence classes of closed multi-fusion bimodules. Following \cite{kong2015boundary}, the physical interpretation of \cite[Theorem 3.3.7]{kong2018center} is that the 1+1D domain walls in the bulk correspond to the 0+1D domain walls on the boundary, as illustrated in Figure \ref{Fig:fully faithful}. Thus, Proposition \ref{prop:fusionbimod} implies that the bulk 1+1D domain wall corresponding to $\m{C}_M$ is the trivial 1+1D domain wall, which in turn means that $\m{C}_M$ is the category of closed 0+1D defects between $\m{C}_A$ and $\m{C}_B$.
    
\end{remark}
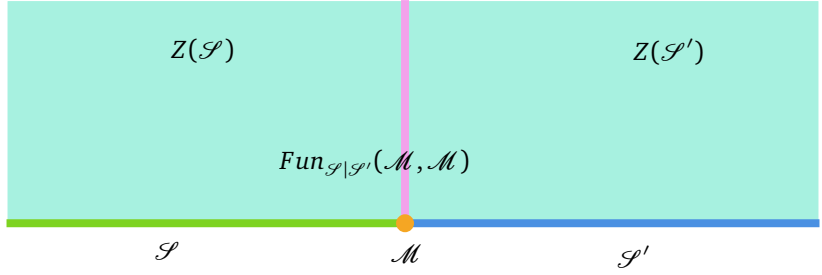
\begin{figure}
    \centering
\tikzset{every picture/.style={line width=0.75pt}} 

\begin{tikzpicture}[x=0.75pt,y=0.75pt,yscale=-0.8,xscale=0.8]

\draw  [color={rgb, 255:red, 255; green, 255; blue, 255 }  ,draw opacity=1 ][fill={rgb, 255:red, 80; green, 227; blue, 194 }  ,fill opacity=0.5 ] (70,40) -- (580,40) -- (580,180) -- (70,180) -- cycle ;
\draw [color={rgb, 255:red, 126; green, 211; blue, 33 }  ,draw opacity=1 ][line width=3]    (70,180) -- (320,180) ;
\draw [color={rgb, 255:red, 74; green, 144; blue, 226 }  ,draw opacity=1 ][line width=3]    (580,180) -- (320,180) ;
\draw [color={rgb, 255:red, 246; green, 159; blue, 236 }  ,draw opacity=0.96 ][line width=3]    (320,180) -- (320,40) ;
\draw  [color={rgb, 255:red, 245; green, 166; blue, 35 }  ,draw opacity=1 ][fill={rgb, 255:red, 245; green, 166; blue, 35 }  ,fill opacity=1 ] (315,180) .. controls (315,177.24) and (317.24,175) .. (320,175) .. controls (322.76,175) and (325,177.24) .. (325,180) .. controls (325,182.76) and (322.76,185) .. (320,185) .. controls (317.24,185) and (315,182.76) .. (315,180) -- cycle ;

\draw (171,62) node [anchor=north west][inner sep=0.75pt]   [align=left] {$\displaystyle Z(\mathcal{S})$};
\draw (161,190) node [anchor=north west][inner sep=0.75pt]   [align=left] {$\displaystyle \mathcal{S}$};
\draw (451,192) node [anchor=north west][inner sep=0.75pt]   [align=left] {$\displaystyle \mathcal{S} '$};
\draw (309,192) node [anchor=north west][inner sep=0.75pt]   [align=left] {$\displaystyle \mathcal{M}$};
\draw (461,62) node [anchor=north west][inner sep=0.75pt]   [align=left] {$\displaystyle Z(\mathcal{S} ')$};
\draw (239,132) node [anchor=north west][inner sep=0.75pt]   [align=left] {$\displaystyle Fun_{\mathcal{S} |\mathcal{S} '}(\mathcal{M} ,\mathcal{M})$};

\end{tikzpicture}
\caption{The physical meaning of center functor}
\label{Fig:fully faithful}
    \end{figure}

Consider the internal hom of the $\mathcal{C}$-module $\mathcal{C}_M$. By Lemma \ref{compute_internalhom}, we have $[M,M]\simeq (M\otimes _M M^R)^L\simeq M$.
Thus, $M=A\otimes_D B$ can be identified with the full centralizer of $M$ in the $\m{C}$-module $\m{C}_M$.
Consequently, starting from two Lagrangian algebras $A, B \in \m{C}$ and a common condensable subalgebra $D$ of $A$ and $B$, we obtain a critical cospan:
\begin{center}
    \begin{tikzcd}
        & M=A\otimes_D B=Z(M) &                   \\
A=Z(a) \arrow[ru] &                     & B=Z(b) .\arrow[lu]
\end{tikzcd}
\end{center}

\begin{remark}
    The algebra $M=A\otimes_D B$ may be decomposable, i.e., $M\simeq \bigoplus_i M_i$, where each $M_i$ is indecomposable.
In this case, the module category $\m{C}_M$ decomposes as a direct sum: $\m{C}_M\simeq \boxplus_i \m{C}_{M_i}$.
All $\m{C}_{M_i}$ must be equivalent to each other, and consequently, all $M_i$ are Morita equivalent.
Let $D_i=A\times_{M_i} B$ be the algebra corresponding to the critical cospan $A\to M_i \leftarrow B$. Then $D_i$ may not be isomorphic to $D_j$ when $i \neq j$.
The algebra $D=A\times_M B$ corresponds to the multiphase point and serves as a common condensable subalgebra of all $D_i$.
\end{remark}

\begin{remark}
    If we choose another 0+1D defect $0\neq x\in \m{C}_M$, then the internal hom $[x,x]_{\m{C}}$ may not be isomorphic to $M$. Furthermore, the resulting subalgebra $D'=A\times_{[x,x]} B$ may not be isomorphic to $A\times_M B$, and the algebra $A\otimes_{D'}B$ may not be isomorphic to $M$ or $[x,x]$. However, we still have that $[x,x]$ is Morita equivalent to $M$.
\end{remark}

\begin{example}
    Let $\mathcal{C}=Z(Vec_{\mathbb{Z}_4})$. The simple objects in $Z(Vec_{\mathbb{Z}_4})$ can be written as $e^am^b$, where $a,b\in \{0,1,2,3\}$. There are $7$ condensable algebras in $Z(Vec_{\mathbb{Z}_4})$ \cite{Chatterjee_2023}:

    $$
    1, \quad 1 \oplus e^2 , \quad 1\oplus e^2m^2 , \quad 1\oplus m^2 ,
    $$
    $$
    1\oplus e \oplus e^2 \oplus e^3 , \quad 1\oplus e^2\oplus m^2\oplus e^2m^2, \quad 1\oplus m\oplus m^2\oplus m^3.
    $$  

    The condensable algebras in the second line are Lagrangian algebras. Let $A=1\oplus e \oplus e^2 \oplus e^3$ and $B=1\oplus e^2\oplus m^2\oplus e^2m^2$. We can compute that $\mathcal{S}=\mathcal{C}_A\simeq Vec_{\mathbb{Z}_4}$. Let's denote the objects in $\mathcal{S}$ as $1, M, M^2,M^3$. Let $b=1\oplus M^2$ be a 1d condensable algebra in $\mathcal{S}$. We can compute that $\mathcal{C}_B\simeq {}_b\mathcal{S}_b\simeq Vec_{\mathbb{Z}_2\times \mathbb{Z}_2}^{\omega}$, where $\omega$ corresponds to the mixed anomaly, and we have $Z(1_{\mathcal{S}})\simeq A$ and $Z(b)\simeq B$.

    The 0+1D domain wall between $\mathcal{S}$ and ${}_b \mathcal{S}_b$ is ${}_1\mathcal{S}_b\simeq Fun_{\mathcal{S}}(\mathcal{S},\mathcal{S}_b)$. Let's denote the simple objects in $\mathcal{S}_b$ as $X:=1\oplus M^2$ and $Y:=M\oplus M^3$. Then we can compute that there are two simple objects in $Fun_{\mathcal{S}}(\mathcal{S},\mathcal{S}_b)$:
    $$
    F_1: 1\mapsto  X, \quad M\mapsto Y, \quad M^2\mapsto X,\quad M^3\mapsto  Y,
    $$
    $$
    F_2: 1\mapsto  Y, \quad M\mapsto X, \quad M^2\mapsto Y,\quad M^3\mapsto  X.
    $$

    Let us compute the full centralizers $Z(F_1)=[F_1,F_1]_{\mathcal{C}}$ and $Z(F_2)=[F_2,F_2]_{\mathcal{C}}$. The $\mathcal{C}=Z(\mathcal{S})$-module structure on $Fun_{\mathcal{S}}(\mathcal{S},\mathcal{S}_b)$ is induced by the forgetful functor $Z(\mathcal{S})\to \mathcal{S}$. The $\mathcal{S}$-module structure on $Fun_{\mathcal{S}}(\mathcal{S},\mathcal{S}_b)$ is given by the relations
$$
M\odot F_1\simeq F_2,\quad M\odot F_2\simeq F_1.
$$
Thus, by $Hom_{Fun_{\mathcal{S}}(\mathcal{S},\mathcal{S}_b)}(Z\odot F_1,F_1)\simeq Hom_{Z(\mathcal{S})}(Z,[F_1,F_1])$ and Equation \ref{com_internal_hom}, we can conclude that $[F_1,F_1]$ is the direct sum of the simple objects whose image under the forgetful functor is $1$ or $M^2$. Thus, we obtain $[F_1,F_1]=1\oplus e\oplus e^2\oplus e^3\oplus m^2\oplus em^2\oplus e^2m^2\oplus e^3m^2$. We can also obtain $[F_2,F_2]=1\oplus e\oplus e^2\oplus e^3\oplus m^2\oplus em^2\oplus e^2m^2\oplus e^3m^2$. So we have $M=[F_1,F_1]=[F_2,F_2]$.

    It can be computed that $D=A\times_M B= 1\oplus e^2$. We can reconstruct $M$ by $A\otimes_D B=(1\oplus e \oplus e^2 \oplus e^3)\otimes_{1\oplus e^2}(1\oplus e^2\oplus m^2\oplus e^2m^2)= 1\oplus e\oplus e^2\oplus e^3\oplus m^2\oplus em^2\oplus e^2m^2\oplus e^3m^2$.

\end{example}
\section{The classification of condensable algebras in $Z(Vec_G^\omega)$}\label{section5}

Let $G$ be a finite group and let $\omega\in H^3(G,U(1))$ be a 3-cohomology class. We denote by $Vec_G^\omega$ the  fusion category of finite dimensional $G$-graded vector spaces with associator $\omega$. Whenever explicit cocycle data are required, we choose a normalized representative $\hat{\omega}\in Z^3(G,U(1))$ of $\omega$. For a representative $\hat{\omega}\in Z^3(G,U(1))$ of $\omega$, we use the convention that the associator of $Vec_G^\omega$ is $X_g\otimes(Y_h\otimes Z_k)\longrightarrow (X_g\otimes Y_h)\otimes Z_k$ and acts by multiplication by $\hat{\omega}(g,h,k)$ on objects $X_g,Y_h,Z_k$ of degrees $g,h,k\in G$. 

The classification of condensable algebras in $Z(Vec_G^\omega)$ has been studied in \cite{davydov2010modular,davydov2017lagrangian}. To facilitate the identification of embedding relations, we propose a new classification method parametrizes condensable algebras. Based on this parametrization, we derive conditions to determine whether an embedding relation exists between two condensable algebras and, if so, how many inequivalent embeddings exist. 

\subsection{The quadruple description of condensable algebras in $Z(Vec_G^\omega)$}\label{section5.class}

In this section, we prove Theorem~\ref{classalgebra}. We parametrize condensable algebras in $Z(Vec_G^\omega)$ by quadruples $(H,F,\alpha,\phi)$ and determine when two such quadruples correspond to isomorphic condensable algebras.

The first two parameters $(H,F)$ are obtained in \cite{davydov2010modular,davydov2017lagrangian}. For a condensable algebra in $Z(Vec_G^\omega)$, one can show that its degree $e$ part is isomorphic to a function algebra $Fun(G/H)$ with $G$ action for some subgroup $H \leqslant G$. The algebra $Fun(G/H)$ is itself a condensable algebra in $Z(Vec_G^\omega)$, and we have
$Z(Vec_G^\omega)_{Fun(G/H)}^{loc} \simeq Z(Vec_H^{\omega|_H})$.
A condensable algebra in $Z(Vec_G^\omega)$ whose degree $e$ part is isomorphic to $Fun(G/H)$ is the \emph{transfer} (see \cite{davydov2010modular}) of a condensable algebra in $Z(Vec_G^\omega)_{Fun(G/H)}^{loc} \simeq Z(Vec_H^{\omega|_H})$ with trivial degree $e$ part. A condensable algebra with trivial degree $e$ part in $Z(Vec_H^{\omega|_H})$ is equivalent to a commutative central algebra $A\in Vec_H^{\omega|_H}$. What's more, its support is a normal subgroup $F \trianglelefteq H$, and $\dim(A_h)\leqslant 1$ for all $h\in H$ \cite[Lemma 3.4.1]{davydov2010modular} \cite[Corollary 2.13,Lemma 2.14, Properties 2.15]{davydov2017lagrangian}. In other words, as an object of $Vec_H^{\omega|_H}$, we have $A=\oplus_{f\in F} \mathbb{C}_f$.

The last two parameters $(\alpha,\phi)$ parametrize commutative central algebras in $Vec_H^{\omega|_H}$ that are isomorphic to $\oplus_{f\in F} \mathbb{C}_f$ as object. In this subsection, we first establish a classification of such commutative central algebras in terms of these parameters. We then turn to the conjugacy problem arising in this description. Since $Fun(G/H)\simeq Fun(G/gHg^{-1})$, there is a conjugacy issue in the choice of the pair $(H,F)$. Even after fixing $(H,F)$, a further conjugacy issue remains for the parameters $(\alpha,\phi)$. We show that these two levels of conjugacy can be treated in a unified manner, as summarized in Diagram~\ref{quadruples_equdig}.

\begin{definition}
    A {\it central algebra} $(A,\epsilon_A)$ in a tensor category $\m{C}$ is an algebra $A=(A,m_A,u_A)$ in $\m{C}$ equipped with a half-braiding $\epsilon$, such that $(A,\epsilon_A)$ is also an algebra in $Z(\m{C})$. $(A,\epsilon_A)$ is called {\it commutative central algebra} if $m_A=m_A \circ \epsilon_A$.

\end{definition}

\begin{proposition}\cite[Proposition 5.1]{bruguieres2011exact}\label{exact}
Let $F:\m{C}\rightarrow \m{D}$ be a tensor functor between tensor categories admitting a right adjoint $R$. Then $A := R(1)$ carries a natural structure of a commutative central algebra in $\m{C}$, with half-braiding denoted by $\epsilon$. Moreover, if $R$ is faithful and exact, then $\m{C}_{(A,\epsilon)}$ is a tensor category, and there exists a tensor equivalence
$K : \m{D} \rightarrow \m{C}_{(A,\epsilon)}$
such that the following diagram of tensor functors commutes up to tensor isomorphism.

    \begin{center}
        \begin{tikzcd}
            \mathcal{C} \arrow[r, "F"] \arrow[rd, "F_A"'] & \mathcal{D} \arrow[d, "K"]   \\
                                                          & {\mathcal{C}_{(A,\epsilon)}},
            \end{tikzcd}
    \end{center} here $F_A: \m{C}\rightarrow \m{C}_{(A,\epsilon)}, X\mapsto X\otimes A$.
\end{proposition}

\begin{remark}
    Additionally, $R$ is a lax monoidal functor, given by natural transformations $\nu_0:1\to R(1)$ and $\nu(X, Y): R(X) \otimes R(Y) \to R(X \otimes Y)$ for all $X, Y \in \mathcal{D}$. $\nu_0$ and $\nu$ are uniquely determined by the following diagram\cite[Proposition 3.84]{aguiar2010monoidal}:
\begin{equation}\label{lax0}
    \begin{tikzcd}
        1 \arrow[d, "id"'] \arrow[r, "u_1"] & RF(1) \arrow[d, "R(\mu _0)"] \\
        1 \arrow[r, "\nu_0"]                & R(1)                         
        \end{tikzcd}
\end{equation}
\begin{equation}\label{lax}
    \begin{tikzcd}
        R(X)\otimes R(Y) \arrow[d, "u_{R(X)\otimes R(Y)}"'] \arrow[rr, "{\nu(X,Y)}"] &  & R(X\otimes Y)                                           \\
        RF(R(X)\otimes R(Y)) \arrow[rr, "{R(\mu(R(X),R(Y))^{-1})}"]                     &  & R(FR(X)\otimes FR(Y)) \arrow[u, "R(v_{X}\otimes v_Y)"'],
        \end{tikzcd}
   \end{equation} where $\mu_0,\mu$ correspond to the monoidal structure of $F$ and $u,v$ are the unit and counit of the adjoint pair $F\dashv R$. The unit of $R(1)$ is $u_{R(1)}=\nu_0$, and the multiplication of $R(1)$ is $m_{R(1)}=\nu(1,1)$. The half-braiding $\epsilon$ of $R(1)$ is defined by the following commutative diagram:

\begin{equation}\label{halfbraiding}
    \begin{tikzcd}
        R(1)\otimes X \arrow[d, "id \otimes u_X"'] \arrow[rrd, "\simeq"] &  &         &  & X\otimes R(1) \arrow[llll, "\epsilon_X"'] \arrow[d, "u_X\otimes id"] \arrow[lld, "\simeq"'] \\
        R(1)\otimes R(F(X)) \arrow[rr, "{\nu(1,F(X))}"]                     &  & R(F(X)) &  & R(F(X))\otimes R(1) \arrow[ll, "{\nu(F(X),1)}"'].                                             
        \end{tikzcd}
\end{equation} The isomorphisms along the slanted arrows are referred to as the Hopf isomorphisms, induced by the Hopf adjunction\cite[Theorem 6.6]{bruguieres2011hopf}.
\end{remark}

It is well known that monoidal functors between $Vec_{G_1}^{\omega_1}$ and $Vec_{G_2}^{\omega_2}$ are parametrized by pairs $(f,\nu)$, where $f$ is a group homomorphism $f:G_1\to G_2$ and $\nu$ is an element in a torsor over $H^2(G_1,U(1))$ (see \cite[Proposition~2.6.1]{etingof2016tensor}). For representatives $\hat{\omega_1},\hat{\omega_2},\hat{\nu}$ of $\omega_1,\omega_2,\nu$, respectively, we have
$d\hat{\nu}\,\hat{\omega}_1=  f^*(\hat{\omega}_2)$.
The cochain $\hat{\nu}$ encodes the monoidal structure of the monoidal functor $F$, namely the natural isomorphisms
$F(\mathbb{C}_g)\otimes F(\mathbb{C}_h)\stackrel{\hat{\nu}(g,h)}{\longrightarrow} F(\mathbb{C}_g\otimes \mathbb{C}_h)$.
We denote such a monoidal functor by $F_{f,\nu}$.

\begin{remark}\label{fe rightad}
In \cite{hannah2023frobenius}, the right adjoint $R$ of the tensor functor
$F_{f,\nu} : Vec_{G_1}^{\omega_1} \rightarrow Vec_{G_2}^{\omega_2}$
is explicitly constructed. For an object $V = \bigoplus_{h \in G_2} V_h$ in $Vec_{G_2}^{\omega_2}$, define

$$
R(V) := \bigoplus_{g \in G_1} R(V)_g, \quad \text{where} \quad R(V)_g := V_{f(g)}, \quad \forall g \in G_1.
$$
For a morphism $a : \bigoplus_{h \in G_2} V_h \rightarrow \bigoplus_{h \in G_2} W_h$, define
$$
R(a)(v_g) := a(v_g) \in W_{f(g)} = R(W)_g, \quad \forall v_g \in V_{f(g)} = R(V)_g.
$$

The functor $R : Vec_{G_2}^{\omega_2} \rightarrow Vec_{G_1}^{\omega_1}$ defined above is faithful and exact if and only if the group homomorphism $f : G_1 \rightarrow G_2$ is surjective (see \cite[Lemma~3.2]{hannah2023frobenius}). 
    
\end{remark}

We present the following classification result for commutative central algebras in pointed fusion category.

\begin{proposition}\label{prop:classs_cen}
Let $H$ be a finite group and $F \trianglelefteq H$ a normal subgroup. Let $\sigma \in H^3(H, U(1))$ be a $3$-cohomology class. Then there is a bijection between the following sets:
\begin{itemize}
    \item $X$: the set of isomorphism classes of commutative separable central algebras in $Vec_H^{\sigma}$ whose underlying objects are isomorphic to $\oplus_{f \in F} \mathbb{C}_f$.
    \item $Y$: the set of pairs $(\alpha, [\phi])$, where $\alpha \in H^3(H/F, U(1))$ and $[\phi]$ is a natural isomorphism class of monoidal functors $\phi: \text{Vec}_H^\sigma \to \text{Vec}_{H/F}^\alpha$ whose underlying group map is the quotient map $\pi: H \to H/F$, taken modulo monoidal autoequivalences of $\text{Vec}_{H/F}^\alpha$ induced by the identity map on $H/F$.
\end{itemize}
 \end{proposition}
\begin{remark}
    For given $H, F, \sigma, \alpha$, the monoidal functor $\phi: Vec_H^\sigma \to Vec_{H/F}^{\alpha}$ exists if and only $\pi^*\alpha = \sigma$. 
\end{remark}
\begin{remark}
    An element in $Y$ can be represented by a pair $(\alpha,\phi)$, where $\alpha \in H^3(H/F,U(1))$ and $\phi$ is a monoidal functor $\phi: Vec_H^\sigma \to Vec_{H/F}^\alpha$ wohse underlying group is the quotient map $\pi:H\to H/F$. Two representing pairs $(\alpha,\phi)$ and $(\alpha',\phi')$ are equivalent if and only if $\alpha=\alpha'$ and there exists a monoidal lifting
$\theta: Vec_{H/F}^\alpha \simeq Vec_{H/F}^{\alpha}$
of the identity functor on $Vec_{H/F}$ such that the following diagram commutes up to monoidal natural isomorphism:
    \begin{center}
        \begin{tikzcd}
            Vec_H^\sigma \arrow[d, "="] \arrow[r, "\phi"] & Vec_{H/F}^\alpha \arrow[d, "\theta"] \\
            Vec_H^\sigma \arrow[r, "\phi'"]               & Vec_{H/F}^\alpha .                   
            \end{tikzcd}
    \end{center}
\end{remark}

\begin{proof}
    Let $A$ be a commutative separable central algebra in $Vec_H^{\sigma}$ with support on a normal subgroup $F \trianglelefteq H$. Then the category of right $A$-modules $(Vec_H^\sigma)_A$ is a semisimple, pointed fusion category whose fusion rules coincide with the group multiplication of $H/F$. By the classification theorem of pointed fusion categories, there exists $\alpha \in H^3(H/F, U(1))$ such that we have an equivalence of fusion categories $T: (Vec_H^{\sigma})_A\simeq Vec_{H/F}^\alpha$. This equivalence yields a tensor functor $\phi: Vec_H^{\sigma} \to Vec_{H/F}^\alpha$. The group homomorphism underlying $\phi$ is the quotient map $\pi: H \to H/F$. Thus, we have a map 
    $$
    \mathbb{F}: X \to Y, \quad \quad A\mapsto  (\phi: Vec_H^\sigma \stackrel{F_A}{\longrightarrow} (Vec_H^\sigma)_A\stackrel{T,\simeq}{\longrightarrow} Vec_{H/F}^\alpha).
    $$

    Consider a monoidal functor $\phi=F_{\pi,\nu}: Vec_H^{\sigma} \to Vec_{H/F}^\alpha$. By Remark \ref{fe rightad}, $\phi$ admits a faithful and exact right adjoint $\phi^R$, and $\phi^R(\mathbb{C}_e)$ is a commutative central algebra in $Vec_H^{\sigma}$ whose underlying objects is isomorphic to $\oplus_{f\in F}\mathbb{C}_f$. If two monoidal functors $\phi_1, \phi_2: Vec_H^{\sigma} \rightarrow Vec_{H/F}^{\alpha}$ are isomorphic, then the commutative central algebras they induce are also isomorphic. Moreover, suppose $\psi$ is a monoidal autoequivalence of $Vec_{H/F}^\alpha$ such that the corresponding group homomorphism is the identity map. If $\phi_1 = \psi \circ \phi_2$, then the corresponding right adjoints $\phi_1^R, \phi_2^R$ satisfy $\phi_1^R = \phi_2^R \circ \psi^{-1}$. By Diagram~\ref{lax}, such a $\psi$ does not affect the lax monoidal structure of the adjunction, and hence $\phi_1$ and $\phi_2$ induce isomorphic commutative central algebras. Thus, we have a map
    $$
    \mathbb{A}: Y \to X ,\quad \quad(\phi: Vec_H^\sigma \to Vec_{H/F}^\alpha) \mapsto \phi^R(\mathbb{C}_e).
    $$

    By definition, $\mathbb{A}(\mathbb{F}(A)) = (T\circ F_A)^R(\mathbb{C}_e)\simeq F_A^R(T^{-1}(\mathbb{C}_e)) \simeq A$. Thus, we have $\mathbb{A}\circ \mathbb{F}=\mathrm{id}_X$. Conversely, for $(\phi:Vec_H^\sigma \to Vec_{H/F}^\alpha) \in Y$, we have $\mathbb{F}(\mathbb{A}(\phi))= T\circ F_{\phi^R(\mathbb{C}_e)}$, and we have the following commutative diagram up to monoidal natural isomorphism
    \begin{center}
        \begin{tikzcd}
            Vec_H^{\sigma} \arrow[r, "\phi"] \arrow[rd, "F_{\phi^R(\mathbb{C}_e)}"'] & Vec_{H/F}^{\alpha} \arrow[d, "{\simeq}"]                        &                    \\
                                                                                     & (Vec_{H}^{\sigma})_{\phi^R(\mathbb{C}_e)} \arrow[r, "{T,\simeq}"] & Vec_{H/F}^{\alpha}.
            \end{tikzcd}
    \end{center}
    where $K$ is the equivalence in Proposition \ref{exact}. Thus, $\phi$ and $\mathbb{F}(\mathbb{A}(\phi))$ differ by a monoidal autoequivalence of $Vec_{H/F}^\alpha$, and this autoequivalence is induced by the identity map on $H/F$. Therefore, we have $\mathbb{F}\circ \mathbb{A}=\mathrm{id}_Y$. This completes the proof of the bijection between $X$ and $Y$.
\end{proof}

\begin{remark}\label{remark:classs_cen}
    An element in the set $Y$ can also be represented by a pair $(\alpha, \nu)$, where $F_{\pi,\nu} : Vec_H^{\sigma} \to Vec_{H/F}^\alpha$ is the corresponding monoidal functor. According to \cite[Proposition~2.6.1]{etingof2016tensor}, the set of natural isomorphism classes of such functors (for fixed $\pi$, $\sigma$, and $\alpha$) is a torsor over $H^2(H, U(1))$. The monoidal autoequivalences of $Vec_{H/F}^\alpha$ induced by the identity on $H/F$ are classified by $H^2(H/F, U(1))$. Therefore, for a fixed $\alpha$, the set of equivalence classes of $\nu$ is a torsor over $\coker(\pi^*: H^2(H/F, U(1)) \to H^2(H, U(1)))$.
\end{remark}

We have now established a description of condensable algebras in $Z(Vec_G^\omega)$ in terms of quadruples $(H,F,\alpha,\phi)$. It remains to determine when two quadruples $(H,F,\alpha,\phi)$ and $(H',F',\alpha',\phi')$ correspond to isomorphic condensable algebras in $Z(Vec_G^\omega)$.

\begin{lemma}\label{lemma:send}

    Let $A$ and $A'$ be commutative separable central algebras in $Vec_H^{\sigma}$ and $Vec_{H'}^{\sigma'}$, with underlying objects $\bigoplus_{f \in F} \mathbb{C}_f$ and $\bigoplus_{f' \in F'} \mathbb{C}_{f'}$, corresponding to the pairs $(\alpha, \nu)$ and $(\alpha', \nu')$, respectively.
Given a monoidal equivalence
$F_{l,\zeta}: Vec_H^\sigma \simeq Vec_{H'}^{\sigma'}$,
where $l: H\simeq H'$ is a group isomorphism satisfying $l(F)=F'$, then the algebra $A'$ is isomorphic, as a commutative central algebra, to the image of $A$ under $F_{l,\zeta}$ if and only if there exists a monoidal lifting
$\theta: Vec_{H/F}^\alpha \to Vec_{H'/F'}^{\alpha'}$
of the induced equivalence $Vec_{H/F}\simeq Vec_{H'/F'}$ such that the following diagram commutes up to a monoidal natural isomorphism:

    \begin{equation}\label{quadruples_equ}
        \begin{tikzcd}
            Vec_H^{\sigma} \arrow[d, "F_{l,\zeta}"] \arrow[r, "{F_{\pi,\nu}}"] & Vec_{H/F}^\alpha \arrow[d, "\theta"] \\
            Vec_{H'}^{\sigma'}    \arrow[r, "{F_{\pi,\nu'}}"]           & Vec_{H'/F'}^{\alpha'} .              
            \end{tikzcd}
    \end{equation}

\end{lemma}

\begin{proof}
    If $F_{l,\zeta}(A)\simeq A'$ as commutative central algebra, then $(F_{\pi,\nu'}\circ F_{l,\zeta})^R(\mathbb{C}_e)\simeq F_{l,\zeta}^{-1}(F_{\pi,\nu'}^R(\mathbb{C}_e))\simeq A$. By Proposition \ref{exact}, there exist monoidal equivalences $K:Vec_{H/F}^\alpha\to (Vec_H^\sigma)_A$ and $K':Vec_{H'/F'}^{\alpha'}\to (Vec_H^\sigma)_A$. Thus, we have the following commutative diagram up to a monoidal natural isomorphism:
    \begin{center}
        \begin{tikzcd}
            Vec_H^{\sigma} \arrow[dd, "{F_{l,\zeta}}"'] \arrow[r, "{F_{\pi,\nu}}"] \arrow[rd, "F_A"] & Vec_{H/F}^{\alpha} \arrow[d, "{K,\simeq}"] \\
                                                                                                    & (Vec_H^\sigma)_A \arrow[d, "{K^{'-1},\simeq}"] \\
            Vec_{H'}^{\sigma'} \arrow[r, "{F_{\pi,\nu'}}"]                                          & Vec_{H'/F'}^{\alpha'} .                        
            \end{tikzcd}
    \end{center}
    Thus, the monoidal structure of $\theta$ is given by $K^{'-1}\circ K$.

    Conversely, if such a monoidal structure on $\theta$ exists, then by Proposition~\ref{exact}, the induced commutative central algebra of $\theta\circ F_{\pi,\nu}\simeq F_{\pi,\nu'}\circ F_{l,\zeta}$ is unique, thus $ A\simeq F_{\pi,\nu}^R(\theta^{-1}(\mathbb{C}_e))\simeq F_{l,\zeta}^{R}(F_{\pi,\nu'}(\mathbb{C}_e))\simeq F_{l,\zeta}^{-1}(A')$.
\end{proof}

\begin{lemma}\label{lemma:zetag}
   Let $H, H'$ be subgroups of $G$. Then an isomorphism $Fun(G/H) \simeq Fun(G/H')$ is encoded by an element $g \in G$ such that $H' = {}^gH := gHg^{-1}$. Moreover, the isomorphism encoded by $g$ induces the following commutative diagram:
    \begin{equation}\label{com_zetag}
        \begin{tikzcd}
Z(Vec_G^\omega)_{Fun(G/H)}^{loc} \arrow[d, "\simeq"] \arrow[r, "\simeq"] & Z(Vec_H^{\omega|_H}) \arrow[d, "\tilde{F}_g"] \arrow[r, "Forget"] & Vec_H^{\omega|_H} \arrow[d, "{F_{l_g,\zeta_g|_H}}"] \\
Z(Vec_G^\omega)_{Fun(G/H')}^{loc} \arrow[r, "\simeq"]                    & Z(Vec_{H'}^{\omega|_{H'}}) \arrow[r, "Forget"]                       & Vec_{H'}^{\omega|_H'}.                           
\end{tikzcd}
    \end{equation}
    Here, $l_g: H \to {}^g H$ given by $h \mapsto ghg^{-1}$ is the conjugation isomorphism, and $\tilde{F}_g$ denotes the lift of $F_{l_g, \zeta_g|_H}$. The cocycle class $\zeta_g$ is defined in Appendix~\ref{appendix:A}.
\end{lemma}
\begin{proof}
    An isomorphism between $\text{Fun}(G/H)$ and $\text{Fun}(G/H')$ is induced by a $G$-set isomorphism of $G/H$ and $G/H'$. The $G$-sets $G/H$ and $G/H'$ are isomorphic if and only if $H$ and $H'$ are conjugate in $G$. Specifically, for $g \in G$ such that $H' = {}^gH$, the isomorphism $\text{Fun}(G/H) \xrightarrow{\simeq} \text{Fun}(G/H')$ is defined on the basis of characteristic functions by $\delta_{tH} \mapsto \delta_{tg^{-1}H'}$, where $\delta_{tH}$ denotes the characteristic function of the coset $tH$.
    
    The reason that $F_{l_g, \zeta_g|_H}$ makes the above diagram commutative is presented in Appendix \ref{appendix:A}. 
\end{proof}

Based on the above discussion, we obtain the following classification theorem.

\begin{theorem}\label{class-algebra}

    Let $G$ be a finite group, $\omega \in H^3(G, U(1))$. Then the isomorphic classes of condensable algebras in $Z(Vec_G^\omega)$ are classified by the equivalent class of quadruples $(H,F,\alpha,\phi)$ where $F\trianglelefteq H \leqslant G $, $\alpha \in H^3(H/F, U(1))$ and $\phi: Vec_H^{\omega|_H}\rightarrow Vec_{H/F}^\alpha$ is a monoidal lifting of the forgetful functor $Vec_H \rightarrow Vec_{H/F}$.

    Two quadruples $(H,F,\alpha,\phi)$ and $(H',F',\alpha',\phi')$ are equivalent if and only if there exists a group element $g\in G$ such that $H' = {}^gH$, $F' = {}^gF$, and exists a monoidal lifting $\theta: Vec_{H/F}^\alpha \simeq Vec_{H'/F'}^{\alpha'}$ of the induced equivalence $Vec_{H/F} \simeq Vec_{H'/F'}$ such that the following diagram commutes up to monoidal natural isomorphism:
    \begin{equation}\label{quadruples_equdig}
        \begin{tikzcd}
            Vec_H^{\omega|_H} \arrow[d, "F_{l_g,\zeta_g|_H}"] \arrow[r, "\phi"] & Vec_{H/F}^\alpha \arrow[d, "\theta"] \\
            Vec_{H'}^{\omega|_{H'}}    \arrow[r, "{\phi'}"]           & Vec_{H'/F'}^{\alpha'} .              
            \end{tikzcd}
    \end{equation}
    
\end{theorem}
\begin{proof}
    
    We only need to prove the equivalence condition for two quadruples. Suppose that two quadruples $(H,F,\alpha,\phi)$ and $(H',F',\alpha',\phi')$ are equivalent. Then their degree-$e$ parts is isomorphic, i.e., $\text{Fun}(G/H) \simeq \text{Fun}(G/H')$. Thus, by Lemma \ref{lemma:zetag}, there exists $g \in G$ such that $H' = {}^g H$ and the commutative diagram \ref{com_zetag} holds. Consider the commutative separable central algebra $A \in Vec_H^{\omega|_H}$ represented by $(F, \alpha, \phi)$. Let $A' = F_{l_g, \zeta_g|_H}(A) \in Vec_{H'}^{\omega|_{H'}}$. Without loss of generality, we can assume that $A'$ is represented by $(F'={}^g F, \alpha', \phi')$. By Lemma \ref{lemma:send}, there exists a monoidal lift $\theta: \text{Vec}_{H/F}^\alpha \to \text{Vec}_{H'/F'}^{\alpha'}$ of the induced equivalence $\text{Vec}_{H/F} \simeq \text{Vec}_{H'/F'}$ such that Diagram \ref{quadruples_equdig} commutes up to monoidal natural isomorphism.
    
    On the other hand, suppose there exists $g \in G$ such that $H' = {}^gH$, $F' = {}^gF$, and there is a monoidal lift $\theta: Vec_{H/F}^\alpha \to Vec_{H'/F'}^{\alpha'}$ of the induced equivalence $Vec_{H/F} \simeq Vec_{H'/F'}$ such that Diagram \ref{quadruples_equdig} commutes up to monoidal natural isomorphism. Then, by Lemma \ref{lemma:send}, the commutative separable central algebra represented by $(F', \alpha', \phi')$ is isomorphic to the image of the algebra represented by $(F, \alpha, \phi)$ under the monoidal equivalence $F_{l_g, \zeta_g|_H}$. It follows from Diagram \ref{com_zetag} that the two quadruples $(H, F, \alpha, \phi)$ and $(H', F', \alpha', \phi')$ are equivalent. 
\end{proof}
\begin{remark}
    Note that the monoidal lifting $\phi$ exists if and only if $\pi^*\alpha=\omega|_H$ and, if exists, the liftings form a torsor over $Z^2(H
    ,U(1))$. The monoidal lifting $\theta$ exists if and only if $l_g^*(\alpha')=\alpha$ and, if exists, the liftings form a torsor over $Z^2(H/F,U(1))$. 
\end{remark}
\begin{remark}
    We also use the quadruple $(H,F,\alpha,\nu)$ to denote a condensable algebra in $Z(Vec_G^\omega)$, where $\nu$ is the monoidal structure of $\phi=F_{\pi,\nu}$.
\end{remark}

\begin{remark}\label{remark:conact} 
    We can describe the conjugacy action of $G$ on quadruples more precisely. For $g\in G$, define the action of $g$ on a quadruple $(H,F,\alpha,\phi)$ by
$$
g.(H,F,\alpha,\phi):=({}^gH,\, {}^gF,\, \alpha^g,\, L_g(\phi)).
$$
Here $\alpha^g\in H^3({}^gH/{}^gF,U(1))$ is defined by $\alpha^{g}(gxg^{-1},gyg^{-1},gzg^{-1})=\alpha(x,y,z)$ for $x,y,z\in H/F$. In fact, $\alpha^g=(l_g^*)^{-1}(\alpha)$. The action of $g$ on $\phi = F_{\pi, \nu}$ is given by $L_g(\phi) = F_{\pi, \nu'}$, where $\nu' = (l_g^*)^{-1}(\nu \cdot (\zeta_g|_H)^{-1})$.

The relevant part of Theorem \ref{class-algebra} is equivalent to saying that two quadruples $(H,F,\alpha,\phi)$ and $(H',F',\alpha',\phi')$ are equivalent if and only if there exists $g \in G$ such that $({}^gH, {}^gF, \alpha^g, [L_g(\phi)]) = (H', F', \alpha', [\phi'])$. For simplicity, we shall henceforth identify $\phi$ with its equivalence class, allowing the equivalence condition to be written more compactly as $g.(H,F,\alpha,\phi) = (H', F', \alpha', \phi')$.
\end{remark}

\begin{remark}\label{rmk:conj}
    When $G$ is Abelian, the conjugacy problem for $(H,F,\alpha)$ is vacuous, and the classification becomes much easier. Two quadruples $(H,F,\alpha,\phi)$ and $(H',F',\alpha',\phi')$ are equivalent only if $H=H'$, $F=F'$, and $\alpha=\alpha'$. Moreover, if $\omega=1$, then the conjugacy problem for $\phi$ is also vacuous.
    
    When $G$ is not Abelian, the group 
$N = \{g \in G \mid H = {}^gH, \, F = {}^gF\}$ 
may act nontrivially on the set 
$A = \{\alpha \in H^3(H/F, U(1)) \, | \pi^*\alpha = \omega|_H\}$. 
The equivalence classes of quadruples are classified by the conjugacy classes of $(H,F)$, then by the set $A/N$ for fixed $(H,F)$, then by the torsor over $\coker(H^2(H/F,U(1)) \to H^2(H,U(1)))$ quotient the stabilizer subgroup $N_\alpha \subset N$ for fixed $(H,F,\alpha)$. 

\end{remark}

\begin{remark}\label{remark:class}
    We have the following useful special cases:

    \begin{enumerate}
        \item When $H = F$, $\alpha$ must be trivial and the quadruple $(H, H, 1,\phi)$ represents a Lagrangian
        algebra. In this case, Lagrangian algebras of the form $(H, H, 1,\phi)$ are classified by a torsor over $H^2(H,U(1))$ quotient $N_G(H)$. 
        \item When $F = 1$ is the trivial group, the equivalence class of the quadruples for a fixed $H$ is unique. A canonical choice in this case is $\alpha = \omega|_H$ and $\phi$ as the identity functor.
    \end{enumerate}
\end{remark}

\begin{remark}\label{gammaepslion}
    In \cite[Proposition~2.15]{davydov2017lagrangian}, commutative central algebras $A=\oplus_f\mathbb{C}_f$ in $Vec_H^\omega$ are parametrized by pairs $(\gamma,\epsilon)$. Here $\gamma \in C^2(F,U(1))$ specifies the multiplication of $A$ and satisfies
$d\gamma=\omega|_F$, and $\epsilon$ is a function
$\epsilon:H\times F\rightarrow U(1)$ encoding the half-braiding of $A$, namely
$\epsilon_V: V\otimes A\to A\otimes V$. The data $(\gamma,\epsilon)$ satisfy the following conditions:

        \begin{equation}\label{da1}
\epsilon_{gh}(f)=\epsilon_g(hfh^{-1})\epsilon_h(f)\omega(g,h|f), \quad g,h\in H,f\in F
\end{equation}
\begin{equation}\label{da2}
    \gamma(f,g)\epsilon_h(fg)=\omega(h|f,g)\epsilon_h(f)\epsilon_h(g)\gamma(hfh^{-1}, hgh^{-1})\quad h\in H,f,g \in F
\end{equation}
\begin{equation}\label{da3}
    \gamma(f,g)=\epsilon_f(g)\gamma(fgf^{-1},f).
\end{equation}
Here 
$$
\omega(g,h|f):= \omega(g,h,f)^{-1}\omega(g,hfh^{-1},h)\omega(ghf(gh)^{-1},g,h)^{-1},
$$
$$
\omega(h|f,g):= \omega(h,f,g)\omega(hfh^{-1},h,g)^{-1}\omega(hfh^{-1},hgh^{-1},h).
$$
By Propositions~\ref{exact} and~\ref{prop:classs_cen}, we can choose a representative $\hat{\nu}$ of $\nu$ in a torsor over
$\coker(\pi^*: H^2(H/F, U(1)) \rightarrow H^2(H, U(1)))$
to express $\gamma$ and $\epsilon$ as follows:

$$
\gamma(f_1,f_2)=\hat{\nu}(f_1,f_2)^{-1},
$$
$$
\epsilon_h(f)=\hat{\nu}(hfh^{-1},h)\hat{\nu}(h,f)^{-1}.
$$

\end{remark}

\subsection{The embedding relationships}\label{section5.embedding}

As discussed in Section~\ref{section4.bimod}, embeddings between condensable algebras play an important role. We will classify these embeddings using the classification of condensable algebras developed in Section~\ref{section5.class}.

\begin{remark}
By \cite[Proposition 2.3.2]{davydov2010modular} and \cite[Lemma 3.1]{davydov2013structure}, the existence of an embedding of a condensable algebra $A$ into a condensable algebra $B$ is equivalent to the existence of a condensable algebra $B'\in \m{C}_A^{loc}$ such that $B$ is isomorphic, as an algebra in $\m{C}$, to the image of $B'$ under the lax forgetful functor $\m{C}_A^{loc}\to \m{C}$.
\end{remark}

\begin{definition}
    Two embeddings of condensable algebras $A \to B$ and $A \to C$ in a braided fusion category $\m{C}$ are {\it equivalent} if there is an algebra isomorphism $B \simeq C$ over $A$ or, equivalently, if the corresponding condensable algebras $B'$ and $C'$ are isomorphic in $\m{C}_A^{loc}$ as algebras.
\end{definition}
\begin{remark}
    For brevity, we sometimes simply say "embedding" instead of "equivalence class of embeddings". 
\end{remark}

\begin{definition}
    For condensable algebras $A$ and $B$ in a braided fusion category $\m{C}$, if there is more than one equivalent embedding class from $A$ into $B$ (i.e., $B$ corresponds to more than one condensable algebra in $\m{C}_A^{loc}$), then we say that $B$ is {\it split} under the condensation of $A$.
\end{definition}

\begin{proposition}\cite{frohlich2006correspondences}\label{twostep}
    Let $\m{C}$ be a braided fusion category, and let $A$ and $B$ be condensable algebras in $\m{C}$ such that $A$ can be embedded into $B$. Then we have $\m{C}_B^{loc} \simeq (\m{C}_A^{loc})_B^{loc}$ as braided fusion categories.
\end{proposition}

This proposition implies that the result of a single-step condensation of $B$ coincides with the result of a two-step condensation: first condensing $A$ and then $B$.

Given two condensable algebras in $Z(Vec_{G}^{\omega})$ represented by $(H,F,\alpha,\phi)$ and $(K,E,\beta,\psi)$, our goal is to determine when an embedding from one into the other exists, and if so, to count the number of equivalence classes of such embeddings. We begin by investigating the case when $F=1$.

\begin{lemma}\label{lemma:epartembed}
    Let $A,B$ be two condensable algebras in $Z(Vec_G^\omega)$ represented by $(H,1,\omega|_H,id_{Vec_H^{\omega|_H}})$ and $(K,E,\beta,\psi)$, respectively. Then an embedding from $A$ to $B$ is given by a condensable algebra in $Z(Vec_H^{\omega|_H})$ represented by $(K',E',\beta',\psi')$ such that $B\simeq (K',E',\beta', \psi')$ in $Z(Vec_G^\omega)$. Namely, there exists $g\in G$ such that ${}^gK\leqslant H$ and $g.(K,E,\beta,\psi)=(K',E',\beta',\psi')$.
\end{lemma}
\begin{proof}
    According to \cite[Corollary 2.13, Lemma 2.14, and Property 2.15]{davydov2017lagrangian} and Theorem \ref{class-algebra}, we have $A = Fun(G/H)$ and $Z(Vec_G^\omega)_{A}^{loc} \simeq Z(Vec_H^{\omega|_H})$, and the image of $(K',E',\beta',\psi') \in Z(Vec_H^{\omega|_H})$ under the forgetful functor $Z(Vec_G^\omega)_{A}^{loc} \to Z(Vec_G^\omega)$ can also be represented by the same quadruple $(K',E',\beta',\psi')$. Thus, an embedding from $A$ to $B$ is determined by an element $g \in G$ such that ${}^gK \leqslant H$ and $g.(K,E,\beta,\psi) = (K',E',\beta',\psi')$.
\end{proof}

We now turn to the investigation of embeddings between condensable algebras sharing the isomorphic degree $e$ part.

\begin{lemma}\label{lemma:eta}
    Let $H$ be a finite group and $\sigma\in H^3(H,U(1))$. Let $A$ and $B$ be commutative central algebras in $Vec_H^\sigma$ with underlying objects $\oplus_{f\in F}\mathbb{C}_f$ and $\oplus_{f\in E}\mathbb{C}_f$, corresponding to the pairs $(\alpha,\phi)$ and $(\beta,\psi)$, respectively. Then, there exists an embedding from $A$ to $B$ in $Z(Vec_H^\sigma)$ if and only if $F \trianglelefteq E \trianglelefteq  H$ and there exists a monoidal functor $\eta$ lifting the functor $Vec_{H/F}^{\alpha} \to Vec_{H/E}^{\beta}$ such that the following diagram commutes up to a monoidal natural isomorphism:
    \begin{equation}\label{eta}
        \begin{tikzcd}
            Vec_H^{\sigma} \arrow[r, "\phi"] \arrow[rd, "\psi"'] & Vec_{H/F}^{\alpha} \arrow[d, "\eta"] \\
                                                                 & Vec_{H/E}^{\beta}            .       
            \end{tikzcd}
    \end{equation} 
\end{lemma}
\begin{proof}
    When $\eta$ exists, we obtain an algebra homomorphism 
$$
A=\phi^R(\mathbb{C}_e) \stackrel{\phi^R(u_{\eta^R(\mathbb{C}_e)})}{\longrightarrow} \phi^R(\eta^R(\mathbb{C}_e)) =B.
$$
Here, $u_{\eta^R(\mathbb{C}_e)}$ is the unit of the algebra $\eta^R(\mathbb{C}_e) \in Vec_{H/F}^{\alpha}$. It can be verified that $\phi^R(u_{\eta^R(\mathbb{C}_e)})$ is also a morphism in $Z(Vec_H^\sigma)$; thus, we obtain an algebra homomorphism from $A$ to $B$ in $Z(Vec_H^\sigma)$.

Conversely, the existence of an embedding from $A$ to $B$ implies that $F \trianglelefteq E \trianglelefteq H$ and that $B$ is the image of a commutative algebra $B'$ in $(Vec_H^\sigma)_A \simeq Vec_{H/F}^{\alpha}$ under the forgetful functor. Furthermore, $B' \in Vec_{H/F}^{\alpha}$, whose underlying object is $\oplus_{\bar{f} \in E/F} \mathbb{C}_{\bar{f}}$, corresponds to the pair $(\beta', \eta)$ associated with $Vec_{H/F}^{\alpha}$. Since $\phi^R(\eta^R(\mathbb{C}_e)) \simeq B = \psi^R(\mathbb{C}_e)$, it follows from Proposition \ref{prop:classs_cen} that Diagram \ref{eta} commutes.
\end{proof}

\begin{corollary}\label{sameepart}
    Let $A$ and $B$ be condensable algebras in $Z(Vec_G^{\omega})$ represented by $(H,F,\alpha,\phi)$ and $(H,E,\beta,\psi)$, respectively. Then an embedding from $A$ to $B$ is encoded by a pair $(g,\eta)$, where $g \in G$ satisfies $F \trianglelefteq {}^g E\trianglelefteq H$ and ${}^gH=H$, and $\eta$ is a monoidal functor such that the following diagram commutes up to a monoidal natural isomorphism:
    \begin{equation}\label{eta2}
        \begin{tikzcd}
            Vec_H^{\omega|_H} \arrow[r, "\phi"] \arrow[rd, "L_g(\psi)"'] & Vec_{H/F}^\alpha \arrow[d, "\eta"] \\
                                                                         & Vec_{H/{}^{g} E}^{\beta^g}      .  
            \end{tikzcd}
    \end{equation}
    
    Two such embeddings, encoded by $(g,\eta)$ and $(g',\eta')$, are equivalent if and only if the quadruples $(H/F, {}^{g}E/F, \beta^g, \eta)$ and $(H/F, {}^{g'}E/F, \beta^{g'}, \eta')$ are equivalent in $Z(Vec_{H/F}^\alpha)$.
\end{corollary}
\begin{proof}
    By Lemma \ref{lemma:epartembed}, $A$ and $B$ are images under the forgetful functor of the condensable algebras $A'$ and $B'$ in $Z(Vec_H^{\omega|_H})$, represented by the quadruples $(H,F,\alpha,\phi)$ and $(H,E,\beta,\psi)$, respectively. According to \cite[Corollary 4.5]{schauenburg2001monoidal}, we have $Z(Vec_H^{\omega|_H})_{A'}^{loc} \simeq Z((Vec_H^{\omega|_H})_{A'})$, whereby the proof of Lemma \ref{lemma:eta} implies that an embedding from $A'$ to $B'$ is given by $(H/F, E'/F, \beta', \eta) \in Z(Vec_{H/F}^{\alpha})$ such that $B' \simeq (H,E',\beta',\eta\circ \phi)$ in $Z(Vec_H^{\omega|_H})$. It then follows from Lemma \ref{lemma:epartembed} that an embedding from $A$ to $B$ is determined by $(H/F,E'/F,\beta',\eta) \in Z(Vec_{H/F}^{\alpha})$ and an element $g \in G$ satisfying $g.(H,E,\beta,\psi) = (H,E', \beta', \eta\circ \phi)$ in $Z(Vec_G^\omega)$. Consequently, Diagram \ref{eta2} commutes.
\end{proof}

\begin{lemma}\label{Lemma: for_fun}

Let $H$ be a finite group and $\sigma\in H^3(H,U(1))$. Let $A$ be a condensable algebra in $Z(Vec_H^\sigma)$ represented by $(H,F,\alpha,\phi)$. Suppose that $K$ is a subgroup satisfying $F \trianglelefteq K \leqslant  H$, and let $B=Fun((H/F)/(K/F))$ be the condensable algebra in $Z(Vec_{H/F})$. Then, under the lax braided monoidal functor
$$
U: Z(Vec_{H/F}^\alpha)\simeq Z(Vec_H^\sigma)_A^{loc}\stackrel{Forget}{\rightarrow}Z(Vec_H^\sigma),
$$
the image algebra $U(B)$ is represented by $(K,F,\alpha|_K,\phi|_K)$, where $\phi|_K$ serves as shorthand for $\phi|_{Vec_K^{\omega|_K}}$.

\end{lemma}

\begin{proof} 
    Denote by $\phi^R$ the right adjoint of $\phi$, and let $\nu:\phi(X)\otimes \phi(Y)\to \phi(X\otimes Y)$ be the monoidal structure of $\phi$. By \cite[Corollary 4.11,Corollary 6.28]{flake2024projection}, we have $U(B)\simeq \phi^R(B)$ in $Z(Vec_H^\sigma)$, with multiplication given by $\phi^R(B)\otimes \phi^R(B)\stackrel{\nu^{-1}}{\rightarrow}\phi^R(B\otimes B)\stackrel{\phi^R(m_B)}{\rightarrow}\phi^R(B)$, where $m_B$ is the multiplication of $B$. The half-braiding of $\phi^R(B)$ (i.e., the $H$ action on $\phi^R(B)$) is given by 
    $$ 
    X\otimes \phi^R(B)\stackrel{(u\otimes id)\circ \nu^{-1}}{\longrightarrow} \phi^R(\phi(X)\otimes B)\stackrel{\phi^R(\beta_{\phi(X)})}{\rightarrow}\phi^R(B\otimes \phi(X))\stackrel{((\nu^{-1})\circ(id\otimes u))^{-1}}{\longrightarrow} \phi^R(B)\otimes X 
    $$ 
    where $u$ is the unit of the adjoint pair $\phi\dashv \phi^R$ and $\beta$ is the half-braiding of $B$. The $H$ action on $\phi^R(B)_e$ agrees with the $H/F$ action on $B$. Therefore, we obtain $\phi^R(B)_e\simeq Fun(H/K)$.  
     
    Consider the equivalence $D:Z(Vec_H^\sigma)_{Fun(H/K)}^{loc}\simeq Z(Vec_K^{\sigma|_K})$. We have $D(\phi^R(B))= \oplus_{f\in F}\mathbb{C}_f$, with multiplication $D(\phi^R(B))_{f_1}\otimes D(\phi^R(B))_{f_2}\stackrel{\nu^{-1}(f_1,f_2)}{\longrightarrow} \mathbb{C}_{f_1f_2}$ and half-braiding $\mathbb{C}_k\otimes D(\phi^R(B))_f\stackrel{\nu(kfk^{-1},k)\nu^{-1}(k,f)}{\longrightarrow} D(\phi^R(B))_{kfk^{-1}}\otimes \mathbb{C}_k$. Thus, by Remark~\ref{gammaepslion}, the algebra $\phi^R(B)\in Z(Vec_H^\sigma)$ is represented by $(K,F,\alpha|_K,\phi|_K)$. 
\end{proof}

\begin{proposition}\label{prop:transrep}
    Let $A$ be a condensable algebra in $Z(Vec_G^\omega)$ represented by $(H,F,\alpha,\phi)$. If a condensable algebra $D \in Z(Vec_G^\omega)_A^{loc} \simeq Z(Vec_{H/F}^{\alpha})$ is represented by $(K/F,E/F,\beta,\eta)$, then its image $B=U(D)$ in $Z(Vec_G^\omega)$ under the forgetful functor $U: Z(Vec_G^\omega)_A^{loc} \to Z(Vec_G^\omega)$ can be represented by the quadruple $(K,E,\beta,\eta \circ \phi|_{{Vec_K}^{\omega|_K}})$.
\end{proposition}
\begin{proof}
    Denote by $A'$ and $B'$ the condensable algebras in $Z(Vec_H^{\omega|_H})$ such that $A \simeq U_1(A')$ and $B \simeq U_1(B')$, where $U_1: Z(Vec_H^{\omega|_H}) \to Z(Vec_G^\omega)$ is the forgetful functor. By the proof of Lemma \ref{lemma:epartembed}, finding the quadruple representation of $B$ reduces to finding the quadruple representation of $B'$.
    
    Consider $Fun((H/F)/(K/F)) \in Z(Vec_{H/F}^\alpha)$ and denote by $C$ its image $U_2(Fun((H/F)/(K/F))) \in Z(Vec_H^{\omega|_H})$, where $U_2: Z(Vec_{H/F}^{\alpha}) \simeq Z(Vec_H^{\omega|_H})_{A'}^{loc} \to Z(Vec_H^{\omega|_H})$ is the forgetful functor. By Lemma~\ref{Lemma: for_fun} $C$ is represented by $(K,F,\alpha|_K,\phi|_{K})$. We have the following commutative diagram:

    \begin{equation}\label{4forget}
        \begin{tikzcd}
Z(Vec_H^{\omega|_H})                                                              & Z(Vec_{H/F}^\alpha)\simeq Z(Vec_{H}^{\omega|_H})_{A'}^{loc} \arrow[l, "U_2"'] \\
Z(Vec_K^{\omega|_K})\simeq Z(Vec_H^{\omega|_H})_{Fun(H/K)}^{loc} \arrow[u, "U_3"] & Z(Vec_{K/F}^{\alpha|_{K/F}})\simeq Z(Vec_H^{\omega|_H})_C^{loc}\simeq Z(Vec_{H/F}^{\alpha})_{Fun((H/F)/(K/F))}^{loc} \arrow[u, "U_4"'] \arrow[l, "U_5"'].                   
\end{tikzcd}
    \end{equation}
Here, $U_2, U_3, U_4$, and $U_5$ are the forgetful functors. By Lemma \ref{lemma:epartembed} and Corollary \ref{sameepart}, we obtain the following commutative diagram:
\begin{center}
    \begin{tikzcd}
{(K,E,\beta,\eta\circ \phi|_K)=U_2(D)}  & {D=(K/F,E/F,\beta,\eta)} \arrow[l, "U_2"', maps to]                            \\
{(K,E,\beta,\eta\circ\phi|_K)} \arrow[u, "U_3", maps to] & {(K/F,E/F,\beta,\eta)} \arrow[u, "U_4"', maps to] \arrow[l, "U_5"', maps to].
\end{tikzcd}
\end{center}
Thus, $B = U(D) \simeq U_2 \circ U_1(D)$ can be represented by $(K,E,\beta,\eta \circ \phi|_K)$.
\end{proof}

Based on the above discussion, we obtain the following theorem describing the embedding relationships between condensable algebras in $Z(Vec_G^\omega)$.

\begin{theorem}\label{emb}
    Assume $A,B\in Z(Vec_G^\omega)$ are condensable algebras represented by quadruples $(H,F,\alpha,\phi)$ and $(K,E,\beta,\psi)$, respectively. An algebra embedding $A\to B$ is encoded by a pair $(g,\eta)$ where $g\in G$ is such that $F \trianglelefteq {}^gE \trianglelefteq {}^gK \leqslant H$, and $\eta$ is a monoidal structure on the induced functor $Vec_{{}^gK/F}^{\alpha|_{^g K/F}}\to Vec_{^gK/^gE}^{\beta^g}$ that makes the following diagram commutative up to a monoidal natural isomorphism
    \begin{equation}\label{diagram_emb}
        \begin{tikzcd}
            Vec_{{}^gK}^{\omega|_{{}^gK}} \arrow[r, "\phi|_{{}^gK}"] \arrow[rd, "L_{g}(\psi)"'] & Vec_{{}^gK/F}^{\alpha|_{{}^gK/F}} \arrow[d, "\eta"] \\
                                                                                                 & Vec_{{}^gK/{}^gE}^{\beta^g}       .                 
            \end{tikzcd}
    \end{equation} 

    The algebra embeddings encoded by $(g, \eta)$ and $(g', \eta')$ are equivalent if and only if the two quadruples $({}^gK/F, {}^gE/F,\beta^g, \eta)$ and $({}{}^{g'}K/F, {}^{g'}E/F, \beta^{g'}, \eta')$ associated to $Z(Vec_{H/F}^{\alpha})$ are equivalent.

\end{theorem}
\begin{proof}

By definition, an embedding from $A$ to $B$ is given by a condensable algebra $B' \in Z(Vec_G^\omega)_A^{loc} \simeq Z(Vec_{H/F}^{\alpha})$ such that $B$ is isomorphic to the image of $B'$ under the forgetful functor $U: Z(Vec_{H/F}^{\alpha}) \to Z(Vec_G^\omega)$. If $B'$ is represented by $(K'/F,E'/F,\beta',\eta)$, then by Proposition \ref{prop:transrep}, its image $U(B')$ can be represented by $(K',E',\beta',\eta \circ \phi|_{K'})$. Consequently, an embedding from $A$ to $B$ is determined by the data $(K'/F,E'/F,\beta',\eta) \in Z(Vec_{H/F}^{\alpha})$ together with an element $g \in G$ satisfying $g.(K,E,\beta,\psi) = (K',E',\beta',\eta \circ \phi|_{K'})$ in $Z(Vec_G^\omega)$. Therefore, Diagram \ref{diagram_emb} commutes.
\end{proof}

\begin{corollary}\label{embcor}
    For the embeddings from $A$ represented by $(H, F, \alpha, \phi)$ to $B$ represented by $(K, E, \beta, \psi)$ encoded by $(g,\eta)$, we fix $g$. We have the following cases:
    
    \begin{enumerate}
        \item If $F = {}^g E$, then the equivalent class of $\eta$ is unique.
    \item If $F = 1$, then $\phi$ is trivial. In this case, $\eta$ is completely determined by $\psi$, so $\eta$ is unique.
    
    \end{enumerate}
    \end{corollary}

\begin{lemma}
    For a short exact sequence of finite groups $1 \to N \stackrel{\iota}{\longrightarrow } G \stackrel{\pi}{\longrightarrow} K \to 1$, there is a natural exact sequence
    \begin{center}
        \begin{tikzcd}
            1 \arrow[r] & {H^2(K,\mathbb{Z})} \arrow[r, "\pi^*"] & {H^2(G,\mathbb{Z})} \arrow[r, "\iota^*"] & {H^2(N,\mathbb{Z})} \arrow[r, "d_3"] & {H^3(K,\mathbb{Z})} \arrow[r, "\pi^*"] & {H^3(G,\mathbb{Z})}.
            \end{tikzcd}
    \end{center} 
\end{lemma}
\begin{proof}
    Since $N$ is finite, $H^1(N,\mathbb{Z})=1$. The $E_3$-page of the spectral sequence associated to the short exact sequence is as follows

    \begin{center}
        \begin{tikzpicture}

    \matrix [matrix of math nodes,row sep=6mm]
    {
    2 & [5mm]  |(a)| H^2(N) &  [5mm]  & [5mm] & [5mm]  \\
    1&  1 & 1  & 1  & 1  \\
    0& \mathbb{Z} &  H^1(K) & H^2(K) & |(b)| H^3(K) \\
    & 0 & 1 & 2 & 3 \\
    };
    \tikzstyle{every node}=[midway,auto,font=\scriptsize]
    \draw[thick] (-2.7,-2) -- (-2.7,2.3) ;
    \draw[thick] (-3.5,-1.2) -- (4.0,-1.2) ;
    \draw[-stealth] (a) -- node {$d_3$} (b);
    \end{tikzpicture}
    \end{center}

    We have $E_{\infty}^{2,0}=H^2(K,\mathbb{Z})$, and we have $\ker (d_3)= E_4^{0,2}=E_{\infty}^{0,2}= \image (\pi^*)$. By the edge morphism, we have $\image(d_3)=\ker (\pi^*)$.
\end{proof}

\begin{corollary}
        For a short exact sequence $1\to N \stackrel{\iota}{\longrightarrow } G \stackrel{\pi}{\longrightarrow} K \to 1$ of finite Abelian groups, the induced homomorphism $\pi^*: H^3(K,\mathbb{Z})\to H^3(G,\mathbb{Z})$ is injective.
    \end{corollary}

\begin{proof}
    One has $H^2(G,\bZ)\simeq G$. Thus $d_3=0$.
\end{proof}
\begin{corollary}\label{cor_abelian}
    Let $G$ be a finite abelian group, and let
    $A, B \in Z(\mathrm{Vec}_G^\omega)$ be condensable
    algebras represented by the quadruples
    $(H, F, \alpha, \phi)$ and $(K, E, \beta, \psi)$,
    respectively. For each fixed $g$, there is at
    most one  embeddings
    $A \hookrightarrow B$ represented by $(g,\eta)$.
\end{corollary}
\begin{proof}
    According to Theorem \ref{emb}, when $G$ is Abelian, an algebra embedding $A \to B$ exists only if $F \trianglelefteq E \trianglelefteq K \leqslant H$ and $\pi^*\beta = \alpha|_{K/F}$. Since $K$ is abelian, $\pi^*: H^2(K/F,U(1)) \to H^2(K,U(1))$
    is injective. Hence, for each fixed $g$,
    the embedding, if it exists, is unique
    up to equivalence.
\end{proof}

\begin{remark}\label{rmk:nosplit}
If $G$ is abelian and $\omega=1$, then by Remark~\ref{rmk:conj} and Corollary~\ref{cor_abelian}, there is at most one embedding between any two condensable algebras in $Z(Vec_G^\omega)$. In other words, the splitting phenomenon does not occur.
\end{remark}

\begin{remark}

    The splitting phenomenon also occurs in non-group theoretical cases (see, for example, \cite{chen2020topological}). Consider the condensable algebras in the double Ising category $Z(\textbf{Ising})\simeq \textbf{Ising}\boxtimes \textbf{Ising}$. The simple objects in the $\textbf{Ising}$ category are denoted by $1$, $\sigma$, and $\psi$. The category $Z(\textbf{Ising})$ has a condensable algebra $A =(1 \boxtimes 1) \oplus (\psi \boxtimes \psi)$, for which we have  $Z(\textbf{Ising})_A^{loc}\simeq Z(Vec_{\mathbb{Z}_2})$. Furthermore, the Lagrangian algebra $(1 \boxtimes 1) \oplus (\psi \boxtimes \psi)\oplus (\sigma\boxtimes \sigma)\in Z(\textbf{Ising})$ splits into two distinct Lagrangian algebras $1\oplus e$ and $1\oplus m$ in $Z(Vec_{\mathbb{Z}_2})$.
\end{remark}

\begin{example}\label{common-subalgebras}

    We provide an example where a pair of condensable algebras may have more than one maximal common subalgebra. Let $G=S_5$. There are two Lagrangian algebras in $Z(Vec_{S_5})$ represented by $(\langle (12)\rangle, \langle (12)\rangle,1,1)$ and $(\langle (123)\rangle, \langle (123)\rangle,1,1)$, respectively. Here we use $1$ to denote the trivial $\alpha$ and $\phi$. The condensable algebras represented by $(\langle (12), (123) \rangle,1,1,1)$ and $(\langle (12), (345)\rangle,1,1,1)$ are common subalgebras. Since $\langle (12), (123) \rangle\simeq S_3 $ and $\langle (12), (345)\rangle\simeq \mathbb{Z}_6$, these two common subalgebras are distinct and maximal.

\end{example}
\section{Examples}\label{section6}

In this section, we classify the condensable algebras in $Z(Vec_G^\omega)$
for several groups $G$ and describe their embedding relations. To facilitate
the classification, we first perform a coarse classification according to their Frobenius--Perron dimensions. For a non-degenerate braided fusion category
$\mathcal C$ and a condensable algebra $A$ in $\mathcal C$, we have $
\fpdim(\mathcal C_A^{loc})
=
\fpdim(\mathcal C)/\fpdim(A)^2$\cite{davydov2013witt}.

It is well known that $\fpdim(Z(Vec_G^\omega))=|G|^2$. Let $A$ be a condensable
algebra in $Z(Vec_G^\omega)$ represented by $(H,F,\alpha,\phi)$. Then
$Z(Vec_G^\omega)_A^{loc}\simeq Z(Vec_{H/F}^\alpha)$, and hence
$$
\fpdim(A)^2
=
\frac{\fpdim(Z(Vec_G^\omega))}
{\fpdim(Z(Vec_{H/F}^\alpha))}
=
([G:H]\cdot |F|)^2.
$$

Given $F\trianglelefteq H\leq G $, the map $\pi^*: H^l(H/F,U(1)) \to H^l(H,U(1))$ can be computed using the Lyndon-Hochschild-Serre spectral sequence. By the natural isomorphism $H^l(G, U(1)) \simeq H^{l+1}(G, \mathbb{Z})$, problems concerning group cohomology with $U(1)$-coefficients can be translated into problems with $\mathbb{Z}$-coefficients. The structure of group cohomology rings with $\mathbb{Z}$-coefficients is relatively well understood. Therefore, in the computations of this section, we will primarily use group cohomology with $\mathbb{Z}$-coefficients.

\subsection{Examples of  Abelian group $G$}

When $G$ is an abelian group and $\omega=1$, we do not need to consider the conjugacy problem. By Remark~\ref{rmk:nosplit}, the splitting phenomenon does not occur in the example considered in this subsection.

\subsubsection{$G=\bZ_{p^2}$}

Let $p$ be a prime number. Let $1$ denote the trivial group. According to Theorem~\ref{class-algebra} and Corollary~\ref{embcor}, we list all condensable algebras in $Z(Vec_{\mathbb{Z}_{p^2}})$ in Table~\ref{tab:Zp2}. The case $p=2$ was calculated in \cite{Chatterjee_2023}.

\begin{table}[!h]
    \centering
    
\begin{tabular}{|p{0.3\textwidth}|p{0.3\textwidth}|p{0.3\textwidth}|}
\hline 
dimension & $\displaystyle ( H,F)$ & number \\
\hline 
$\displaystyle 1$ & $\displaystyle (\mathbb{Z}_{P^{2}} ,1)$ & $\displaystyle 1$ \\
\hline 
\multirow{2}{*}{$\displaystyle p$} & $\displaystyle (\mathbb{Z}_{p^{2}} ,\mathbb{Z}_{p})$ & $\displaystyle p$ \\
\cline{2-3} 
& $\displaystyle (\mathbb{Z}_{p} ,1)$ & $\displaystyle 1$ \\
\hline 
\multirow{3}{*}{$\displaystyle p^{2}$} & $\displaystyle (\mathbb{Z}_{p^{2}} ,\mathbb{Z}_{p^{2}})$ & $\displaystyle 1$ \\
\cline{2-3} 
& $\displaystyle (\mathbb{Z}_{p} ,\mathbb{Z}_{p})$ & $\displaystyle 1$ \\
\cline{2-3} 
& $\displaystyle ( 1,1)$ & $\displaystyle 1$ \\
\hline
\end{tabular}
    \caption{Condensable algebras in $Z(Vec_{\mathbb{Z}_{p^2}})$}
    \label{tab:Zp2}
    \end{table}

    The only nontrivial case arises when $(H,F)=(\mathbb{Z}_{p^2},\mathbb{Z}_p)$. We need to compute 
    $
    \ker(\pi^*: H^3(\mathbb{Z}_p,U(1))\to H^3(\mathbb{Z}_{p^2},U(1)))
    $
    and
    $
    \coker(\pi^*: H^2(\mathbb{Z}_p,U(1))\to H^2(\mathbb{Z}_{p^2},U(1))).
    $
    This is equivalent to computing
    $
    \ker(\pi^*: H^4(\mathbb{Z}_p,\mathbb{Z})\to H^4(\mathbb{Z}_{p^2},\mathbb{Z}))
    $
    and
    $
    \coker(\pi^*: H^3(\mathbb{Z}_p,\mathbb{Z})\to H^3(\mathbb{Z}_{p^2},\mathbb{Z})).
    $
    
    It is well known that the group cohomology ring of $\mathbb{Z}_n$ with integer coefficients is
    $$
    H^*(\mathbb{Z}_n,\mathbb{Z}) \cong \mathbb{Z}[x]/(n x), \quad \text{with } |x|=2.
    $$
    We use the same notation for spectral sequences as in \cite{hatcher2002algebraic}. Consider the short exact sequence
$$
1 \to \mathbb{Z}_p \stackrel{\iota}{\to} \mathbb{Z}_{p^2} \stackrel{\pi}{\to} \mathbb{Z}_p \to 1.
$$
This gives rise to a convergent first quadrant spectral sequence
$
E_2^{m,n} = H^m(\mathbb{Z}_p, H^n(\mathbb{Z}_p, \mathbb{Z})),
$
and by computation we have
$
E_4^{m,n} = E_{\infty}^{m,n}.
$

Consider the filtration on $H^2(\mathbb{Z}_{p^2}, \mathbb{Z})$ induced by the spectral sequence. We obtain the short exact sequence
$$
1 \to \mathbb{Z}_p = E_{\infty}^{2,0} \to H^2(\mathbb{Z}_{p^2}, \mathbb{Z}) \to E_{\infty}^{0,2} = \mathbb{Z}_p \to 1.
$$
Let $x$ denote a generator of $H^2(\mathbb{Z}_p, \mathbb{Z})$, and $z$ a generator of $H^2(\mathbb{Z}_{p^2}, \mathbb{Z})$. Then we have
$$
\pi^*: H^2(\mathbb{Z}_p, \mathbb{Z})\to H^2(\mathbb{Z}_{p^2}, \mathbb{Z}), \quad x\mapsto p z.
$$
Using the ring structure of $H^*(\bZ_n,\bZ)$, we have

$$
\pi^*: H^4(\mathbb{Z}_p, \mathbb{Z})\to H^4(\mathbb{Z}_{p^2}, \mathbb{Z}), \quad x^2\mapsto p^2 z^2.
$$

So $
    \ker(\pi^*: H^4(\mathbb{Z}_p,\mathbb{Z})\to H^4(\mathbb{Z}_{p^2},\mathbb{Z})) = \bZ_p
    $
    and
    $
    \coker(\pi^*: H^3(\mathbb{Z}_p,\mathbb{Z})\to H^3(\mathbb{Z}_{p^2},\mathbb{Z})) = 1.
    $ This implies that for $(\mathbb{Z}_{p^2}, \mathbb{Z}_p, \alpha, \phi)$, there are $p$ choices for $\alpha$, and for each $\alpha$, there is a unique choice of $\phi$.

     Using Theorem~\ref{emb} and Corollary~\ref{embcor}, we obtain the embedding relations of condensable algebras as shown in Figure~\ref{Fig:Zp2}.

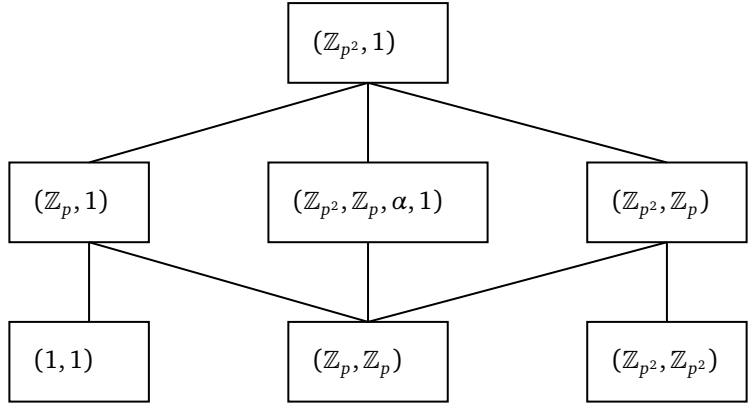
\begin{figure}[!h]
    \centering
    \tikzset{every picture/.style={line width=0.75pt}} 

    \begin{tikzpicture}[x=0.75pt,y=0.75pt,yscale=-1,xscale=1]
    
    \draw   (280,30) -- (360,30) -- (360,70) -- (280,70) -- cycle ;
    
    \draw   (430,190) -- (510,190) -- (510,230) -- (430,230) -- cycle ;
    
    \draw   (280,190) -- (360,190) -- (360,230) -- (280,230) -- cycle ;
    
    \draw   (430,110) -- (510,110) -- (510,150) -- (430,150) -- cycle ;
    
    \draw   (270,110) -- (380,110) -- (380,150) -- (270,150) -- cycle ;
    \draw   (140,110) -- (210,110) -- (210,150) -- (140,150) -- cycle ;
    \draw   (140,190) -- (210,190) -- (210,230) -- (140,230) -- cycle ;
    \draw    (180,150) -- (180,190) ;
    \draw    (320,70) -- (320,110) ;
    \draw    (470,150) -- (470,190) ;
    \draw    (320,150) -- (320,190) ;
    \draw    (320,70) -- (180,110) ;
    \draw    (320,70) -- (470,110) ;
    \draw    (180,150) -- (320,190) ;
    \draw    (470,150) -- (320,190) ;
    
    \draw (291,42) node [anchor=north west][inner sep=0.75pt]   [align=left] {$\displaystyle (\mathbb{Z}_{p^{2}} ,1)$};
    \draw (441,202) node [anchor=north west][inner sep=0.75pt]   [align=left] {$\displaystyle (\mathbb{Z}_{p^{2}} ,\mathbb{Z}_{p^{2}})$};
    \draw (291,202) node [anchor=north west][inner sep=0.75pt]   [align=left] {$\displaystyle (\mathbb{Z}_{p} ,\mathbb{Z}_{p})$};
    \draw (441,122) node [anchor=north west][inner sep=0.75pt]   [align=left] {$\displaystyle (\mathbb{Z}_{p^{2}} ,\mathbb{Z}_{p})$};
    \draw (281,122) node [anchor=north west][inner sep=0.75pt]   [align=left] {$\displaystyle (\mathbb{Z}_{p^{2}} ,\mathbb{Z}_{p} ,\alpha ,1)$};
    \draw (151,122) node [anchor=north west][inner sep=0.75pt]   [align=left] {$\displaystyle (\mathbb{Z}_{p} ,1)$};
    \draw (151,202) node [anchor=north west][inner sep=0.75pt]   [align=left] {$\displaystyle ( 1,1)$};

    \end{tikzpicture}
    
    \caption{The embedding relations of condensable algebras in $Z(Vec_{p^2})$}
    \label{Fig:Zp2}
\end{figure}

 We use $1$ to denote both the trivial $\alpha$ and the trivial $\phi$. If a condensable algebra has both $\alpha$ and $\phi$ trivial, we omit these labels in the diagram. Here, $\phi$ being trivial means that it is parametrized by a torsor over the trivial group. Each line in the diagram represents a unique embedding from the condensable algebra above to the one below.  We call two condensable algebras {\it adjacent} if there is no other condensable algebra between them. We only draw the embedding relations of adjacent condensable algebras. For convenience in drawing the diagram, we only depict one of the $p - 1$ nontrivial $(\mathbb{Z}_{p^2}, \mathbb{Z}_p, \alpha, 1)$.

\subsubsection{$G=\bZ_{p}\times \bZ_p $}

Let $p$ be a prime number. $\bZ_p\times \bZ_p$ has $p+1$ subgroups isomorphic to $\bZ_p$. In the following diagrams, we denote them as $\bZ_{p,(i)}, i\in \{1,2,...p+1\}$. We list all condensable algebras in $Z(Vec_{\mathbb{Z}_{p}\times \mathbb{Z}_p})$ in the Table~\ref{Tab:ZpZp}. The case $p=2$ was calculated in \cite{Chatterjee_2023,bhardwaj2403hasse}.

\begin{table}[!h]
    \centering
    
\begin{tabular}{|p{0.33\textwidth}|p{0.33\textwidth}|p{0.33\textwidth}|}
\hline 
dimension & $\displaystyle ( H,F)$ & number \\
\hline 
$\displaystyle 1$ & $\displaystyle (\mathbb{Z}_{p} \times \mathbb{Z}_{p} ,1)$ & $\displaystyle 1$ \\
\hline 
\multirow{2}{*}{$\displaystyle p$} & $\displaystyle (\mathbb{Z}_{p} \times \mathbb{Z}_{p} ,\mathbb{Z}_{p,(i)})$ & $\displaystyle  p$ \\
\cline{2-3} 
& $\displaystyle (\mathbb{Z}_{p,(i)} ,1)$ & $\displaystyle  1$ \\
\hline 
\multirow{3}{*}{$\displaystyle p^{2}$} & $\displaystyle (\mathbb{Z}_{p} \times \mathbb{Z}_{p} ,\mathbb{Z}_{p} \times \mathbb{Z}_{p})$ & $\displaystyle p$ \\
\cline{2-3} 
& $\displaystyle (\mathbb{Z}_{p,(i)} ,\mathbb{Z}_{p,(i)})$ & $\displaystyle   1$ \\
\cline{2-3} 
& $\displaystyle ( 1,1)$ & $\displaystyle 1$ \\
\hline
\end{tabular}
    \caption{Condensable algebras in $Z(Vec_{\mathbb{Z}_{p}\times \mathbb{Z}_p})$}
    \label{Tab:ZpZp}
    \end{table}

    The only nontrivial case arises when $(H,F)=(\mathbb{Z}_{p}\times \mathbb{Z}_p,\mathbb{Z}_p)$.  Using the Künneth formula, we know that for the generator $x^2 \in H^4(\mathbb{Z}_p,\bZ)$, its pullback $\pi^*(x^2)$ is also a generator of $H^4(\mathbb{Z}_p \times \mathbb{Z}_p, \bZ)$. So we have
    $$
    \ker(\pi^*: H^4(\mathbb{Z}_p,\mathbb{Z})\to H^4(\mathbb{Z}_{p}\times \bZ_p,\mathbb{Z}))=1
    $$
    and
    $$
    \coker(\pi^*: 1=H^3(\mathbb{Z}_p,\mathbb{Z})\to H^3(\mathbb{Z}_{p}\times \bZ_p,\mathbb{Z})=\bZ_p) =\bZ_p.
    $$

    Using Theorem~\ref{emb} and Corollary~\ref{embcor}, we obtain the embedding relations of condensable algebras as shown in Figure~\ref{Fig:ZpZp}.

\begin{figure}[!h]
\centering

\tikzset{every picture/.style={line width=0.75pt}} 

\begin{tikzpicture}[x=0.75pt,y=0.75pt,yscale=-0.9,xscale=0.9]

\draw   (290,50) -- (390,50) -- (390,90) -- (290,90) -- cycle ;

\draw   (300,210) -- (393.34,210) -- (393.34,250) -- (300,250) -- cycle ;
\draw   (70,210) -- (250,210) -- (250,250) -- (70,250) -- cycle ;
\draw   (370,130) -- (450,130) -- (450,170) -- (370,170) -- cycle ;

\draw   (190,130) -- (341,130) -- (341,170) -- (190,170) -- cycle ;
\draw   (470,210) -- (540,210) -- (540,250) -- (470,250) -- cycle ;
\draw    (420,170) -- (510,210) ;
\draw    (340,90) -- (280,130) ;
\draw    (270,170) -- (354,210) ;
\draw    (260,170) -- (170,210) ;
\draw    (340,90) -- (410,130) ;
\draw    (410,170) -- (354,210) ;

\draw (202,142) node [anchor=north west][inner sep=0.75pt]   [align=left] {$\displaystyle (\mathbb{Z}_{p} \times \mathbb{Z}_{p} ,\mathbb{Z}_{p,( i)} ,1,\phi )$};
\draw (482,222) node [anchor=north west][inner sep=0.75pt]   [align=left] {$\displaystyle ( 1,1)$};
\draw (82,220) node [anchor=north west][inner sep=0.75pt]   [align=left] {$\displaystyle (\mathbb{Z}_{p} \times \mathbb{Z}_{p} ,\mathbb{Z}_{p} \times \mathbb{Z}_{p} ,1,\phi )$};
\draw (381,142) node [anchor=north west][inner sep=0.75pt]   [align=left] {$\displaystyle (\mathbb{Z}_{p,(i)} ,1)$};
\draw (301,62) node [anchor=north west][inner sep=0.75pt]   [align=left] {$\displaystyle (\mathbb{Z}_{p} \times \mathbb{Z}_{p} ,1)$};
\draw (308,222) node [anchor=north west][inner sep=0.75pt]   [align=left] {$\displaystyle (\mathbb{Z}_{p,( i)} ,\mathbb{Z}_{p,( i)})$};

\end{tikzpicture}

    \caption{The embedding relations of condensable algebras in $Z(Vec_{\mathbb{Z}_{p}\times \mathbb{Z}_p})$}
    \label{Fig:ZpZp}
\end{figure}
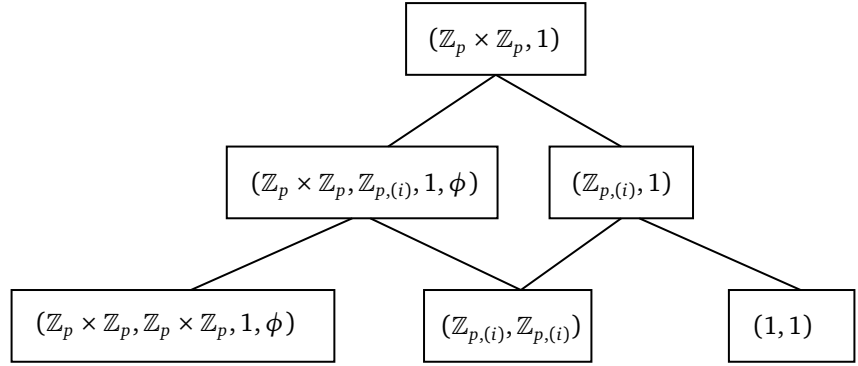

For simplicity, the diagram only shows one example from each type of similar condensable algebra.

\subsubsection{$G=\bZ_{pq}$}

Let $p,q$ be prime numbers and $p \neq q$.  Similar to the cases discussed earlier, we can compute using the spectral sequence that
$$
\pi^*:H^2(\mathbb{Z}_p,\mathbb{Z}) \to H^2(\mathbb{Z}_{pq},\mathbb{Z}),\quad x \mapsto qz,
$$
where $x$ and $z$ are generators. Using the structure of the cohomology ring, we obtain
$$
\pi^*:H^4(\mathbb{Z}_p,\mathbb{Z}) \to H^4(\mathbb{Z}_{pq},\mathbb{Z}),\quad x^2 \mapsto q^2 z^2,
$$
thus
$$
\ker(\pi^*: H^4(\mathbb{Z}_p,\mathbb{Z}) \to H^4(\mathbb{Z}_{pq},\mathbb{Z})) = 1,\quad
\coker(\pi^*: H^3(\mathbb{Z}_p,\mathbb{Z}) \to H^3(\mathbb{Z}_{pq},\mathbb{Z})) = 1.
$$

We list all condensable algebras in $Z(Vec_{\mathbb{Z}_{p}\times \mathbb{Z}_q})$ in Table~\ref{Tab:Zpq} and draw the diagram of embedding relations in Figure~\ref{Fig:Zpq}.

\begin{table}[!h]
    \centering
    
\begin{tabular}{|p{0.33\textwidth}|p{0.33\textwidth}|p{0.33\textwidth}|}
\hline 
dimension & $\displaystyle ( H,F)$ & number \\
\hline 
$\displaystyle 1$ & $\displaystyle (\mathbb{Z}_{pq} ,1)$ & $\displaystyle 1$ \\
\hline 
\multirow{2}{*}{$\displaystyle p$} & $\displaystyle (\mathbb{Z}_{pq} ,\mathbb{Z}_{p})$ & $\displaystyle 1$ \\
\cline{2-3} 
& $\displaystyle (\mathbb{Z}_{q} ,1)$ & $\displaystyle 1$ \\
\hline 
\multirow{2}{*}{$\displaystyle q$} & $\displaystyle (\mathbb{Z}_{pq} ,\mathbb{Z}_{q})$ & $\displaystyle 1$ \\
\cline{2-3} 
& $\displaystyle (\mathbb{Z}_{p} ,1)$ & $\displaystyle 1$ \\
\hline 
\multirow{4}{*}{$\displaystyle pq$} & $\displaystyle (\mathbb{Z}_{pq} ,\mathbb{Z}_{pq})$ & $\displaystyle 1$ \\
\cline{2-3} 
& $\displaystyle (\mathbb{Z}_{p} ,\mathbb{Z}_{p})$ & $\displaystyle 1$ \\
\cline{2-3} 
& $\displaystyle (\mathbb{Z}_{q} ,\mathbb{Z}_{q})$ & $\displaystyle 1$ \\
\cline{2-3} 
& $\displaystyle ( 1,1)$ & $\displaystyle 1$ \\
\hline
\end{tabular}
    \caption{Condensable algebras in $Z(Vec_{\mathbb{Z}_{pq}})$}
    \label{Tab:Zpq}
    \end{table}

\begin{figure}[!h]
    \centering
   \tikzset{every picture/.style={line width=0.75pt}} 

    \begin{tikzpicture}[x=0.75pt,y=0.75pt,yscale=-1,xscale=1]
    
    \draw   (300,50) -- (380,50) -- (380,90) -- (300,90) -- cycle ;
    \draw   (390,220) -- (470,220) -- (470,260) -- (390,260) -- cycle ;
    
    \draw   (240,220) -- (320,220) -- (320,260) -- (240,260) -- cycle ;
    
    \draw   (390,130) -- (470,130) -- (470,170) -- (390,170) -- cycle ;
    \draw   (240,120) -- (320,120) -- (320,160) -- (240,160) -- cycle ;
    \draw   (110,120) -- (190,120) -- (190,160) -- (110,160) -- cycle ;
    
    \draw   (520,130) -- (600,130) -- (600,170) -- (520,170) -- cycle ;
    
    \draw   (110,220) -- (190,220) -- (190,260) -- (110,260) -- cycle ;
    \draw   (520,220) -- (600,220) -- (600,260) -- (520,260) -- cycle ;
    \draw    (340,90) -- (150,120) ;
    \draw    (340,90) -- (280,120) ;
    \draw    (150,160) -- (150,220) ;
    \draw    (150,160) -- (430,220) ;
    \draw    (280,160) -- (560,220) ;
    \draw    (280,160) -- (280,220) ;
    \draw    (560,170) -- (560,220) ;
    \draw    (430,170) -- (280,220) ;
    \draw    (430,170) -- (150,220) ;
    \draw    (560,170) -- (430,220) ;
    \draw    (340,90) -- (560,130) ;
    \draw    (340,90) -- (430,130) ;
    
    \draw (121,132) node [anchor=north west][inner sep=0.75pt]   [align=left] {$\displaystyle (\mathbb{Z}_{pq} ,\mathbb{Z}_{p})$};
    \draw (251,232) node [anchor=north west][inner sep=0.75pt]   [align=left] {$\displaystyle (\mathbb{Z}_{q} ,\mathbb{Z}_{q})$};
    \draw (401,232) node [anchor=north west][inner sep=0.75pt]   [align=left] {$\displaystyle (\mathbb{Z}_{p} ,\mathbb{Z}_{p})$};
    \draw (311,62) node [anchor=north west][inner sep=0.75pt]   [align=left] {$\displaystyle (\mathbb{Z}_{pq} ,1)$};
    \draw (251,132) node [anchor=north west][inner sep=0.75pt]   [align=left] {$\displaystyle (\mathbb{Z}_{q} ,1)$};
    \draw (401,142) node [anchor=north west][inner sep=0.75pt]   [align=left] {$\displaystyle (\mathbb{Z}_{pq} ,\mathbb{Z}_{q})$};
    \draw (531,142) node [anchor=north west][inner sep=0.75pt]   [align=left] {$\displaystyle (\mathbb{Z}_{p} ,1)$};
    \draw (121,232) node [anchor=north west][inner sep=0.75pt]   [align=left] {$\displaystyle (\mathbb{Z}_{pq} ,\mathbb{Z}_{pq})$};
    \draw (541,232) node [anchor=north west][inner sep=0.75pt]   [align=left] {$\displaystyle ( 1,1)$};

    \end{tikzpicture}
    \caption{The embedding relations of condensable algebras in $Z(Vec_{\mathbb{Z}_{pq}})$}
    \label{Fig:Zpq}
\end{figure}
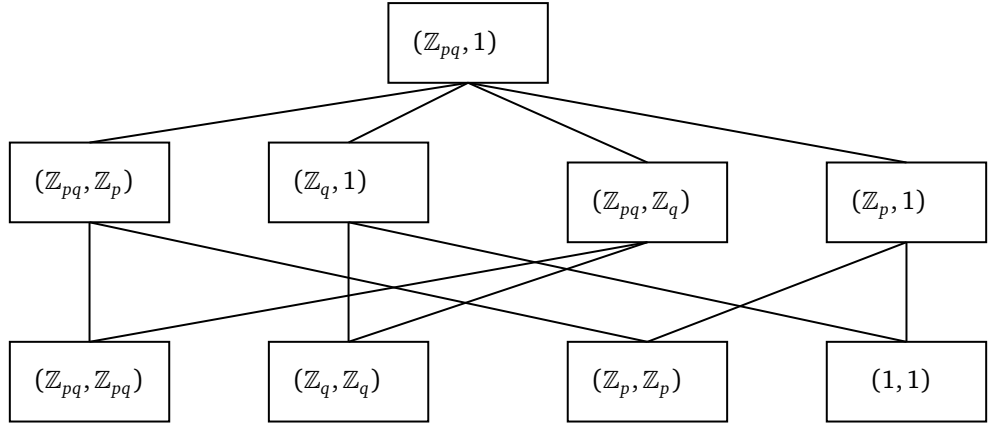

\subsection{Examples of non-Abelian group $G$}

Since $G$ is a non-Abelian group, we should consider the conjugacy problem. Our method is to first classify the pairs $(H,F)$ up to conjugacy, and then for a fixed $(H,F)$, we compute $\ker(\pi^*: H^3(H/F,U(1))\to H^3(H,U(1)))$ and $\coker(\pi^*: H^2(H/F,U(1))\to H^2(H,U(1)))$. Lastly, if these two groups do not have trivial elements (i.e., there are no trivial $\alpha$ or $\phi$ corresponding to $(H,F)$), then we try to find $g$ and $\theta$ that satisfy Diagram \ref{quadruples_equ} to classify the $\alpha$ and $\phi$ up to equivalence.
\subsubsection{$G=D_8$}

Let $D_8$ denote the dihedral group 
$
D_8 = \langle x, a \mid x^2 = a^4 = (ax)^2 = 1 \rangle.
$
It has the following subgroups:
$$
D_8,\ \langle a \rangle,\ \langle xa, a^2 \rangle,\ \langle x, a^2 \rangle,\ \langle a^2 \rangle,\ \langle x \rangle,\ \langle xa^2 \rangle,\ \langle xa \rangle,\ \langle xa^3 \rangle,\ 1.
$$
Among them, $\langle x \rangle$ and $\langle xa^2 \rangle$ are conjugate, and $\langle xa \rangle$ and $\langle xa^3 \rangle$ are conjugate. We list all condensable algebras in $Z(Vec_{D_8})$ in the Table~\ref{tableD8}. \cite{bhardwaj2403hasse} also computed the same case.

\begin{table}[!h]
    \centering      
\begin{tabular}{|p{0.3\textwidth}|p{0.37\textwidth}|p{0.3\textwidth}|}
\hline 
 dimension & $\displaystyle ( H,F)$ & number \\
\hline 
 $\displaystyle 1$ & $\displaystyle ( D_{8} ,1)$ & $\displaystyle 1$ \\
\hline 
 \multirow{4}{*}{$\displaystyle 2$} & $\displaystyle \left( D_{8} ,\langle a^{2} \rangle \right)$ & $\displaystyle 4$ \\
\cline{2-3} 
   & $\displaystyle ( \langle a \rangle ,1)$ & $\displaystyle 1$ \\
\cline{2-3} 
   & $\displaystyle \left( \langle xa,a^{2} \rangle ,1\right)$ & $\displaystyle 1$ \\
\cline{2-3} 
   & $\displaystyle \left( \langle x,a^{2} \rangle ,1\right)$ & $\displaystyle 1$ \\
\hline 
 \multirow{11}{*}{$\displaystyle 4$} & $\displaystyle ( D_{8} ,\langle a \rangle )$ & $\displaystyle 2$ \\
\cline{2-3} 
   & $\displaystyle \left( D_{8} ,\langle x,a^{2} \rangle \right)$ & $\displaystyle 2$ \\
\cline{2-3} 
   & $\displaystyle \left( D_{8} ,\langle xa,a^{2} \rangle \right)$ & $\displaystyle 2$ \\
\cline{2-3} 
   & $\displaystyle \left( \langle a^{2} \rangle ,1\right)$ & $\displaystyle 1$ \\
\cline{2-3} 
   & $\displaystyle ( \langle x \rangle ,1)$ or $\displaystyle \left( \langle xa^{2} \rangle ,1\right)$ & $\displaystyle 1$ \\
\cline{2-3} 
   & $\displaystyle ( \langle xa \rangle ,1)$ or $\displaystyle \left( \langle xa^{3} \rangle ,1\right)$ & $\displaystyle 1$ \\
\cline{2-3} 
   & $\displaystyle \left( \langle a \rangle ,\langle a^{2} \rangle \right)$ & $\displaystyle 2$ \\
\cline{2-3} 
   & $\displaystyle \left( \langle x,a^{2} \rangle ,\langle a^{2} \rangle \right)$ & $\displaystyle 2$ \\
\cline{2-3} 
   & $\displaystyle \left( \langle x,a^{2} \rangle ,\langle x \rangle \right)$ or \ $\displaystyle \left( \langle x,a^{2} \rangle ,\langle xa^{2} \rangle \right)$ & $\displaystyle 2$ \\
\cline{2-3} 
   & $\displaystyle \left( \langle xa,a^{2} \rangle ,\langle a^{2} \rangle \right)$ & $\displaystyle 2$ \\
\cline{2-3} 
   & $\displaystyle \left( \langle xa,a^{2} \rangle ,\langle xa \rangle \right)$or $\displaystyle \left( \langle xa,a^{2} \rangle ,\langle xa^{3} \rangle \right)$ & $\displaystyle 2$ \\
\hline 
 \multirow{8}{*}{$\displaystyle 8$} & $\displaystyle ( D_{8} ,D_{8})$ & $\displaystyle 2$ \\
\cline{2-3} 
   & $\displaystyle ( \langle a \rangle ,\langle a\rangle )$ & $\displaystyle 1$ \\
\cline{2-3} 
   & $\displaystyle \left( \langle x,a^{2} \rangle ,\langle x,a^{2} \rangle \right)$ & $\displaystyle 2$ \\
\cline{2-3} 
   & $\displaystyle \left( \langle xa,a^{2} \rangle ,\langle xa,a^{2} \rangle \right)$ & $\displaystyle 2$ \\
\cline{2-3} 
   & $\displaystyle \left( \langle a^{2} \rangle ,\langle a^{2} \rangle \right)$ & $\displaystyle 1$ \\
\cline{2-3} 
   & $\displaystyle ( \langle x \rangle ,\langle x\rangle )$ or $\displaystyle \left( \langle xa^{2} \rangle ,\langle xa^{2} \rangle \right)$ & $\displaystyle 1$ \\
\cline{2-3} 
   & $\displaystyle ( \langle xa \rangle ,\langle xa\rangle )$ or $\displaystyle \left( \langle xa^{3} \rangle ,\langle xa^{3} \rangle \right)$ & $\displaystyle 1$ \\
\cline{2-3} 
   & $\displaystyle ( 1,1)$ & $\displaystyle 1$ \\
 \hline
\end{tabular}
    \caption{The condensable algebras in $Z(Vec_{D_8})$}
    \label{tableD8}
    \end{table}

    For $(H,F)=(D_8,\langle a^2\rangle)$, we need to compute $\pi^*: H^l(D_8/\langle a^2\rangle,\bZ)\to H^l(D_8,\bZ), l=3,4$. Consider the short exact sequence: 
    $$
    1 \to \langle a^2\rangle=\mathbb{Z}_2 \stackrel{\iota}{\to} D_8 \stackrel{\pi}{\to} D_8/\langle a^2\rangle=\bZ_2\times \bZ_2 \to 1.
    $$
    This gives rise to a convergent first quadrant spectral sequence
    $
    E_2^{m,n} = H^m(\bZ_2\times \bZ_2, H^n(\bZ_2, \mathbb{Z})).
    $
    
    We have the following $E_2^{m,n}$ page in the spectral sequence:

    \begin{center}
        \begin{tikzpicture}

    \matrix [matrix of math nodes,row sep=6mm]
    {
    4 & [4mm] \bZ_2 &  [4mm]\bZ_2^2  & [4mm]\bZ_2^3 & [4mm] \bZ_2^4 & [4mm] \bZ_2^5 & [4mm] \bZ_2^6 \\
    3& 1     &   1     &  1      &   1         &   1      &  1  \\
    2 & \bZ_2 & \bZ_2^2 & \bZ_2^3 &   \bZ_2^4   & \bZ_2^5  &  \bZ_2^6  \\
    1 & 1     & 1       &  1      &   1         &   1      &  1  \\
    0& \bZ   & 1       & \bZ_2^2 &   \bZ_2     & \bZ_2^3  &  \bZ_2^2    \\
    & 0 & 1 & 2 & 3 & 4 & 5\\
    };
    \tikzstyle{every node}=[midway,auto,font=\scriptsize]
    \draw[thick] (-2.7,-3.3) -- (-2.7,3.3) ;
    \draw[thick] (-3.5,-2.4) -- (4.0,-2.4) ;
    \end{tikzpicture}
    \end{center}
    
    All differentials $d_2$ in this page are trivial. So $E_2=E_3$, $E_{\infty}^{2,0}=\bZ_2^2$. The group cohomology ring of $D_8$ is given in \cite{handel1993products}: 
    $$
    H^*(D_8,\bZ)= \bZ[a_2,b_2,c_3,d_4]/(2a_2,2b_2,2c_3,4d_4,b_2^2+a_2b_2,c_3^2+a_2d_4)
    $$
    where $|a_2|=|b_2|=2,|c_3|=3,|d_4|=4$. 
    
    We have the edge morpgism:
\begin{center}

    \begin{tikzcd}
        {E_2^{2,0}} \arrow[r, two heads]                                                         & {E_{\infty}^{2,0}=\mathbb{Z}_2^2} \arrow[d, hook] \\
        {\mathbb{Z}_2^2=H^2(D_8/\mathbb{Z}_2,\mathbb{Z})} \arrow[u, "\simeq"] \arrow[r, "\pi^*"] & {H^2(D_8,\mathbb{Z})=\mathbb{Z}_2^2}.            
        \end{tikzcd}
    
    \end{center}

Using the ring structure of $H^*(D_8,\bZ)$, we have
$$
\ker(\pi^*:\bZ_2^3=H^4(D_8/\bZ_2,\bZ)\to H^4(D_8,\bZ)=\bZ_2^2 \times \bZ_4)=\bZ_2.
$$

Due to the exact sequence
$$
1 \to E_{\infty}^{2,0} \to H^2(D_8,\mathbb{Z}) \to E_{\infty}^{0,2} \to 1,
$$
we obtain that $E_{\infty}^{0,2} = 1$. Therefore, the differential $d_3: E_3^{0,2} \to E_3^{3,0}$ on the $E_3$ page must be an isomorphism, which implies $E_{\infty}^{3,0} = 1$. As a result, we have the following edge morphism:
\begin{center}
    \begin{tikzcd}
        {\mathbb{Z}_2=E_2^{3,0}} \arrow[r, two heads]                                          & {E_{\infty}^{3,0}=1} \arrow[d, hook] \\
        {\mathbb{Z}_2=H^3(D_8/\mathbb{Z}_2,\mathbb{Z})} \arrow[u, "\simeq"] \arrow[r, "\pi^*"] & {H^3(D_8,\mathbb{Z})}.            
        \end{tikzcd}
\end{center}
So we have
$$
\coker(\pi^*:\bZ_2=H^3(D_8/\bZ_2,\bZ)\to H^3(D_8,\bZ)=\bZ_2)=\bZ_2.
$$
Thus, for $(H,F)=(D_8,\langle a^2\rangle)$, there are $4$ condensable algebras.

For $(H,F)=(D_8,\langle x,a^2\rangle)$, we consider the exact sequence
$$
1\to \langle x,a^2\rangle =\bZ_2\times \bZ_2 \to D_8 \to D_8/\langle x,a^2\rangle \to 1.
$$
This gives rise to a convergent first quadrant spectral sequence
$
E_2^{m,n} = H^m(\bZ_2 , H^n(\bZ_2\times \bZ_2, \mathbb{Z})).
$
It is important to note that the conjugation action of $D_8/\langle x,a^2\rangle$ on $\langle x,a^2\rangle$ is nontrivial. Therefore, $H^m(\mathbb{Z}_2 , H^n(\mathbb{Z}_2 \times \mathbb{Z}_2, \mathbb{Z}))$ is a group cohomology with local coefficients. A similar analysis of the spectral sequence as above yields:
$$
\ker(\pi^*:\bZ_2=H^4(D_8/\langle x,a^2\rangle,\bZ)\to H^4(D_8,\bZ)=\bZ_2^2\times \bZ_4 )=1,
$$
$$
\coker(\pi^*:1=H^3(D_8/\langle x,a^2\rangle,\bZ)\to H^3(D_8,\bZ)=\bZ_2)=\bZ_2.
$$
Thus, for $(H,F)=(D_8,\langle x,a^2\rangle)$, there are $2$ condensable algebras.

For $(H,F)=(D_8,\langle a\rangle)$, we need to consider the spectral sequence $E_2^{m,n}=H^m(\bZ_2 , H^n(\bZ_4, \mathbb{Z}))$. We have $E_2=E_{\infty}$\cite{chen2013symmetry}. A similar analysis of the spectral sequence as above yields:
$$
\ker(\pi^*:\bZ_2=H^4(D_8/\langle a\rangle,\bZ)\to H^4(D_8,\bZ)=\bZ_2^2\times  \bZ_4)=1,
$$
$$
\coker(\pi^*:1=H^3(D_8/\langle a\rangle,\bZ)\to H^3(D_8,\bZ)=\bZ_2)=\bZ_2.
$$
Thus, for $(H,F)=(D_8,\langle a\rangle)$, there are $2$ condensable algebras. The classification of the remaining types of condensable algebras is either trivial or has already been discussed in previous subsections.

\begin{remark}
    The condensable algebras discussed above all have $H=D_8$. Thus, by Remark \ref{remark:class}, the conjugacy problem is trivial. For a fixed pair $(H,F)$, if $\ker(\pi^*:H^4(H/F,\mathbb{Z})\to H^4(H,\mathbb{Z}))$ or $\coker(\pi^*:H^3(H/F,\mathbb{Z})\to H^3(H,\mathbb{Z}))$ is isomorphic to $\mathbb{Z}_2$, then the conjugacy actions of $G$ on these sets are trivial. This follows from the fact that the group of automorphisms of $\mathbb{Z}_2$ is only the trivial one. Consequently, when $(H,F)$ is fixed, the conjugacy problems for $Z(Vec_{D_8})$  are all trivial.
\end{remark}

\cite{bhardwaj2403hasse} obtained the embedding relations for the condensable algebras in $Z(Vec_{D_8})$.
However, they did not point out the splitting phenomenon.
We list the adjacent condensables which have more than one embedding in Table~\ref{tableD8adj}. In fact, these adjacent condensable algebras all have 2 embeddings. We write $1$ and $\phi$ to distinguish whether the fourth element of the quadruple is trivial or non-trivial. If the splitting phenomenon is due to conjugation of subgroups, we write "and" in the table to show the split.

\begin{table}[!h]
    \centering
    
\begin{tabular}{|p{0.43\textwidth}|p{0.57\textwidth}|}
\hline 
$\displaystyle ( H,F,\alpha ,\phi )$ & $\displaystyle ( H',F',\alpha ',\phi ')$ \\
\hline 
$\displaystyle \left( D_{8} ,\langle a^{2} \rangle ,1,1\right)$ & $\displaystyle ( D_{8} ,\langle a \rangle ,1,1)$ \\
\hline 
$\displaystyle \left( D_{8} ,\langle a^{2} \rangle ,1,1\right)$ & $\displaystyle \left( D_{8} ,\langle x,a^{2} \rangle ,1,1\right)$ \\
\hline 
$\displaystyle \left( D_{8} ,\langle a^{2} \rangle ,1,1\right)$ & $\displaystyle \left( D_{8} ,\langle xa,a^{2} \rangle ,1,1\right)$ \\
\hline 
$\displaystyle \left( D_{8} ,\langle a^{2} \rangle ,1,\phi \right)$ & $\displaystyle ( D_{8} ,\langle a \rangle ,1,\phi )$ \\
\hline 
$\displaystyle \left( D_{8} ,\langle a^{2} \rangle ,1,\phi \right)$ & $\displaystyle \left( D_{8} ,\langle x,a^{2} \rangle ,1,\phi \right)$ \\
\hline 
$\displaystyle \left( D_{8} ,\langle a^{2} \rangle ,1,\phi \right)$ & $\displaystyle \left( D_{8} ,\langle xa,a^{2} \rangle ,1,\phi \right)$ \\
\hline 
$\displaystyle \left( \langle xa,a^{2} \rangle ,1,1,1\right)$ & $\displaystyle ( \langle xa \rangle ,1,1,1)$ and $\displaystyle \left( \langle xa^{3} \rangle ,1,1,1\right)$ \\
\hline 
$\displaystyle \left( \langle xa,a^{2} \rangle ,1,1,1\right)$ & $\displaystyle \left( \langle xa,a^{2} \rangle ,\langle xa \rangle ,1,1\right)$ and $\displaystyle \left( \langle xa,a^{2} \rangle ,\langle xa^{3} \rangle ,1,1\right)$ \\
\hline 
$\displaystyle \left( \langle xa,a^{2} \rangle ,1,1,1\right)$ & $\displaystyle \left( \langle xa,a^{2} \rangle ,\langle xa \rangle ,1,\phi \right)$ and $\displaystyle \left( \langle xa,a^{2} \rangle ,\langle xa^{3} \rangle ,1,\phi \right)$ \\
\hline 
$\displaystyle \left( \langle x,a^{2} \rangle ,1,1,1\right)$ & $\displaystyle ( \langle x \rangle ,1,1,1)$ and $\displaystyle \left( \langle xa^{2} \rangle ,1,1,1\right)$ \\
\hline 
$\displaystyle \left( \langle x,a^{2} \rangle ,1,1,1\right)$ & $\displaystyle \left( \langle x,a^{2} \rangle ,\langle x \rangle ,1,1\right)$ and $\displaystyle \left( \langle x,a^{2} \rangle ,\langle xa^{2} \rangle ,1,1\right)$ \\
\hline 
$\displaystyle \left( \langle x,a^{2} \rangle ,1,1,1\right)$ & $\displaystyle \left( \langle x,a^{2} \rangle ,\langle x \rangle ,1,\phi \right)$ and $\displaystyle \left( \langle x,a^{2} \rangle ,\langle xa^{2} \rangle ,1,\phi \right)$ \\
\hline
\end{tabular}
\caption{The adjacent condensables in $Z(Vec_{D_8})$ with more than one embedding}
    \label{tableD8adj}
    
    \end{table}

\subsubsection{$G=\bZ_p\rtimes \bZ_2$}

Let $p \neq 2$ be a prime number. We list all condensable algebras in $Z(Vec_{\mathbb{Z}_{p}\rtimes \mathbb{Z}_2})$ in Table~\ref{Tab:ZpZ2}. The case $p=3$ was calculated in \cite{Chatterjee_2023,bhardwaj2403hasse}.

\begin{table}[!h]
    \centering
    
\begin{tabular}{|p{0.33\textwidth}|p{0.33\textwidth}|p{0.33\textwidth}|}
\hline 
dimension & $\displaystyle ( H,F)$ & number \\
\hline 
$\displaystyle 1$ & $\displaystyle (\mathbb{Z}_{p} \rtimes \mathbb{Z}_{2} ,1)$ & $\displaystyle 1$ \\
\hline 
$\displaystyle 2$ & $\displaystyle (\mathbb{Z}_{p} ,1)$ & $\displaystyle 1$ \\
\hline 
\multirow{2}{*}{$\displaystyle p$} & $\displaystyle (\mathbb{Z}_{p} \rtimes \mathbb{Z}_{2} ,\mathbb{Z}_{p})$ & $\displaystyle 1$ \\
\cline{2-3} 
& $\displaystyle (\mathbb{Z}_{2} ,1)$ & $\displaystyle 1$ \\
\hline 
\multirow{4}{*}{$\displaystyle 2p$} & $\displaystyle (\mathbb{Z}_{p} \rtimes \mathbb{Z}_{2} ,\mathbb{Z}_{p} \rtimes \mathbb{Z}_{2})$ & $\displaystyle 1$ \\
\cline{2-3} 
& $\displaystyle (\mathbb{Z}_{p} ,\mathbb{Z}_{p})$ & $\displaystyle 1$ \\
\cline{2-3} 
& $\displaystyle (\mathbb{Z}_{2} ,\mathbb{Z}_{2})$ & $\displaystyle 1$ \\
\cline{2-3} 
& $\displaystyle ( 1,1)$ & $\displaystyle 1$ \\
\hline
\end{tabular}
\caption{Condensable algebras in $Z(Vec_{\mathbb{Z}_{p}\rtimes \mathbb{Z}_2})$}
    \label{Tab:ZpZ2}
    \end{table}

    The group cohomology ring of $\bZ_p\rtimes \bZ_2$ is given in \cite{handel1993products}: 

    $$
    H^*(\bZ_p\rtimes \bZ_2,\bZ)=\bZ[a_2,d_4]/(2a_2,pd_4).
    $$

    The only potentially nontrivial case is when $(H,F) = (\mathbb{Z}_p \rtimes \mathbb{Z}_2, \mathbb{Z}_p)$. A similar spectral sequence analysis as above yields:

    $$
    \ker(\pi^*:\bZ_2=H^4(\bZ_2,\bZ)\to H^4(\bZ_p \rtimes \bZ_2,\bZ)=\bZ_2 \times\bZ_p)=1,
    $$
    $$
    \coker(\pi^*:1=H^3(\bZ_2,\bZ)\to H^3(\bZ_p \rtimes \bZ_2,\bZ)=\bZ_2)=1.
    $$

    Thus, for $(H,F)=(\mathbb{Z}_p \rtimes \mathbb{Z}_2, \mathbb{Z}_p)$, there is only $1$ condensable algebra.  Using Theorem~\ref{emb} and Corollary~\ref{embcor}, we obtain the embedding relations of condensable algebras as shown in Figure~\ref{Fig:ZpZ2}.

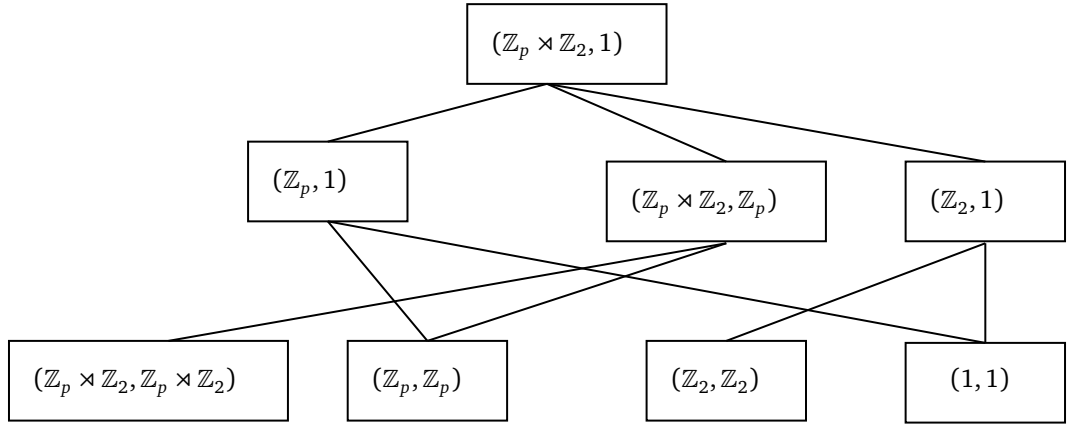
\begin{figure}[!h]
    \centering

\tikzset{every picture/.style={line width=0.75pt}} 

\begin{tikzpicture}[x=0.75pt,y=0.75pt,yscale=-1,xscale=1]

\draw   (320,31) -- (420,31) -- (420,71) -- (320,71) -- cycle ;
\draw   (410,200) -- (490,200) -- (490,240) -- (410,240) -- cycle ;

\draw   (260,200) -- (340,200) -- (340,240) -- (260,240) -- cycle ;

\draw   (390,110) -- (500,110) -- (500,150) -- (390,150) -- cycle ;
\draw   (210,100) -- (290,100) -- (290,140) -- (210,140) -- cycle ;

\draw   (540,110) -- (620,110) -- (620,150) -- (540,150) -- cycle ;

\draw   (90,200) -- (230,200) -- (230,240) -- (90,240) -- cycle ;
\draw   (540,201) -- (620,201) -- (620,241) -- (540,241) -- cycle ;
\draw    (360,71) -- (250,100) ;
\draw    (250,140) -- (300,200) ;
\draw    (250,140) -- (580,201) ;
\draw    (580,151) -- (580,201) ;
\draw    (450,151) -- (300,200) ;
\draw    (450,151) -- (170,200) ;
\draw    (580,151) -- (450,200) ;
\draw    (360,71) -- (580,110) ;
\draw    (360,71) -- (450,110) ;

\draw (331,43) node [anchor=north west][inner sep=0.75pt]   [align=left] {$\displaystyle (\mathbb{Z}_{p} \rtimes \mathbb{Z}_{2} ,1)$};
\draw (401,122) node [anchor=north west][inner sep=0.75pt]   [align=left] {$\displaystyle (\mathbb{Z}_{p} \rtimes \mathbb{Z}_{2} ,\mathbb{Z}_{p})$};
\draw (561,212) node [anchor=north west][inner sep=0.75pt]   [align=left] {$\displaystyle ( 1,1)$};
\draw (551,122) node [anchor=north west][inner sep=0.75pt]   [align=left] {$\displaystyle (\mathbb{Z}_{2} ,1)$};
\draw (221,112) node [anchor=north west][inner sep=0.75pt]   [align=left] {$\displaystyle (\mathbb{Z}_{p} ,1)$};
\draw (271,212) node [anchor=north west][inner sep=0.75pt]   [align=left] {$\displaystyle (\mathbb{Z}_{p} ,\mathbb{Z}_{p})$};
\draw (421,212) node [anchor=north west][inner sep=0.75pt]   [align=left] {$\displaystyle (\mathbb{Z}_{2} ,\mathbb{Z}_{2})$};
\draw (101,212) node [anchor=north west][inner sep=0.75pt]   [align=left] {$\displaystyle (\mathbb{Z}_{p} \rtimes \mathbb{Z}_{2} ,\mathbb{Z}_{p} \rtimes \mathbb{Z}_{2})$};

\end{tikzpicture}
\caption{The embedding relations of condensable algebras in $Z(Vec_{\mathbb{Z}_{p}\rtimes \mathbb{Z}_2})$}
    \label{Fig:ZpZ2}

\end{figure}

\subsubsection{$G=Q_8$}

Let $Q_8 = \{1, -1, i, -i, j, -j, k, -k\}$ be the quaternion group. The subgroups of $Q_8$ are $Q_8$, $I = \langle i \rangle$, $J = \langle j \rangle$, $K = \langle k \rangle$, $\langle -1 \rangle$, and ${1}$. All of these subgroups are normal. We list all condensable algebras in $Z(Vec_{Q_8})$ in the Table~\ref{Tab:Q8}.

\begin{table}[!h]
    \centering
    
\begin{tabular}{|p{0.33\textwidth}|p{0.33\textwidth}|p{0.33\textwidth}|}
\hline 
dimension & $\displaystyle ( H,F)$ & number \\
\hline 
$\displaystyle 1$ & $\displaystyle ( Q_{8} ,1)$ & $\displaystyle 1$ \\
\hline 
\multirow{4}{*}{$\displaystyle 2$} & $\displaystyle ( Q_{8} ,\langle -1\rangle )$ & $\displaystyle 4$ \\
\cline{2-3} 
& $\displaystyle ( I,1)$ & $\displaystyle 1$ \\
\cline{2-3} 
& $\displaystyle ( J,1)$ & $\displaystyle 1$ \\
\cline{2-3} 
& $\displaystyle ( K,1)$ & $\displaystyle 1$ \\
\hline 
\multirow{7}{*}{$\displaystyle 4$} & $\displaystyle ( Q_{8} ,I)$ & $\displaystyle 2$ \\
\cline{2-3} 
& $\displaystyle ( Q_{8} ,J)$ & $\displaystyle 2$ \\
\cline{2-3} 
& $\displaystyle ( Q_{8} ,K)$ & $\displaystyle 2$ \\
\cline{2-3} 
& $\displaystyle ( I,\langle -1\rangle )$ & $\displaystyle 2$ \\
\cline{2-3} 
& $\displaystyle ( J,\langle -1\rangle )$ & $\displaystyle 2$ \\
\cline{2-3} 
& $\displaystyle ( K,\langle -1\rangle )$ & $\displaystyle 2$ \\
\cline{2-3} 
& $\displaystyle ( \langle -1\rangle ,1)$ & $\displaystyle 1$ \\
\hline 
\multirow{6}{*}{$\displaystyle 8$} & $\displaystyle ( Q_{8} ,Q_{8})$ & $\displaystyle 1$ \\
\cline{2-3} 
& $\displaystyle ( I,I)$ & $\displaystyle 1$ \\
\cline{2-3} 
& $\displaystyle ( J,J)$ & $\displaystyle 1$ \\
\cline{2-3} 
& $\displaystyle ( K,K)$ & $\displaystyle 1$ \\
\cline{2-3} 
& $\displaystyle ( \langle -1\rangle ,\langle -1\rangle )$ & $\displaystyle 1$ \\
\cline{2-3} 
& $\displaystyle ( 1,1)$ & $\displaystyle 1$ \\
\hline
\end{tabular}
\caption{Condensable algebras in $Z(Vec_{Q_8})$}
    \label{Tab:Q8}
    
    \end{table}

    The group cohomology ring of $Q_8$ is given in \cite{tomoda2008remarks} Theorem 3.6:

    $$
    H^*(Q_8,\bZ)=\bZ[\alpha_2,\beta_2,\gamma_4]/(2\alpha_2,2\beta_2,8\gamma_4,\alpha_2\beta_2- 4\gamma_4,\alpha_2^2,\beta_2^2).
    $$
    Here, $|\alpha_2|=|\beta_2|=2,|\gamma_4|=4$. 

    For $(H,F)=(Q_8,\langle -1\rangle)$, we consider the exact sequence 
    $$
    1\to \bZ_2=\langle -1\rangle\stackrel{\iota}{\to} Q_8 \stackrel{\pi}{\to} Q_8/\langle -1\rangle=\bZ_2\times \bZ_2 \to 1
    $$
    and the spectral sequence $E_2^{m,n}=H^m(Q_8/\langle -1\rangle,H^n(\langle -1\rangle,\bZ))$. Similar to the previous discussion about $D_8$, all differentials $d_2$ in $E_2$ page are trivial. So $E_2=E_3$, $E_{\infty}^{2,0}=\bZ_2^2$. 

    We write $\mathbb{Z}_2^{(1)} \times \mathbb{Z}_2^{(2)}$ to distinguish the two components of $\mathbb{Z}_2 \times \mathbb{Z}_2$. Let $x, y$ be the generators of $H^2(\mathbb{Z}_2^{(1)} \times \mathbb{Z}_2^{(2)}, \mathbb{Z})$, corresponding to the images under $\pi^*$ of the generators of 
$
H^2(\mathbb{Z}_2^{(1)} \times \mathbb{Z}_2^{(2)}/\mathbb{Z}_2^{(2)}, \mathbb{Z})
$
and
$
H^2(\mathbb{Z}_2^{(1)} \times \mathbb{Z}_2^{(2)}/\mathbb{Z}_2^{(1)}, \mathbb{Z}),
$
respectively.

From the discussion above, we have the following:
$$
\pi^*: H^2(Q_8 / \langle -1 \rangle, \mathbb{Z}) \to H^2(Q_8, \mathbb{Z}), \quad x \mapsto \alpha_2, \quad y \mapsto \beta_2.
$$

Using the structure of the cohomology ring, we obtain:
$$
\pi^*: H^4(Q_8 / \langle -1 \rangle, \mathbb{Z}) \to H^4(Q_8, \mathbb{Z}), \quad x^2 \mapsto \alpha_2^2, \quad y^2 \mapsto \beta_2^2, \quad xy \mapsto \alpha_2 \beta_2 = 4\gamma_4.
$$

Therefore,
$$
\ker(\pi^*: H^4(Q_8 / \langle -1 \rangle, \mathbb{Z}) \to H^4(Q_8, \mathbb{Z})) = \mathbb{Z}_2 \times \mathbb{Z}_2 = \langle x^2, y^2 \rangle.
$$

We also have
$$
\coker(\pi^*: H^3(Q_8 / \langle -1 \rangle, \mathbb{Z}) \to H^3(Q_8, \mathbb{Z})=1) =1.
$$
Thus, for $(H,F)=(Q_8,\langle -1\rangle)$, there are $4$ condensable algebras.

For $(H,F)=(Q_8,J)$, we have 
$$
\pi^*: H^2(Q_8/J, \mathbb{Z}) \to H^2(Q_8, \mathbb{Z}), \quad x \mapsto \alpha_2,
$$
where $x$ is the generator of $H^2(Q_8/J, \mathbb{Z})=H^2(\bZ_2,\bZ)$. We use $x$ to denote the generator because we view $J/\langle -1 \rangle$ as the subgroup $\mathbb{Z}_2^{(2)}$ inside $Q_8 / \langle -1 \rangle \simeq \mathbb{Z}_2^{(1)} \times \mathbb{Z}_2^{(2)}$. Similarly, in the computation of embedding relations later on, we regard $I/\langle -1 \rangle$ as the subgroup $\mathbb{Z}_2^{(1)}$ of $Q_8 / \langle -1 \rangle$, and $K/\langle -1 \rangle$ as the diagonal subgroup $\mathbb{Z}_2^{(12)} = \{(0,0), (1,1)\} \subset \mathbb{Z}_2^{(1)} \times \mathbb{Z}_2^{(2)}$.

Using the structure of the cohomology ring, we obtain:
$$
\ker(\pi^*: H^4(Q_8 / J, \mathbb{Z}) \to H^4(Q_8, \mathbb{Z})) = \mathbb{Z}_2=\langle x^2 \rangle.
$$

We also have
$$
\coker(\pi^*: H^3(Q_8 /J, \mathbb{Z}) \to H^3(Q_8, \mathbb{Z})=1) =1.
$$
Thus for $(H,F)=(Q_8,J)$, there are $2$ condensable algebras. The classification of the remaining types of condensable algebras is either trivial or has already been discussed in previous subsections.

To obtain the embedding diagram of $Z(Vec_{Q_8})$, it is helpful to consider the condensable algebras in $Z(Vec_{\mathbb{Z}_2^{(1)} \times \mathbb{Z}_2^{(2)}}^{x^2})$. We list all condensable algebras in $Z(Vec_{\mathbb{Z}_2^{(1)} \times \mathbb{Z}_2^{(2)}}^{x^2})$ in the Table~\ref{tab:Z2Z2x2}. The diagram of embedding relations are shown in Figure~\ref{Fig:Z2Z2x2}. We use red boxes to indicate that the condensable algebra has $\alpha = x^2$ (we abuse notation and also denote the generator of $H^4(\mathbb{Z}_2^{(1)} \times \mathbb{Z}_2^{(2)}/\mathbb{Z}_2^{(2)}, \mathbb{Z})$ by $x^2$), and green boxes to indicate that $\alpha$ corresponds to the nontrivial element in $H^4(\mathbb{Z}_2, \mathbb{Z})$. All other condensable algebras, either uncolored or without specified $\alpha$ or $\phi$, have trivial $\alpha$ and $\phi$.

\begin{center}
    
\begin{table}[!h]
    \centering
    
\begin{tabular}{|p{0.33\textwidth}|p{0.33\textwidth}|p{0.33\textwidth}|}
\hline 
dimension & $\displaystyle ( H,F)$ & number \\
\hline 
$\displaystyle 1$ & $\displaystyle \left(\mathbb{Z}_{2}^{( 1)} \times \mathbb{Z}_{2}^{( 2)} ,1\right)$ & $\displaystyle 1$ \\
\hline 
\multirow{6}{*}{$\displaystyle 2$} & $\displaystyle \left(\mathbb{Z}_{2}^{( 1)} \times \mathbb{Z}_{2}^{( 2)} ,\mathbb{Z}_{2}^{( 1)}\right)$ & $\displaystyle 0$ \\
\cline{2-3} 
& $\displaystyle \left(\mathbb{Z}_{2}^{( 1)} \times \mathbb{Z}_{2}^{( 2)} ,\mathbb{Z}_{2}^{( 2)}\right)$ & $\displaystyle 2$ \\
\cline{2-3} 
& $\displaystyle \left(\mathbb{Z}_{2}^{( 1)} \times \mathbb{Z}_{2}^{( 2)} ,\mathbb{Z}_{2}^{( 12)}\right)$ & $\displaystyle 0$ \\
\cline{2-3} 
& $\displaystyle \left(\mathbb{Z}_{2}^{( 1)} ,1\right)$ & $\displaystyle 1$ \\
\cline{2-3} 
& $\displaystyle \left(\mathbb{Z}_{2}^{( 2)} ,1\right)$ & $\displaystyle 1$ \\
\cline{2-3} 
& $\displaystyle \left(\mathbb{Z}_{2}^{( 12)} ,1\right)$ & $\displaystyle 1$ \\
\hline 
\multirow{5}{*}{$\displaystyle 4$} & $\displaystyle \left(\mathbb{Z}_{2}^{( 1)} \times \mathbb{Z}_{2}^{( 2)} ,\mathbb{Z}_{2}^{( 1)} \times \mathbb{Z}_{2}^{( 2)}\right)$ & $\displaystyle 0$ \\
\cline{2-3} 
& $\displaystyle \left(\mathbb{Z}_{2}^{( 1)} ,\mathbb{Z}_{2}^{( 1)}\right)$ & $\displaystyle 0$ \\
\cline{2-3} 
& $\displaystyle \left(\mathbb{Z}_{2}^{( 2)} ,\mathbb{Z}_{2}^{( 2)}\right)$ & $\displaystyle 1$ \\
\cline{2-3} 
& $\displaystyle \left(\mathbb{Z}_{2}^{( 12)} ,\mathbb{Z}_{2}^{( 12)}\right)$ & $\displaystyle 0$ \\
\cline{2-3} 
& $\displaystyle ( 1,1)$ & $\displaystyle 1$ \\
\hline
\end{tabular}
    \caption{The condensable algebras in $Z(Vec_{\mathbb{Z}_2^{(1)} \times \mathbb{Z}_2^{(2)}}^{x^2})$}
    \label{tab:Z2Z2x2}
    \end{table}
\end{center}

\begin{figure}
    \centering
\tikzset{every picture/.style={line width=0.75pt}} 

\begin{tikzpicture}[x=0.75pt,y=0.75pt,yscale=-0.9,xscale=0.9]

\draw  [fill={rgb, 255:red, 208; green, 2; blue, 27 }  ,fill opacity=0.4 ] (300,50) -- (410,50) -- (410,90) -- (300,90) -- cycle ;
\draw   (430,220) -- (480,220) -- (480,260) -- (430,260) -- cycle ;
\draw  [fill={rgb, 255:red, 184; green, 233; blue, 134 }  ,fill opacity=0.8 ] (370,130) -- (450,130) -- (450,170) -- (370,170) -- cycle ;
\draw  [fill={rgb, 255:red, 208; green, 2; blue, 27 }  ,fill opacity=0.4 ] (20,130) -- (150,130) -- (150,170) -- (20,170) -- cycle ;
\draw   (220,220) -- (310,220) -- (310,260) -- (220,260) -- cycle ;
\draw    (260,170) -- (260,220) ;
\draw    (360,90) -- (260,130) ;
\draw    (360,90) -- (510,130) ;
\draw    (360,90) -- (80,130) ;
\draw    (360,90) -- (410,130) ;
\draw  [fill={rgb, 255:red, 208; green, 2; blue, 27 }  ,fill opacity=0.4 ] (176,130) -- (340,130) -- (340,170) -- (176,170) -- cycle ;
\draw   (470,130) -- (550,130) -- (550,170) -- (470,170) -- cycle ;
\draw  [fill={rgb, 255:red, 184; green, 233; blue, 134 }  ,fill opacity=0.8 ] (570,130) -- (650,130) -- (650,170) -- (570,170) -- cycle ;
\draw    (360,90) -- (610,130) ;
\draw    (80,170) -- (260,220) ;
\draw    (420,170) -- (460,220) ;
\draw    (520,170) -- (460,220) ;
\draw    (610,170) -- (460,220) ;

\draw (381,140) node [anchor=north west][inner sep=0.75pt]   [align=left] {$\displaystyle \left(\mathbb{Z}_{2}^{( 1)} ,1\right)$};
\draw (27,140) node [anchor=north west][inner sep=0.75pt]   [align=left] {$\displaystyle \left(\mathbb{Z}_{2}^{( 1)} \times \mathbb{Z}_{2}^{( 2)} ,\mathbb{Z}_{2}^{( 2)}\right)$};
\draw (231,230) node [anchor=north west][inner sep=0.75pt]   [align=left] {$\displaystyle \left(\mathbb{Z}_{2}^{( 2)} ,\mathbb{Z}_{2}^{( 2)}\right)$};
\draw (311,60) node [anchor=north west][inner sep=0.75pt]   [align=left] {$\displaystyle \left(\mathbb{Z}_{2}^{( 1)} \times \mathbb{Z}_{2}^{( 2)} ,1\right)$};
\draw (184,140) node [anchor=north west][inner sep=0.75pt]   [align=left] {$\displaystyle \left(\mathbb{Z}_{2}^{( 1)} \times \mathbb{Z}_{2}^{( 2)} ,\mathbb{Z}_{2}^{( 2)} ,x^{2} ,\phi \right)$};
\draw (481,140) node [anchor=north west][inner sep=0.75pt]   [align=left] {$\displaystyle \left(\mathbb{Z}_{2}^{( 2)} ,1\right)$};
\draw (580,140) node [anchor=north west][inner sep=0.75pt]   [align=left] {$\displaystyle \left(\mathbb{Z}_{2}^{( 12)} ,1\right)$};
\draw (437.31,232) node [anchor=north west][inner sep=0.75pt]   [align=left] {$\displaystyle ( 1,1)$};

\end{tikzpicture}
\caption{The embedding relations of condensable algebras in $Z(Vec_{\mathbb{Z}_2^{(1)} \times \mathbb{Z}_2^{(2)}}^{x^2})$}
\label{Fig:Z2Z2x2}
\end{figure}
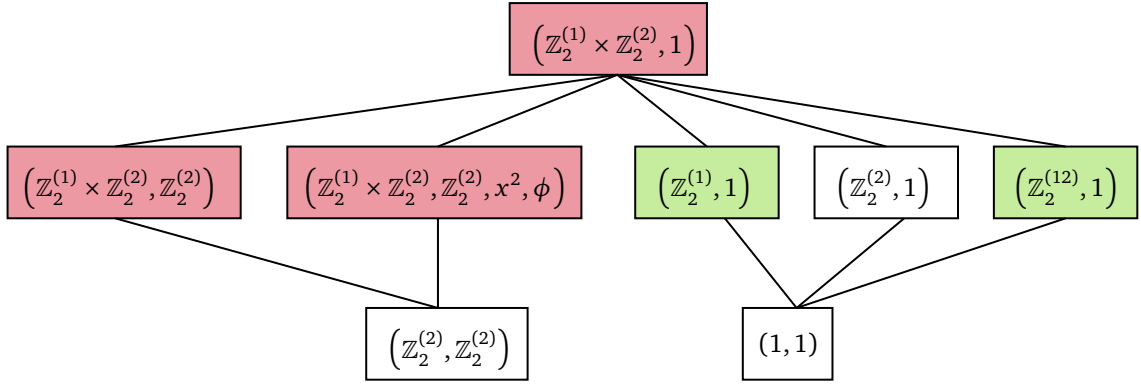

The diagram of embedding relations of $Z(Vec_{Q_8})$ is shown in Figure~\ref{Fig:Q_8}. We use red, blue, and yellow boxes to indicate that the condensable algebra has $\alpha = x^2$, $y^2$, and $x^2 + y^2$, respectively. Green boxes indicate that the corresponding $\alpha$ is the nontrivial element in $H^4(\mathbb{Z}_2, \mathbb{Z})$. All other condensable algebras, either uncolored or without specified $\alpha$ or $\phi$, have trivial $\alpha$ and $\phi$. A bold line represents that there are two distinct embeddings from the condensable algebra above the line to the one below.

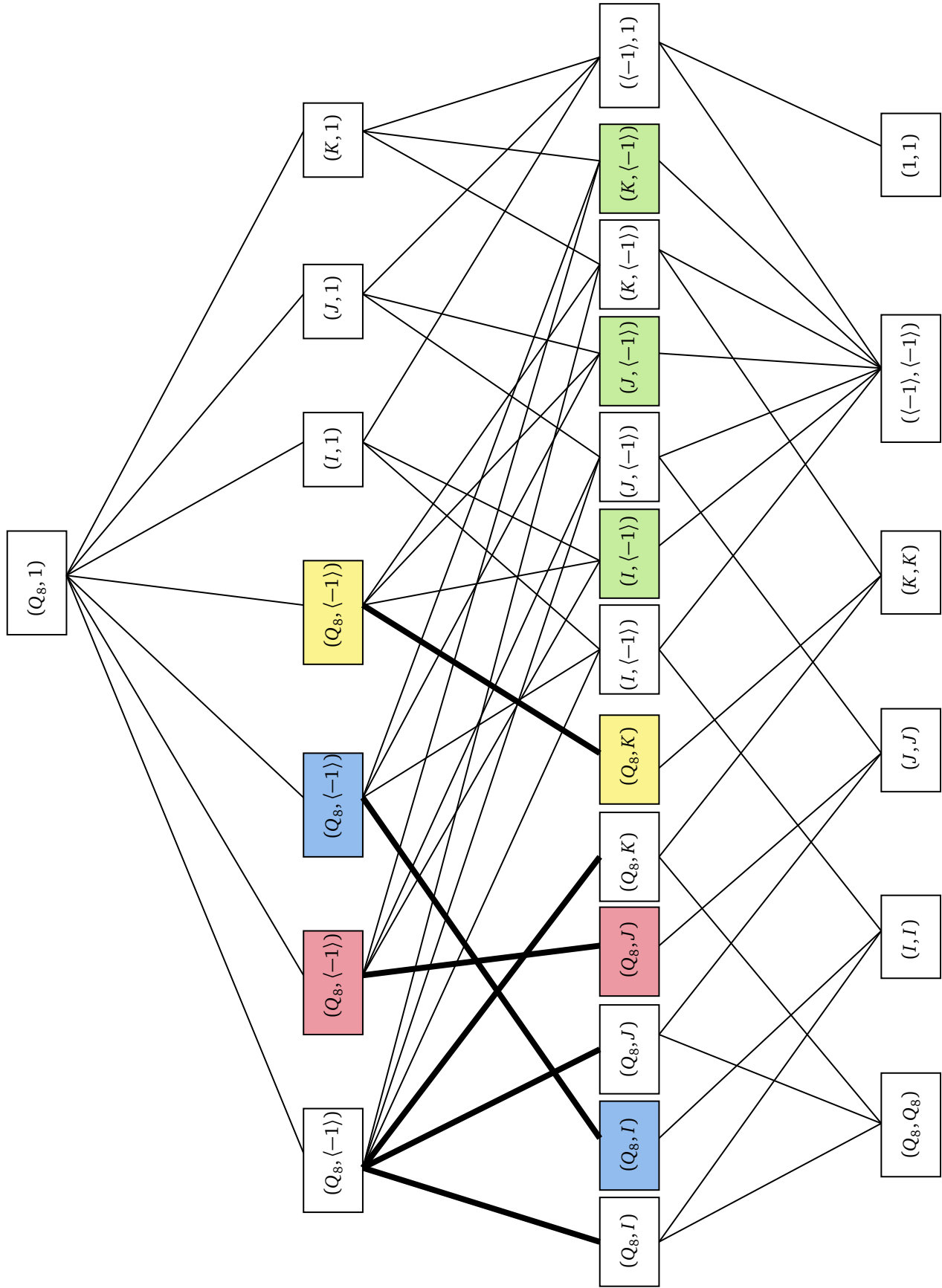
\begin{sidewaysfigure}
    \centering

    \tikzset{every picture/.style={line width=0.75pt}} 

    \begin{tikzpicture}[x=0.75pt,y=0.75pt,yscale=-1,xscale=1]
    
    \draw   (490,10) -- (560,10) -- (560,50) -- (490,50) -- cycle ;
    \draw   (100,210) -- (170,210) -- (170,250) -- (100,250) -- cycle ;
    \draw   (590,210) -- (640,210) -- (640,250) -- (590,250) -- cycle ;
    \draw   (690,210) -- (740,210) -- (740,250) -- (690,250) -- cycle ;
    \draw   (799,210) -- (849,210) -- (849,250) -- (799,250) -- cycle ;
    \draw   (50,410) -- (110,410) -- (110,450) -- (50,450) -- cycle ;
    \draw   (180,410) -- (240,410) -- (240,450) -- (180,450) -- cycle ;
    \draw   (310,410) -- (370,410) -- (370,450) -- (310,450) -- cycle ;
    \draw   (450,410) -- (510,410) -- (510,450) -- (450,450) -- cycle ;
    \draw   (124,600) -- (194,600) -- (194,640) -- (124,640) -- cycle ;
    \draw   (258,600) -- (314,600) -- (314,640) -- (258,640) -- cycle ;
    \draw   (384,600) -- (440,600) -- (440,640) -- (384,640) -- cycle ;
    \draw   (504,600) -- (560,600) -- (560,640) -- (504,640) -- cycle ;
    \draw   (620,600) -- (706,600) -- (706,640) -- (620,640) -- cycle ;
    \draw   (786,600) -- (842,600) -- (842,640) -- (786,640) -- cycle ;
    \draw  [fill={rgb, 255:red, 74; green, 144; blue, 226 }  ,fill opacity=0.6 ] (115,410) -- (175,410) -- (175,450) -- (115,450) -- cycle ;
    \draw  [fill={rgb, 255:red, 208; green, 2; blue, 27 }  ,fill opacity=0.4 ] (246,410) -- (306,410) -- (306,450) -- (246,450) -- cycle ;
    \draw  [fill={rgb, 255:red, 248; green, 231; blue, 28 }  ,fill opacity=0.5 ] (375.67,410) -- (435.67,410) -- (435.67,450) -- (375.67,450) -- cycle ;
    \draw  [fill={rgb, 255:red, 184; green, 233; blue, 134 }  ,fill opacity=0.8 ] (514.67,410) -- (574.67,410) -- (574.67,450) -- (514.67,450) -- cycle ;
    \draw   (580,410) -- (640,410) -- (640,450) -- (580,450) -- cycle ;
    \draw  [fill={rgb, 255:red, 184; green, 233; blue, 134 }  ,fill opacity=0.8 ] (644.67,410) -- (704.67,410) -- (704.67,450) -- (644.67,450) -- cycle ;
    \draw   (710,410) -- (770,410) -- (770,450) -- (710,450) -- cycle ;
    \draw  [fill={rgb, 255:red, 184; green, 233; blue, 134 }  ,fill opacity=0.8 ] (774.67,410) -- (834.67,410) -- (834.67,450) -- (774.67,450) -- cycle ;
    \draw   (846,410) -- (916,410) -- (916,450) -- (846,450) -- cycle ;
    \draw  [fill={rgb, 255:red, 208; green, 2; blue, 27 }  ,fill opacity=0.4 ] (220,210) -- (290,210) -- (290,250) -- (220,250) -- cycle ;
    \draw  [fill={rgb, 255:red, 74; green, 144; blue, 226 }  ,fill opacity=0.6 ] (340,210) -- (410,210) -- (410,250) -- (340,250) -- cycle ;
    \draw  [fill={rgb, 255:red, 248; green, 231; blue, 28 }  ,fill opacity=0.5 ] (470,210) -- (540,210) -- (540,250) -- (470,250) -- cycle ;
    \draw    (530,50) -- (140,210) ;
    \draw    (530,50) -- (260,210) ;
    \draw    (530,50) -- (380,210) ;
    \draw    (530,50) -- (510,210) ;
    \draw    (530,50) -- (620,210) ;
    \draw    (530,50) -- (720,210) ;
    \draw    (530,50) -- (830,210) ;
    \draw [line width=3]    (130,250) -- (80,410) ;
    \draw [line width=3]    (130,250) -- (210,410) ;
    \draw [line width=3]    (130,250) -- (340,410) ;
    \draw    (130,250) -- (480,410) ;
    \draw    (130,250) -- (610,410) ;
    \draw    (130,250) -- (740,410) ;
    \draw [line width=3]    (260,250) -- (280,410) ;
    \draw    (260,250) -- (540,410) ;
    \draw    (260,250) -- (610,410) ;
    \draw    (260,250) -- (810,410) ;
    \draw    (280,450) -- (410,600) ;
    \draw    (550,450) -- (670,600) ;
    \draw    (610,450) -- (670,600) ;
    \draw    (810,450) -- (670,600) ;
    \draw [line width=3]    (380,250) -- (150,410) ;
    \draw    (380,250) -- (480,410) ;
    \draw    (380,250) -- (680,410) ;
    \draw    (380,250) -- (810,410) ;
    \draw    (150,450) -- (290,600) ;
    \draw    (480,450) -- (670,600) ;
    \draw    (680,450) -- (670,600) ;
    \draw [line width=3]    (510,250) -- (410,410) ;
    \draw    (510,250) -- (740,410) ;
    \draw    (510,250) -- (540,410) ;
    \draw    (510,250) -- (680,410) ;
    \draw    (400,450) -- (530,600) ;
    \draw    (750,450) -- (670,600) ;
    \draw    (620,250) -- (480,410) ;
    \draw    (620,250) -- (540,410) ;
    \draw    (620,250) -- (880,410) ;
    \draw    (480,450) -- (290,600) ;
    \draw    (890,450) -- (820,600) ;
    \draw    (890,450) -- (670,600) ;
    \draw    (720,250) -- (610,410) ;
    \draw    (720,250) -- (680,410) ;
    \draw    (720,250) -- (880,410) ;
    \draw    (610,450) -- (410,600) ;
    \draw    (830,250) -- (740,410) ;
    \draw    (830,250) -- (810,410) ;
    \draw    (830,250) -- (880,410) ;
    \draw    (750,450) -- (530,600) ;
    \draw    (80,450) -- (160,600) ;
    \draw    (80,450) -- (290,600) ;
    \draw    (220,450) -- (160,600) ;
    \draw    (220,450) -- (410,600) ;
    \draw    (340,450) -- (160,600) ;
    \draw    (340,450) -- (530,600) ;
    
    \draw (501,22) node [anchor=north west][inner sep=0.75pt]   [align=left] {$\displaystyle ( Q_{8} ,1)$};
    \draw (110,222) node [anchor=north west][inner sep=0.75pt]   [align=left] {$\displaystyle ( Q_{8} ,\langle -1\rangle )$};
    \draw (600,222) node [anchor=north west][inner sep=0.75pt]   [align=left] {$\displaystyle ( I,1)$};
    \draw (701,222) node [anchor=north west][inner sep=0.75pt]   [align=left] {$\displaystyle ( J,1)$};
    \draw (808,222) node [anchor=north west][inner sep=0.75pt]   [align=left] {$\displaystyle ( K,1)$};
    \draw (61,422) node [anchor=north west][inner sep=0.75pt]   [align=left] {$\displaystyle ( Q_{8} ,I)$};
    \draw (190,422) node [anchor=north west][inner sep=0.75pt]   [align=left] {$\displaystyle ( Q_{8} ,J)$};
    \draw (315,422) node [anchor=north west][inner sep=0.75pt]   [align=left] {$\displaystyle ( Q_{8} ,K)$};
    \draw (451,422) node [anchor=north west][inner sep=0.75pt]   [align=left] {$\displaystyle ( I,\langle -1\rangle )$};
    \draw (135,612) node [anchor=north west][inner sep=0.75pt]   [align=left] {$\displaystyle ( Q_{8} ,Q_{8})$};
    \draw (269,612) node [anchor=north west][inner sep=0.75pt]   [align=left] {$\displaystyle ( I,I)$};
    \draw (397,612) node [anchor=north west][inner sep=0.75pt]   [align=left] {$\displaystyle ( J,J)$};
    \draw (515,612) node [anchor=north west][inner sep=0.75pt]   [align=left] {$\displaystyle ( K,K)$};
    \draw (631,612) node [anchor=north west][inner sep=0.75pt]   [align=left] {$\displaystyle ( \langle -1\rangle ,\langle -1\rangle )$};
    \draw (797,612) node [anchor=north west][inner sep=0.75pt]   [align=left] {$\displaystyle ( 1,1)$};
    \draw (126,422) node [anchor=north west][inner sep=0.75pt]   [align=left] {$\displaystyle ( Q_{8} ,I)$};
    \draw (257,422) node [anchor=north west][inner sep=0.75pt]   [align=left] {$\displaystyle ( Q_{8} ,J)$};
    \draw (386.67,422) node [anchor=north west][inner sep=0.75pt]   [align=left] {$\displaystyle ( Q_{8} ,K)$};
    \draw (521,422) node [anchor=north west][inner sep=0.75pt]   [align=left] {$\displaystyle ( I,\langle -1\rangle )$};
    \draw (581,422) node [anchor=north west][inner sep=0.75pt]   [align=left] {$\displaystyle ( J,\langle -1\rangle )$};
    \draw (651,422) node [anchor=north west][inner sep=0.75pt]   [align=left] {$\displaystyle ( J,\langle -1\rangle )$};
    \draw (714,422) node [anchor=north west][inner sep=0.75pt]   [align=left] {$\displaystyle ( K,\langle -1\rangle )$};
    \draw (781,422) node [anchor=north west][inner sep=0.75pt]   [align=left] {$\displaystyle ( K,\langle -1\rangle )$};
    \draw (857,422) node [anchor=north west][inner sep=0.75pt]   [align=left] {$\displaystyle ( \langle -1\rangle ,1)$};
    \draw (230,222) node [anchor=north west][inner sep=0.75pt]   [align=left] {$\displaystyle ( Q_{8} ,\langle -1\rangle )$};
    \draw (351,222) node [anchor=north west][inner sep=0.75pt]   [align=left] {$\displaystyle ( Q_{8} ,\langle -1\rangle )$};
    \draw (480,222) node [anchor=north west][inner sep=0.75pt]   [align=left] {$\displaystyle ( Q_{8} ,\langle -1\rangle )$};

    \end{tikzpicture}
    
        \caption{The embedding relations of condensable algebras in $Z(Vec_{Q_8})$}
        \label{Fig:Q_8}
    \end{sidewaysfigure}

    The splitting phenomenon can be observed by counting the number of condensable algebras. For instance, as shown in Figure~\ref{Fig:Q_8}, there are six condensable algebras of dimension four in $Z(Vec_{Q_8})$ that contain the algebra $A$ represented by $(Q_8, \langle -1 \rangle, 1, 1)$. On the other hand, $Z(Vec_{Q_8})_A^{loc} \simeq Z(Vec_{\bZ_2 \times \bZ_2})$ contains nine condensable algebras of dimension two. Therefore, some four-dimensional condensable algebras in $Z(Vec_{Q_8})$ must split when we condense $A$.

\appendix
\section{The canonical 2-cochain $\zeta_g$}\label{appendix:A} 

Let $G$ be a finite group and let $\omega \in H^3(G,U(1))$ be a 3-cocycle class. 
Setting $H' = gHg^{-1}$ for $g \in G$, we consider the algebra isomorphism 
$f_g: Fun(G/H) \to Fun(G/H')$ defined by $\delta_{tH} \mapsto \delta_{tg^{-1}H'}$. 
This isomorphism induces a braided monoidal functor 
$F_g: Z(Vec_G^\omega)_{Fun(G/H)}^{loc} \simeq Z(Vec_G^\omega)_{Fun(G/H')}^{loc}: (M,m_M) \mapsto (M, m_M \circ (\mathrm{id} \otimes f_g^{-1}) )$, 
where $m_M$ is the right $Fun(G/H)$-module structure on $M$.

Following \cite[Theorem 2.1]{davydov2017lagrangian}, we introduce the following notations:
$$
\omega(a,b | f) := \omega(a,b,f)^{-1} \omega(a,{}^bf,b) \omega({}^{ab}f,a,b)^{-1},
$$
$$
\omega(a | f,g) := \omega(a,f,g) \omega({}^{a}f,a,g)^{-1} \omega({}^{a}f,{}^{a}g,a),
$$
where ${}^{b}f = bfb^{-1}$. In this Appendix, for brevity, the symbols $\omega, \zeta, \ldots$ all denote chosen representatives of their respective cohomology or torsor classes. There are two braided monoidal functors:
$$
D: Z(Vec_G^{\omega})_{Fun(G/H)}^{loc} \simeq Z(Vec_H^{\omega|_H}), \quad M \mapsto M\delta_{H},
$$
$$
E: Z(Vec_H^{\omega|_H}) \simeq Z(Vec_G^{\omega})_{Fun(G/H)}^{loc}, \quad V \mapsto E(V),
$$
which satisfy $E \circ D \simeq \mathrm{id}$ and $D \circ E \simeq \mathrm{id}$. Here, the functor $E(V)$ is explicitly given by
$$
E(V) := \{f: G \to V | f(gh) = \omega(h^{-1},g^{-1} | s) h^{-1}.f(g), |f|=s, \, g \in G, \, h \in H\}.
$$
The notation $|f|$ denotes the degree of $f$, where $|f|=s$ if and only if $|f(x)| = x^{-1}sx$ for all $x \in G$. The $G$-action on $E(V)$ is defined by 
$$
(g.f)(x) = \omega(x^{-1},g | |f|)^{-1} f(g^{-1}x).
$$
Furthermore, we have $E(\mathbb{C}) = Fun(G/H)$, and the right $Fun(G/H)$-module structure on $E(V)$ is given by 
$$
m_{E(V)}(f \otimes \delta_{tH})(x) = \delta_{tH}(x) f(x).
$$

Consider the composition $\tilde{F}_g := D'\circ F_g\circ E : Z(Vec_H^{\omega|_H})\to Z(Vec_{H'}^{\omega|_{H'}})$ such that the following diagram commutes up to braided monoidal natural isomorphism:

\begin{equation}
    \begin{tikzcd}
Z(Vec_G^\omega)_{Fun(G/H)}^{loc} \arrow[d, "F_g"] \arrow[r, "D"] & Z(Vec_H^{\omega|_H}) \arrow[d, "\tilde{F}_g"] \\
Z(Vec_G^\omega)_{Fun(G/H')}^{loc} \arrow[r, "D'"]                & Z(Vec_{H'}^{\omega|_{H'}})                    
\end{tikzcd}
\end{equation}

For $V \in Z(Vec_H^{\omega|_H})$, we have $\tilde{F}_g(V) = E(V)\delta_{H'} = E(V)\delta_{gH}$. Our goal is to determine the monoidal structure on $\tilde{F}_g$. To this end, the monoidal structure of the functor $E$ is explicitly given as follows:
$$
E(V)\otimes_{E(\mathbb{C})} E(W)\stackrel{\mu}{\longrightarrow} E(V\otimes W): f_1\otimes f_2\mapsto \mu(f_1\otimes f_2).
$$
Here, $\mu(f_1\otimes f_2)(x)= \omega(x^{-1}|s_1,s_2)f_1(x)\otimes f_2(x)$ for $x\in G$, where $|f_1|=s_1$ and $|f_2|=s_2$. One can readily verify that $\mu(f_1\otimes f_2)\in E(V\otimes W)$ and that $\mu$ defines a morphism in $Z(Vec_G^\omega)_{Fun(G/H)}^{loc}$. Consequently, the monoidal structure of $\tilde{F}_g$ is also given as follows:
$$
E(V)\delta_{gH}\otimes_{E(\mathbb{C})} E(W)\delta_{gH}\stackrel{\mu}{\longrightarrow} E(V\otimes W)\delta_{gH}: f_1\otimes f_2\mapsto \mu(f_1\otimes f_2).
$$
Here, $f_1$ and $f_2$ are supported on $gH$, and so they are uniquely determined by their values at $g$. For a chosen basis $\{v_i \otimes w_j\}$ of $V \otimes W$, we obtain a basis for $E(V)\delta_{gH} \otimes_{E(\mathbb{C})} E(W)\delta_{gH}$ given by $\{f_{v_i} \otimes f_{w_j}\}$, which satisfy $f_{v_i}(g) = v_i$ and $f_{w_j}(g) = w_j$. It then follows that $\mu(f_{v_i} \otimes f_{w_j})(g) = \omega(g^{-1} | s_1, s_2) v_i \otimes w_j$. Thus, the monoidal structure of $\tilde{F}_g$ is explicitly given by $\zeta_g(x,y) := \omega(g^{-1} | gxg^{-1}, gyg^{-1})$ for $x,y \in H$. A direct calculation verifies that $d\zeta_g(x,y,z) = \omega(gxg^{-1}, gyg^{-1}, gzg^{-1}) / \omega(x,y,z)$.

\begin{remark}
    If we choose a different basis ${f_{v_i} \otimes f_{w_j}}$ for $E(V)\delta_{gH} \otimes_{E(\mathbb{C})} E(W)\delta_{gH}$ satisfying $f_{v_i}(gh) = v_i$ and $f_{w_j}(gh) = w_j$ for some $h \in H$, then the monoidal structure of $\tilde{F}_g$ is given by $\zeta_{gh}(x,y) = \omega(h^{-1}g^{-1} | ghxh^{-1}g^{-1}, ghyh^{-1}g^{-1})$. The braided monoidal functors obtained from different choices are braided monoidally naturally isomorphic. More precisely, they differ by a braided monoidal autoequivalence of $Z(Vec_H^{\omega|_H})$ which is a lift of inner tensor autoequivalence $F_{l_h,\zeta_h|_H}: Vec_H^{\omega|_H}\to Vec_H^{\omega|_H}$. By \cite{galindo2017tensor,natale2017equivalence}, this braided monoidal autoequivalence is braided monoidally naturally isomorphic to the identity functor.
\end{remark}

\bibliographystyle{alpha}
\bibliography{reference}  

@article{kong2014anyon,
  title={Anyon condensation and tensor categories},
  author={Kong, Liang},
  journal={Nuclear Physics B},
  volume={886},
  pages={436--482},
  year={2014},
  publisher={Elsevier}
}

@article{kitaev2006anyons,
  title={Anyons in an exactly solved model and beyond},
  author={Kitaev, Alexei},
  journal={Annals of Physics},
  volume={321},
  number={1},
  pages={2--111},
  year={2006},
  publisher={Elsevier}
}

@article{kong2018center,
  title={The center functor is fully faithful},
  author={Kong, Liang and Zheng, Hao},
  journal={Advances in Mathematics},
  volume={339},
  pages={749--779},
  year={2018},
  publisher={Elsevier}
}

@book{etingof2016tensor,
  title={Tensor categories},
  author={Etingof, Pavel and Gelaki, Shlomo and Nikshych, Dmitri and Ostrik, Victor},
  volume={205},
  year={2016},
  publisher={American Mathematical Soc.}
}

@article{etingof2004finite,
  title={Finite tensor categories},
  author={Etingof, Pavel  and Ostrik, Victor},
  journal={Mosc.Math.J},
  volume={4},
  number={3},
  pages={627--654,782--783},
  year={2004}
}

@article{kitaev2012models,
  title={Models for gapped boundaries and domain walls},
  author={Kitaev, Alexei and Kong, Liang},
  journal={Communications in Mathematical Physics},
  volume={313},
  number={2},
  pages={351--373},
  year={2012},
  publisher={Springer}
}

@article{kong2015boundary,
  title={Boundary-bulk relation for topological orders as the functor mapping higher categories to their centers},
  author={Kong, Liang and Wen, Xiao-Gang and Zheng, Hao},
  journal={arXiv preprint arXiv:1502.01690},
  year={2015}
}

@article{kong2018gapless,
  title={Gapless edges of 2d topological orders and enriched monoidal categories},
  author={Kong, Liang and Zheng, Hao},
  journal={Nuclear Physics B},
  volume={927},
  pages={140--165},
  year={2018},
  publisher={Elsevier}
}

@article{kong2017semisimple,
  title={Semisimple and separable algebras in multi-fusion categories},
  author={Kong, Liang and Zheng, Hao},
  journal={arXiv preprint arXiv:1706.06904},
  year={2017}
}

@article{chen2020topological,
  title={Topological phase transition on the edge of two-dimensional $\mathbb{Z}_2$ topological order},
  author={Chen, Wei-Qiang and Jian, Chao-Ming and Kong, Liang and You, Yi-Zhuang and Zheng, Hao},
  journal={Physical Review B},
  volume={102},
  number={4},
  pages={045139},
  year={2020},
  publisher={APS}
}

@article{davydov2013structure,
  title={On the structure of the Witt group of braided fusion categories},
  author={Davydov, Alexei and Nikshych, Dmitri and Ostrik, Victor},
  journal={Selecta Mathematica},
  volume={19},
  number={1},
  pages={237--269},
  year={2013},
  publisher={Springer}
}

@article{ostrik2026non,
  title={A non-semisimple Witt class},
  author={Ostrik, Victor and Utiralova, Alexandra},
  journal={arXiv preprint arXiv:2602.10519},
  year={2026}
}

@article{coulembier2025simple,
  title={Simple algebras and exact module categories},
  author={Coulembier, Kevin and Stroi{\'n}ski, Mateusz and Zorman, Tony},
  journal={arXiv preprint arXiv:2501.06629},
  year={2025}
}

@article{davydov2013witt,
  title={The Witt group of non-degenerate braided fusion categories},
  author={Davydov, Alexei and M{\"u}ger, Michael and Nikshych, Dmitri and Ostrik, Victor},
  journal={Journal f{\"u}r die reine und angewandte Mathematik (Crelles Journal)},
  volume={2013},
  number={677},
  pages={135--177},
  year={2013},
  publisher={De Gruyter}
}

@article{lu2023boundary,
  title={The boundary phase transitions of the 2+ 1D ℤN topological order via topological Wick rotation},
  author={Lu, Yalei and Yang, Holiverse},
  journal={Journal of High Energy Physics},
  volume={2023},
  number={6},
  pages={1--49},
  year={2023},
  publisher={Springer}
}

@article{Chatterjee_2023,
   title={Holographic theory for continuous phase transitions: Emergence and symmetry protection of gaplessness},
   volume={108},
   ISSN={2469-9969},
   url={http://dx.doi.org/10.1103/PhysRevB.108.075105},
   DOI={10.1103/physrevb.108.075105},
   number={7},
   journal={Physical Review B},
   publisher={American Physical Society (APS)},
   author={Chatterjee, Arkya and Wen, Xiao-Gang},
   year={2023},
   month={aug}
}

@article{frohlich2006correspondences,
  title={Correspondences of ribbon categories},
  author={Fr{\"o}hlich, J{\"u}rg and Fuchs, J{\"u}rgen and Runkel, Ingo and Schweigert, Christoph},
  journal={Advances in Mathematics},
  volume={199},
  number={1},
  pages={192--329},
  year={2006},
  publisher={Elsevier}
}

@article{davydov2010modular,
  title={Modular invariants for group-theoretical modular data. I},
  author={Davydov, Alexei},
  journal={Journal of Algebra},
  volume={323},
  number={5},
  pages={1321--1348},
  year={2010},
  publisher={Elsevier}
}

@article{davydov2017lagrangian,
  title={On Lagrangian algebras in group-theoretical braided fusion categories},
  author={Davydov, Alexei and Simmons, Darren},
  journal={Journal of Algebra},
  volume={471},
  pages={149--175},
  year={2017},
  publisher={Elsevier}
}

@article{hannah2023frobenius,
  title={Frobenius monoidal functors of Dijkgraaf-Witten categories and rigid Frobenius algebras},
  author={Hannah, Samuel and Laugwitz, Robert and Ros Camacho, Ana and others},
  journal={SIGMA. Symmetry, Integrability and Geometry: Methods and Applications},
  volume={19},
  pages={075},
  year={2023},
  publisher={SIGMA. Symmetry, Integrability and Geometry: Methods and Applications}
}

@article{bruguieres2011exact,
  title={Exact sequences of tensor categories},
  author={Bruguieres, Alain and Natale, Sonia},
  journal={International Mathematics Research Notices},
  volume={2011},
  number={24},
  pages={5644--5705},
  year={2011},
  publisher={OUP}
}

@article{bruguieres2011hopf,
  title={Hopf monads on monoidal categories},
  author={Bruguieres, Alain and Lack, Steve and Virelizier, Alexis},
  journal={Advances in Mathematics},
  volume={227},
  number={2},
  pages={745--800},
  year={2011},
  publisher={Elsevier}
}

@article{schauenburg2001monoidal,
  title={The monoidal center construction and bimodules},
  author={Schauenburg, Peter},
  journal={Journal of Pure and Applied Algebra},
  volume={158},
  number={2-3},
  pages={325--346},
  year={2001},
  publisher={Elsevier}
}

@book{aguiar2010monoidal,
  title={Monoidal functors, species and Hopf algebras},
  author={Aguiar, Marcelo and Mahajan, Swapneel Arvind},
  volume={29},
  year={2010},
  publisher={Citeseer}
}

@article{flake2024projection,
  title={Projection formulas and induced functors on centers of monoidal categories},
  author={Flake, Johannes and Laugwitz, Robert and Posur, Sebastian},
  journal={arXiv preprint arXiv:2402.10094},
  year={2024}
}

@article{chen2013symmetry,
  title={Symmetry protected topological orders and the group cohomology of their symmetry group},
  author={Chen, Xie and Gu, Zheng-Cheng and Liu, Zheng-Xin and Wen, Xiao-Gang},
  journal={Physical Review B—Condensed Matter and Materials Physics},
  volume={87},
  number={15},
  pages={155114},
  year={2013},
  publisher={APS}
}

@article{bhardwaj2403hasse,
  title={Hasse diagrams for gapless SPT and SSB phases with non-invertible symmetries,(arXiv preprint) doi: 10.48550},
  author={Bhardwaj, L and Pajer, D and Sch{\"a}fer-Nameki, S and Warman, A},
  journal={arXiv preprint arXiv.2403.00905}
}

@book{hatcher2002algebraic,
  title={Algebraic Topology},
  author={Hatcher, Allen},
  year={2002},
  publisher={Cambridge University Press}
}

@article{handel1993products,
  title={On products in the cohomology of the dihedral groups},
  author={Handel, David},
  journal={Tohoku Mathematical Journal, Second Series},
  volume={45},
  number={1},
  pages={13--42},
  year={1993},
  publisher={Mathematical Institute, Tohoku University}
}

@article{tomoda2008remarks,
  title={Remarks on the cohomology of finite fundamental groups of 3-manifolds},
  author={Tomoda, Satoshi and Zvengrowski, Peter},
  journal={Geom. Topol. Monogr},
  volume={14},
  pages={519--556},
  year={2008}
}

@article{ji2020categorical,
  title={Categorical symmetry and noninvertible anomaly in symmetry-breaking and topological phase transitions},
  author={Ji, Wenjie and Wen, Xiao-Gang},
  journal={Physical Review Research},
  volume={2},
  number={3},
  pages={033417},
  year={2020},
  publisher={APS}
}

@article{kong2020mathematical,
  title={A mathematical theory of gapless edges of 2d topological orders. Part I},
  author={Kong, Liang and Zheng, Hao},
  journal={Journal of High Energy Physics},
  volume={2020},
  number={2},
  pages={1--62},
  year={2020},
  publisher={Springer}
}

@article{kong2021mathematical,
  title={A mathematical theory of gapless edges of 2d topological orders. Part II},
  author={Kong, Liang and Zheng, Hao},
  journal={Nuclear Physics B},
  volume={966},
  pages={115384},
  year={2021},
  publisher={Elsevier}
}

@article{fuchs2002tft,
  title={TFT construction of RCFT correlators I: Partition functions},
  author={Fuchs, J{\"u}rgen and Runkel, Ingo and Schweigert, Christoph},
  journal={Nuclear Physics B},
  volume={646},
  number={3},
  pages={353--497},
  year={2002},
  publisher={Elsevier}
}

@article{davydov2010centre,
  title={Centre of an algebra},
  author={Davydov, Alexei},
  journal={Advances in Mathematics},
  volume={225},
  number={1},
  pages={319--348},
  year={2010},
  publisher={Elsevier}
}

@article{davydov2015functoriality,
  title={Functoriality of the center of an algebra},
  author={Davydov, Alexei and Kong, Liang and Runkel, Ingo},
  journal={Advances in Mathematics},
  volume={285},
  pages={811--876},
  year={2015},
  publisher={Elsevier}
}

@article{ostrik2003module,
  title={Module categories, weak Hopf algebras and modular invariants},
  author={Ostrik, Victor},
  journal={Transformation groups},
  volume={8},
  number={2},
  pages={177--206},
  year={2003},
  publisher={Springer}
}

@article{kong2020algebraic,
  title={Algebraic higher symmetry and categorical symmetry: A holographic and entanglement view of symmetry},
  author={Kong, Liang and Lan, Tian and Wen, Xiao-Gang and Zhang, Zhi-Hao and Zheng, Hao},
  journal={Physical Review Research},
  volume={2},
  number={4},
  pages={043086},
  year={2020},
  publisher={APS}
}

@article{galindo2017tensor,
  title={Tensor functors between Morita duals of fusion categories},
  author={Galindo, C{\'e}sar and Plavnik, Julia Yael},
  journal={Letters in Mathematical Physics},
  volume={107},
  number={3},
  pages={553--590},
  year={2017},
  publisher={Springer}
}

@article{natale2017equivalence,
  title={On the equivalence of module categories over a group-theoretical fusion category},
  author={Natale, Sonia and others},
  journal={SIGMA. Symmetry, Integrability and Geometry: Methods and Applications},
  volume={13},
  pages={042},
  year={2017},
  publisher={SIGMA. Symmetry, Integrability and Geometry: Methods and Applications}
}

\end{document}